\documentclass{article}
\usepackage{graphicx}
\usepackage{amsmath}
\usepackage{amssymb}
\usepackage{amsthm}
\usepackage{amscd}
\usepackage{mathrsfs}
\usepackage{fancyhdr}
\usepackage{float}

\newtheorem{dfn}{Definition}[section]
\newtheorem{thm}[dfn]{Theorem}
\newtheorem{prop}[dfn]{Proposition}
\newtheorem{lem}[dfn]{Lemma}
\newtheorem{cor}[dfn]{Corollary}
\newtheorem{rem}[dfn]{Remark}
\newtheorem{car}[dfn]{Computer Assisted Result}
\newtheorem{ass}[dfn]{Assumption}

\newtheorem{alg}{Algorithm}

\newtheorem{notation}[dfn]{Notation}

\newcommand{\Inv}{{\rm Inv}}
\newcommand{\Int}{{\rm int}}

\newcommand{\re}{{\rm Re}}

\newcommand{\diag}{{\rm diag}}
\newcommand{\sgn}{{\rm sgn}}

\newcommand{\im}{{\rm Im}}
\newcommand{\exit}{{\rm exit}}
\newcommand{\ent}{{\rm ent}}

\numberwithin{equation}{section}

\begin{document}

\title{Rigorous numerics of tubular, conic, star-shaped neighborhoods of slow manifolds for fast-slow systems}

\author{Kaname Matsue\thanks{Institute of Mathematics for Industry, Kyushu University, Fukuoka 819-0395, Japan {\tt kmatsue@imi.kyushu-u.ac.jp}} $^{,}$ \footnote{International Institute for Carbon-Neutral Energy Research (WPI-I$^2$CNER), Kyushu University, Fukuoka 819-0395, Japan}
}
\lhead{Kaname Matsue}
\rhead{Rigorous Numerics of Tubular Neighborhoods of Slow Manifolds}
\maketitle

\begin{abstract}
We provide a rigorous numerical computation method to validate tubular neighborhoods of normally hyperbolic slow manifolds with the explicit radii for the fast-slow system
\begin{equation*}
\begin{cases}
x' = f(x,y,\epsilon), & \\
y' =\epsilon g(x,y,\epsilon). &
\end{cases}
\end{equation*}
Our main focus is the validation of the continuous family of eigenpairs $\{\lambda_i(y;\epsilon), u_i(y;\epsilon)\}_{i=1}^n$ of $f_x(h_\epsilon(y),y,\epsilon)$ over the slow manifold $S_\epsilon = \{x = h_\epsilon(y)\}$ admitting the graph representation.
In order to obtain such a family, we apply the interval Newton-like method with rigorous numerics.
The validated family of eigenvectors generates a vector bundle over $S_\epsilon$ determining normally hyperbolic eigendirections rigorously.
The generated vector bundle enables us to construct a tubular neighborhood centered at slow manifolds with explicit radii.
Combining rate conditions for providing smoothness of center-(un)stable manifolds, we can validate smooth tubular neighborhoods with diffeomorphic family of affine change of coordinates, as well as several extensions such as conic and star-shaped neighborhoods.
Our procedure provides a systematic construction of smooth neighborhoods of slow manifolds in an explicit range $[0,\epsilon_0]$ of $\epsilon$ with rigorous numerics.
\end{abstract}



\section{Introduction}
In this paper, we consider the dynamical system in $\mathbb{R}^n \times \mathbb{R}^l$ of the following form:
\begin{equation}
\label{fast-slow}
\begin{cases}
x' = f(x,y,\epsilon), & \\
y' =\epsilon g(x,y,\epsilon), &
\end{cases}
\end{equation}
where $' = d/dt$ is the time derivative and $f,g$ are $C^r$-functions with $r\geq 1$. 
The factor $\epsilon$ is a nonnegative but sufficiently small real number. 
We shall write (\ref{fast-slow}) as (\ref{fast-slow})$_\epsilon$ if we explicitly represent the $\epsilon$-dependence of the system.
The system (\ref{fast-slow}) can be reformulated with a change of time-scale variable as 
\begin{equation}
\label{slow-fast}
\begin{cases}
\epsilon \dot x = f(x,y,\epsilon), & \\
\dot y = g(x,y,\epsilon), &
\end{cases}
\end{equation}
where $\dot{} = d/d\tau$ and $\tau = t/\epsilon$. 
One tries to analyze the dynamics of (\ref{fast-slow}), 
equivalently (\ref{slow-fast}), by suitably combining the dynamics of the {\em layer problem}
\begin{equation}
\label{layer}
\begin{cases}
x' = f(x,y,0), & \\
y' =0, &
\end{cases}
\end{equation}
and the dynamics of the {\em reduced problem}
\begin{equation}
\label{reduced}
\begin{cases}
0 = f(x,y,0), & \\
\dot y = g(x,y,0), &
\end{cases}
\end{equation}
which are the limiting problems for $\epsilon = 0$ on the fast and the slow time scale, respectively. 
Notice that (\ref{reduced}) makes sense {\em only on $f(x,y,0)=0$}, while (\ref{layer}) makes sense in whole $\mathbb{R}^{n+l}$ as the $y$-parameter family of $x$-systems. 
The meaning of the \lq\lq $\epsilon\to 0$-limit" is thus different between (\ref{fast-slow}) and (\ref{slow-fast}). 
This is why (\ref{fast-slow}) or (\ref{slow-fast}) is a kind of {\em singular perturbation problems}. 
In particular, (\ref{fast-slow}) or (\ref{slow-fast}) is known as {\em fast-slow systems} (or {\em slow-fast systems}), where $x$ dominates the behavior in the fast time scale and $y$ dominates the behavior in the slow time scale.
\subsection{Our aim in this paper}
There are mainly two approaches for understanding dynamics of (\ref{fast-slow})$_\epsilon$: an analytic approach (asymptotic expansion of solutions) and a geometric one (geometric singular perturbation theory).
We shall focus on the geometric one here.
The key concept in the geometric singular perturbation theory is a {\em slow manifold}, a perturbation of critical manifolds which are subsets of nullcline $\{(x,y)\mid f(x,y,0)=0\}$.
Fenichel has proved in \cite{F1979} that, under the normal hyperbolicity, critical manifolds perturb to slow manifolds for sufficiently small $\epsilon > 0$.
A series of his results, which is often called {\em invariant manifold theorems} for fast-slow systems, is nowadays the basis of geometric singular perturbation theory (e,g, \cite{GS, GJM2012, GK, Jones1984, Jones, JKK, JK, KS, L, Mat2, Smo, S, SW2001, TKJ}). 
\par
Concrete studies of dynamics around slow manifolds are often operated under the assumption that slow manifolds are given by the graphs of smooth functions, e.g., $S_\epsilon = \{(x,y,\epsilon)\mid x = h(y,\epsilon)\}$.
Moreover, for simplicity, the slow manifold $S_\epsilon$ is assumed to lie in the subspace $\{x=0\}$ via a nonlinear smooth transformation (e.g. \cite{Jones, JK}), which consequently yields the coordinate system around $S_\epsilon$ so that
\begin{equation}
\label{slow-mfd-straight}
S_\epsilon = \{(0,y,\epsilon)\},\quad \epsilon\in [0,\epsilon_0]
\end{equation}
for some $\epsilon_0$. 
Furthermore, the coordinate system can be chosen so that stable and unstable fibers with base points on $S_\epsilon$ are linear invariant subspaces.
A series of transformation yields the resulting vector field which is often called {\em Fenichel normal form} and is the center of considerations for advanced analysis in fast-slow systems such as the Exchange Lemma (e.g. \cite{JKK, JK, L, TKJ}).
\par
Our focus in this paper relates to such coordinate systems and neighborhoods of slow manifolds in these coordinates from the viewpoint of numerical validations.
The above change of coordinates is realized in the abstract setting in general.
If we apply the above ideas to concrete systems, we have to obtain the (nonlinear) change of coordinates rigorously, which is a nontrivial problem and deeply depends on systems.
If we do not have an explicit way to compute such change of coordinates, which will be almost cases, it is natural to apply numerical calculations to computing slow manifolds first.
In this case, we can {\em never} obtain rigorous change of coordinates to the desiring one, which is due to various numerical errors (roundings, truncations and so on).
On the other hand, several works to validate slow manifolds as well as global trajectories for (\ref{fast-slow})$_\epsilon$ in an explicit range $[0,\epsilon_0]$ of $\epsilon$ with {\em rigorous numerics} based on interval arithmetic (e.g., \cite{Tbook}) have appeared very recently (e.g. \cite{CZ, GJM2012, Mat2}).
All of approaches therein produce neighborhoods of slow manifolds corresponding to tubular neighborhoods in appropriate senses for their aims. 
Nevertheless, they include more or less restrictions on applicability, such as dimensions of phase spaces, choice of candidates, geometry of neighborhoods and so on (see Section \ref{section-preceding} for details).
\par
\bigskip
In this paper, we aim at providing a procedure of tubular neighborhoods of slow manifolds {\em in a non-empirical way} with computer assistance.
An essence of our procedure is the rigorous continuous families of eigenpairs of the linearized matrix $f_x(h_\epsilon(y),y,\epsilon)$ on slow manifold $S_\epsilon = \{x=h_\epsilon(y)\mid y\in Y\}$ for some compact set $Y\subset \mathbb{R}^l$, which is a standard issue of numerical linear algebra, to obtain the coordinate system $(a,b,y)$ as follows:
\begin{equation*}
S_\epsilon = \{x=h_\epsilon(y) = 0\},\quad (\ref{fast-slow})_\epsilon \Leftrightarrow \begin{cases}
a' = A(y)a + F_1(a,b,y,\epsilon), &\\
b' = B(y)b + F_2(a,b,y,\epsilon), &\\
y' = \epsilon g(a,b,y,\epsilon) &
\end{cases}\text{ around }S_\epsilon,
\end{equation*}
where $A$ and $B$ are diagonal matrices such that $0 < \lambda_A < \re \lambda$ holds for all $\lambda \in {\rm Spec}(A(y))$ and $y\in Y$, and that $0 > \lambda_B > \re \lambda$ holds for all $\lambda \in {\rm Spec}(B(y))$ and $y\in Y$.
We combine an enclosing procedure of eigenpairs with validations of slow manifolds to obtain vector bundles over slow manifolds, which is a byproduct for constructing tubular neighborhoods.
A standard approach for constructing {\em isolating blocks} \cite{ZM} with the validated families of eigenpairs gives a smooth family of isolating blocks, namely, a tubular neighborhood of slow manifolds with explicit radii.

\par
\bigskip
The rest of this paper is organized as follows.
In Section \ref{section-preliminary}, we gather preliminaries of isolating blocks, and invariant manifold validations with computer assistance as well as related topics for general systems provided by Capi\'{n}ski and Zgliczy\'{n}ski \cite{CZ2015, CZ2016}.
In Section \ref{section-inv-mfd}, we apply these preliminaries to (\ref{fast-slow}) for validating smooth slow manifolds.
Discussions there contain existence arguments in \cite{Mat2}.
In Section \ref{section-eigenpairs}, we provide a validation procedure of a continuous family of eigenpairs with fixed norms of eigenfunctions for continuous real matrix-valued functions. 
This procedure works not only for real eigenvalues but also complex ones.
In Section \ref{section-main}, we provide algorithms for validating slow manifolds, associated vector bundles and tubular neighborhoods of validated slow manifolds with explicit radii, which can be validated for general fast-slow systems with rigorous numerics.
We also provide a procedure of extended neighborhoods of slow manifolds called conic and star-shaped neighborhoods centered at slow manifolds.
Sample validation results are shown in Section \ref{section-examples} for demonstrating the applicability.

\subsection{Preceding approaches for validating enclosures of slow manifolds}
\label{section-preceding}
Before moving to concrete discussions, we briefly compare several preceding works for validating slow manifolds with rigorous numerics.
The essential issues of unsolved problems in corresponding preceding works are written in bold letters.

\subsubsection{Gameiro-Gedeon-Kalies-Kokubu-Mishcaikow-Oka \cite{GGKKMO}}
In \cite{GGKKMO}, a rigorous numerical procedure of {\em singular isolating neighborhoods}, a terminology of isolations in the (singularly perturbed) Conley index theory (e.g. \cite{Mis, Smo}), is discussed.
This is a purely topological approach.
Authors provide a systematic way to construct singular isolating neighborhoods with the help of polygonal approximation of flows in \cite{BKM2007}, which gives us a polygonal decomposition of critical manifolds so that flows intersect all boundaries of polygons transversely.
\par
An essential question in our direction remains open there whether we can validate isolations of slow manifolds as well as slow flows {\bf with an explicit range $[0,\epsilon_0]$ of $\epsilon$}.
Moreover, constructions of isolating neighborhoods based on multi-value map validations may contain extra regions enclosing true trajectories, which cause the wrong accuracy of targeting objects.

\subsubsection{Guckenheimer-Johnson-Meerkamp \cite{GJM2012}}
Authors of \cite{GJM2012} discuss validations of enclosures of slow manifolds. 
In that paper authors concentrate on fast-slow systems {\em with one fast variable and two slow variables}, which aims at validations of singular Hopf bifurcations.
The basis of their procedure is the triangulation of critical manifolds and computations of left and right correction (perturbation) terms of slow manifolds which enclose rigorous slow manifolds with an explicit range $[0,\epsilon_0]$. 
Extension of this method {\bf in more general systems} remains open.

\subsubsection{Czechowski-Zgliczy\'{n}ski \cite{CZ}}
In \cite{CZ}, rigorous numerical validations of periodic orbits for the FitzHugh-Nagumo system, which is well-known as an example of (\ref{fast-slow}), with an explicit range $[0,\epsilon_0]$ of $\epsilon$ is discussed.
Authors validate periodic orbits of the FitzHugh-Nagumo system with specific parameter values by a topological notion called {\em covering relations} (e.g. \cite{ZCov}) with appropriate estimates of vector fields with computer assistance.
Validation of slow manifolds in their context is a construction of {\em isolating segments}, rectangular domain containing compact potion of slow manifolds such that flows intersect boundary transversely in the fast direction.
A remarkable point of this work is a realization of the bridge between singularly perturbed trajectories and ones with a standard approach such as Newton-like method via the $\epsilon$-continuation.
On the other hand, the choice of isolating segments may contain more or less artificial trial and error.
It remains open whether we can choose appropriate isolating segments corresponding to tubular neighborhoods {\bf in a non-empirical way}.

\subsubsection{Matsue \cite{Mat2}}
In \cite{Mat2}, rigorous numerical validations of global trajectories for (\ref{fast-slow}) such as periodic, homoclinic and heteroclinic orbits with an explicit range $[0,\epsilon_0]$ of $\epsilon$ is discussed.
Ideas for validating slow manifolds are based on Jones' discussion in \cite{Jones} as well as a systematic procedure of isolating blocks by \cite{ZM}.
The author also validate cone conditions based on \cite{Jones} and \cite{ZCov}, which guarantees normal hyperbolicity of slow manifolds as well as invariant foliations of stable and unstable manifolds.
This approach takes account of the essence of geometric singular perturbation theory.
Unlike \cite{CZ}, however, there is a restriction of the parameter range $[0,\epsilon_0]$ validating trajectories, which is mainly because small pieces of slow manifolds are attached globally and systematically, but not smoothly via fast-saddle-type blocks; a counterpart of tubular neighborhoods in local setting.
Comparing \cite{CZ}, {\bf the realization of smooth blocks, or smooth and global attachments of small blocks} is of importance for larger $\epsilon$-continuation of trajectories and extension to multi-dimensional slow variables.
Finally note that the above unsolved tasks in the paper \cite{Mat2} motivate the current issue.

\section{Preliminaries}
\label{section-preliminary}
In this section, we gather fundamental tools we need in this paper, which consists of quick reviews of isolating blocks, invariant manifolds and their validations, and related topics.

\subsection{Validations of isolating blocks : review}
\label{section-block}

A concept of {\em isolating blocks} are typically discussed in the Conley index theory (e.g. \cite{Con, Mis}), which studies structures of isolated invariant sets from the algebraic-topological viewpoint.
Central notions are {\em isolating neighborhoods} or {\em index pairs} in the Conley index theory, but we concentrate our attentions on {\em isolating blocks} defined as follows. 
In our case, the blocks can be considered very flexible from the viewpoint of rigorous numerics.
Moreover, isolating blocks play central roles for the existence of slow manifolds, which is discussed in Section \ref{section-inv-mfd}.
Here we review the definition of isolating blocks and its applications to fast-slow systems with computer assistance. 
Detailed discussions in our setting are shown in \cite{Mat2}.

%
%
\begin{dfn}[Isolating block]\rm
\label{dfn-isolation}
Let $N\subset \mathbb{R}^m$ be a compact set. We say $N$ an {\em isolating neighborhood} if $\Inv(N)\subset \Int (N)$ holds, where 
\begin{equation*}
\Inv(N):= \{x\in N \mid \varphi(\mathbb{R},x)\subset N\}
\end{equation*}
for a flow $\varphi: \mathbb{R}\times \mathbb{R}^m\to \mathbb{R}^m$ on $\mathbb{R}^m$.
Next let $B\subset \mathbb{R}^m$ be a compact set and $x\in \partial B$. We say $x$ an {\em exit} ({\em resp. entrance}) point of $B$, if for every solution $\sigma:[-\delta_1,\delta_2]\to \mathbb{R}^m$ through $x= \sigma(0)$, with $\delta_1\geq 0$ and $\delta_2 > 0$ there are $0\leq \epsilon_1 \leq \delta_1$ and $0 < \epsilon_2 \leq \delta_2$ such that for $0 < t \leq \epsilon_2$,
\begin{equation*}
\sigma(t)\not \in B\ (\text{resp. } \sigma(t)\in \Int(B)),
\end{equation*}
and for $-\epsilon_1 \leq t < 0$,
\begin{equation*}
\sigma(t)\not \in \partial B\ (\text{resp. } \sigma(t)\not \in B)
\end{equation*}
hold. $B^{\exit}$ (resp. $B^{\ent}$) denote the set of all exit (resp. entrance) points of the closed set $B$. We call $B^{\exit}$ and $B^{\ent}$ {\em the exit} and {\em the entrance} of $B$, respectively.
Finally $B$ is called {\em an isolating block} if $\partial B = B^{\exit}\cup B^{\ent}$ holds and $B^{\exit}$ is closed in $\partial B$.
\end{dfn}
Obviously, an isolating block is also an isolating neighborhood.
There is a preceding work for the systematic construction of isolating blocks around equilibria \cite{ZM}. 
This method is generalized to (\ref{fast-slow})$_\epsilon$ in \cite{Mat2}, which validates slow manifolds as shown in Section \ref{section-inv-mfd}.
Here we review the {\em predictor-corrector approach} for detecting approximate centers of blocks.
One will see that such procedures are very suitable for analyzing dynamics around invariant manifolds.

\begin{dfn}[cf. \cite{ZG, ZCov, CZ2015}]\rm
\label{dfn-hset}
Let also ${\bf B}_m(x,R)$ be the $m$-dimensional open ball with the center at $x\in \mathbb{R}^m$ and radius $R$.
Let also ${\bf B}_m(R)$ be the $m$-dimensional open ball with the center at the origin, namely, $x=0$ and radius $R$.
Similarly, let ${\bf B}_m$ be the $m$-dimensional open unit ball, namely, $x=0$ and $R=1$.
\par
An {\em $h$-set} consists of the following set, integers and a map:
\begin{itemize}
\item A compact subset $N\subset \mathbb{R}^m$.
\item Nonnegative integers $u(N)$ and $s(N)$ such that $u(N) + s(N) = n$ with $n\leq m$.
\item A homeomorphism $h_N:\mathbb{R}^n\to \mathbb{R}^{u(N)}\times \mathbb{R}^{s(N)}$ satisfying
\begin{equation*}
h_N(N) = \overline{{\bf B}_{u(N)}}\times \overline{{\bf B}_{s(N)}}.
\end{equation*}
\end{itemize}
Similarly, a {\em $ch$-set} consists of the following set, integers and a map:
\begin{itemize}
\item A compact subset $N\subset \mathbb{R}^m$.
\item Nonnegative integers $u(N), s(N)$ and $c(N)$ such that $u(N) + s(N) + c(N) = n$ with $n\leq m$.
\item A homeomorphism $h_N:\mathbb{R}^n\to \mathbb{R}^{u(N)}\times \mathbb{R}^{s(N)}\times \mathbb{R}^{c(N)}$ satisfying
\begin{equation*}
h_N(N) = \overline{{\bf B}_{u(N)}}\times \overline{{\bf B}_{s(N)}}\times \overline{{\bf B}_{c(N)}}.
\end{equation*}
\end{itemize}
Finally define the {\em dimension} of an $h$-set or a $ch$-set $N$ by $\dim N:= n$.
\end{dfn}
We shall write an $h$-set $(N,u(N),s(N),h_N)$ or a $ch$-set $(N,u(N),s(N),c(N),h_N)$ simply by $N$ if no confusion arises. 

\par
\bigskip
Let $\epsilon_0 > 0$ be given
and $(\bar x, \bar y)$ be a (numerical) equilibrium for (\ref{layer}), i.e., $f(\bar x, \bar y,0)\approx 0$, such that $f_x(\bar x, \bar y,0)$ is invertible.
Let $Y$ be a compact neighborhood of $\bar y$ in $\mathbb{R}^l$.
We set the candidate of \lq\lq center line" as follows:
\begin{equation}
\label{new-center}
\left(\bar x + \frac{dx}{dy}(\bar y)(y-\bar y),y\right)\equiv \left(\bar x - f_x(\bar x, \bar y)^{-1}f_y(\bar x, \bar y)(y-\bar y),y\right),
\end{equation}
where $x = x(y)$ is the parametrization of $x$ with respect to $y$ such that $\bar x = x(\bar y)$ and that $f(x(y),y, 0) = 0$, which is actually realized in a small neighborhood of $\bar y$ in $\mathbb{R}^l$ since $f_x(\bar x, \bar y)$ is invertible. Obviously, the identification in (\ref{new-center}) makes sense, which thanks to the Implicit Function Theorem.
\par
Around the center line, we define the affine transformation $T : (z,w)\mapsto (x,y)$ as
\begin{equation*}
(x,y) = T(z,w) := \left(Pz + \bar x - f_x(\bar x, \bar y)^{-1}f_y(\bar x, \bar y)w, w + \bar y\right).
\end{equation*}
where $P$ is a nonsingular matrix diagonalizing $f_x(\bar x, \bar y)$.
In the new $(z,w)$-coordinate, the fast system $x' = f(x,y,\epsilon)$ is transformed into the following:
\begin{align}
\notag
z' &= P^{-1}\left(x' + \overline{f_x}^{-1}\overline{f_y}w'\right)\\
\notag
 	&= P^{-1}\left( f(x,y,\epsilon) + \epsilon\overline{f_x}^{-1}\overline{f_y}g(x,y,\epsilon) \right)\\
\notag
	&= P^{-1}\left( \overline{f_x}(Pz - \overline{f_x}^{-1}\overline{f_y} w) + \hat f(z,w,\epsilon)+ \epsilon \overline{f_x}^{-1}\overline{f_y}g(x,y,\epsilon) \right)\\
\notag
	&= \Lambda z + P^{-1}\left(-\overline{f_y}w + \hat f(z,w,\epsilon)+ \epsilon \overline{f_x}^{-1}\overline{f_y}g\left(Pz + \bar x - \overline{f_x}^{-1} \overline{f_y} w, w + \bar y, \epsilon\right) \right)\\
\label{diag-fast}
	&\equiv  \Lambda z + F(z,w,\epsilon),
\end{align}
where $\overline{f_x} = f_x(\bar x, \bar y)$ and $\overline{f_y}=f_y(\bar x, \bar y)$, and 
\begin{equation*}
\Lambda = \begin{pmatrix}
A & O \\ O & B
\end{pmatrix},\quad A = \diag(\lambda^a_1,\cdots, \lambda^a_{n_u}), \quad B = \diag(\lambda^b_1,\cdots, \lambda^b_{n_s})
\end{equation*}
with $n_u + n_s = n$.
Here every $\lambda_j^a$ and $\lambda_j^b$ is assumed to be real\footnote
{
For the case that eigenvalues $\{\lambda_j\}$ contain complex conjugate pairs, see \cite{ZM} or \cite{Mat2}.
Here we note that we also have the corresponding procedure of isolating blocks even in such a case.
}
for simplicity.
The function $\hat f(z,w,\epsilon)$ denotes the higher order term of $f$ with $O(|z|^2, |w|)$.
Dividing $z$ into $(a,b)$ corresponding to eigenvalues with positive real parts and negative real parts, respectively, we can construct a candidate of desiring blocks.
\par
Note that the higher order term $\hat f(z,w,\epsilon)$ contains the linear term of $w$ as $\overline{f_y}w$ with small errors in a sufficiently small neighborhood $Y$ of $\bar y$.
(\ref{diag-fast}) indicates that the $w$-linear terms are also canceled out in the predictor-corrector approach.
In particular, the residual term $F(z,w,\epsilon)$ is chosen to be $O(|z|^2, |z||w|, |w|^2)$.
\par
We rewrite (\ref{diag-fast}) as the (approximately) block diagonal form:
\begin{equation}
\label{ode-ab-coord}
a' = Aa + F_1(a,b,y,\epsilon),\quad b' = Bb + F_2(a,b,y,\epsilon).
\end{equation}
$F_1$ and $F_2$ are higher order terms depending on $p_0$ and $y_0$. 
Equivalently, writing (\ref{ode-ab-coord}) component-wise, 
\begin{align*}
a_j' &= \lambda_j^a a_j  + F_{1,j}(x,y,\epsilon),\quad \lambda_j^a > 0,\quad j=1,\cdots, n_u,\\
b_j' &= \lambda_j^b b_j  + F_{2,j}(x,y,\epsilon),\quad \lambda_j^b < 0,\quad j=1,\cdots, n_s.
\end{align*}
Let $E\subset \mathbb{R}^n$ be a compact set containing $p_0$. 
Now we assume that each $F_{i,j_i}$, $i=1,2$, $j_1 = 1,\cdots, n_u$, $j_2 = 1,\cdots, n_s$, admits the following enclosure with respect to $E\times Y \times [0,\epsilon_0]$: 
\begin{equation}
\label{error-bounds-fast}
\left\{ F_{i,j_i}(x,y,\epsilon) \mid (x,y)=T(z,w)\in E\times Y, \epsilon \in [0,\epsilon_0]\right\}\subsetneq [\delta_{i,j_i}^-, \delta_{i,j_i}^+].
\end{equation}
Define the set $D_c\subset \mathbb{R}^{n+l}$ by the following: 
\begin{equation}
\label{fast-block}
D_c:= \prod_{j=1}^{n_u} [a_j^-, a_j^+]\times \prod_{j=1}^{n_s} [b_j^-, b_j^+]\times Y,\quad
[a_j^-, a_j^+]:= \left[-\frac{\delta_{1,j}^+}{\lambda^a_j},-\frac{\delta_{1,j}^-}{\lambda^a_j} \right],\quad
[b_j^-, b_j^+]:= \left[-\frac{\delta_{2,j}^-}{\lambda^b_j},-\frac{\delta_{2,j}^+}{\lambda^b_j} \right].
\end{equation}
A series of estimates for error terms involves $E\times Y$ and it only makes sense if it is self-consistent, namely, $TD_c\subset E\times Y$.
Under this self-consistence, we immediately know that
\begin{align*}
a_j ' > 0&\quad \forall (a,b,y,\epsilon) \in D_c\times [0,\epsilon_0] \text{ with }a_j= a_j^+,\\
a_j ' < 0&\quad \forall (a,b,y,\epsilon) \in D_c\times [0,\epsilon_0] \text{ with }a_j= a_j^-,\\
b_j ' < 0&\quad \forall (a,b,y,\epsilon) \in D_c\times [0,\epsilon_0] \text{ with }b_j= b_j^+,\\
b_j ' > 0&\quad \forall (a,b,y,\epsilon) \in D_c\times [0,\epsilon_0] \text{ with }b_j= b_j^-.
\end{align*}
If $\epsilon = 0$, the set $D_c$ is nothing but the isolating block for (\ref{ode-ab-coord})$_0$, equivalently (\ref{layer}).
Once such an isolating block $D_c$ is constructed, one obtains an equilibrium in $TD_c$. 
\begin{prop}[cf. \cite{ZM}]\rm
\label{prop-existence-fixpt}
Let $TD_c$ be an isolating block constructed as above. 
In particular, $TD_c\subset E\times Y$ is assumed.
Then $TD_c$ contains an equilibrium of (\ref{layer}) for all $y \in Y$. 
\end{prop}
This proposition is the consequence of general theory of the Conley index (\cite{Mc}). 
Note that the construction of isolating blocks stated in Proposition \ref{prop-existence-fixpt} around points which are {\em not necessarily equilibria} implies the existence of {\em rigorous} equilibria inside blocks.
With an additional property such as uniqueness or hyperbolicity of equilibria, 
this procedure will provide the smooth $y$-parameter family of equilibria, 
which is stated in Theorem \ref{thm-smooth-slow-mfd}. 
\par
Remark that the above inequalities hold for all $\epsilon \in [0,\epsilon_0]$. This observation is the key point of the construction not only of limiting critical manifolds but of slow manifolds for $\epsilon \in (0,\epsilon_0]$.

\begin{dfn}\rm
Let $D_c\subset \mathbb{R}^{n+l}$ be a $ch$-set constructed by (\ref{fast-block}). 
Assume that $Y = \overline{{\bf B}_l} \subset \mathbb{R}^l$. 
We say $D_c$, equivalently $D:= TD_c$, {\em an affine fast-saddle-type block}. 
Moreover, set
\begin{align*}
D_c^{f,-} &:= \{(a,b,y)\in D_c \mid a_j =a_j^\pm,\ j=1,\cdots, n_u\},\\
D_c^{f,+} &:= \{(a,b,y)\in D_c \mid b_j =b_j^\pm,\ j=1,\cdots, n_s\},\\
D_c^{slow} &:= \{(a,b,y)\in D_c \mid y\in \partial {\bf B}_l\},\\
D^{f,-} &:= TD_c^{f,-},\quad D^{f,+}:= TD_c^{f,+},\quad D^{slow}:= TD_c^{slow}.
\end{align*}
We say $D_c^{f,-}$ (equivalently $D^{f,-}$) {\em the fast-exit of $D$} and $D_c^{f,+}$ (equivalently $D^{f,+}$) {\em the fast-entrance of $D$}. 
\end{dfn}
%


\begin{rem}\rm
We do not assume the transversality of flows on $D_{\bar y} = D\cap \{y=\bar y\}$ with $T(x,\bar y)\in \overline{{\bf B}_{n_u}}\times \overline{{\bf B}_{n_s}} \times \partial {\bf B}_l$.
\end{rem}

\bigskip
This construction can be slightly extended as follows. Let $\{\eta^\alpha\}_{\alpha = u,s}$ be a pair of positive numbers. Defining 
\begin{align}
\notag
&\hat D_c:= \prod_{j=1}^{n_u} [\hat a_j^-, \hat a_j^+]\times \prod_{j=1}^{n_s} [\hat b_j^-, \hat b_j^+]\times Y,\\
\label{fast-block-2}
&[\hat a_j^-, \hat a_j^+]:= \left[-\frac{\delta_{1,j}^+}{\lambda^a_j} - \eta^u,-\frac{\delta_{1,j}^-}{\lambda^a_j} + \eta^u \right],\quad
[\hat b_j^-, \hat b_j^+]:= \left[-\frac{\delta_{2,j}^-}{\lambda^b_j} - \eta^s,-\frac{\delta_{2,j}^+}{\lambda^b_j} + \eta^s \right],
\end{align}
we can prove that $\hat D_c$ is also an affine fast-saddle-type block if $T\hat D_c \subset E\times Y$ holds. We further know
\begin{align*}
a_j ' > 0&\quad \forall (a,b,y,\epsilon) \in \hat D_c\times [0,\epsilon_0] \text{ with }a_j\in [a_j^+, \hat a_j^+],\\
a_j ' < 0&\quad \forall (a,b,y,\epsilon) \in \hat D_c\times [0,\epsilon_0] \text{ with }a_j\in [\hat a_j^-, a_j^-],\\
b_j ' < 0&\quad \forall (a,b,y,\epsilon) \in \hat D_c\times [0,\epsilon_0] \text{ with }b_j\in [b_j^+, \hat b_j^+]\\
b_j ' > 0&\quad \forall (a,b,y,\epsilon) \in \hat D_c\times [0,\epsilon_0] \text{ with }b_j\in [\hat b_j^-, b_j^-].
\end{align*}

This extension leads to the explicit lower bound estimate of distance between $\hat D^{f,\pm}$ and slow manifolds.

\subsection{Logarithmic norms}
\label{section-CZ}
The basic strategy of our smoothness validation is an application of the following result shown in \cite{CZ2015} to time-$t$ maps $\varphi_\epsilon(t,\cdot)$ for (\ref{fast-slow})$_\epsilon$ with sufficiently small $t>0$.
Here we briefly review the smoothness validation procedures of center-(un)stable manifolds discussed in \cite{CZ2015, CZ2016}.

\bigskip
Before our main discussions, we give several notations in this subsection.
\begin{dfn}\rm
For a squared matrix $A\in \mathbb{R}^{n\times n}$, define the matrix norm $m(A)$ by
\begin{equation*}
m(A) = \inf_{z\in \mathbb{R}^n, \|z\| =1}\|Az\|,
\end{equation*}
which in general depends on the norm $\|\cdot \|$ on $\mathbb{R}^n$.
The {\em logarithmic norm} of $A$ denoted by $l(A)$ is given by
\begin{equation*}
l(A) = \lim_{h\to +0}\frac{\|I+hA\|-1}{h}
\end{equation*}
and the {\em logarithmic minimum} of $A$ is given by
\begin{equation*}
m_l(A) = \lim_{h\to +0}\frac{m(I+hA)-1}{h}.
\end{equation*}
\end{dfn}

We gather several fundamental facts of $l(A)$, $m(A)$ and $m_l(A)$ in the following lemma.
\begin{lem}[cf. \cite{CZ2015, CZ2016}]
\label{lem-log}
\begin{enumerate}
\item The limits in the definition of $l(A)$ and $m_l(A)$ exist and we have $m_l(A)=-l(-A)$.
\item For the Euclidean norm, we also have
\begin{align*}
l(A) &= \max\{\lambda \in {\rm Spec}((A+A^T)/2)\},\quad 
m_l(A) = \min\{\lambda \in {\rm Spec}((A+A^T)/2)\}.
\end{align*}
\item Assume $A\in W$ for some compact set $W\subset \mathbb{R}^{n\times n}$.
Assume that $h\in (0,h_0]$ for some $h_0 > 0$. Then we have
\begin{align*}
\|I+hA\| &= 1+hl(A) + r_1(h,A),\quad \|r_1(h,A)\|\leq M_1h^2,\\
m(I+hA) &= 1+hm_l(A) + r_2(h,A),\quad \|r_2(h,A)\|\leq M_2h^2
\end{align*}
for some constants $M_i = M_i(h_0,W) > 0$.
\end{enumerate}
\end{lem}

We further have the following lemma, which is used for validating existence and smoothness of slow manifolds for (\ref{fast-slow}).
\begin{lem}
\label{lem-monotome-log}
Let $A$ be a square matrix and $B$ is a square positive semidefinite matrix.
Then, with the matrix operator norm $\|A\|$ induced by the Euclidean norm, we have
\begin{equation*}
l(A) \leq l(A+B),\quad m_l(A) \leq m_l(A+B).
\end{equation*}
\end{lem}
\begin{proof}
Let $S(A)$ be the symmetrization of $A$: namely, $S(A) = (A+A^T)/2$.
Now Lemma \ref{lem-log}-2 shows that $l(A)$ is the maximum eigenvalue of $S(A)$.
Let $z$ be the associated eigenvector of $l(A)$ with $|z| = 1$. 
Then we have
\begin{align*}
l(A) &= l(A)|z|^2 = z^T S(A)z\\
	&\leq z^T S(A)z + z^T S(B)z = z^T S(A+B)z \\
	&\leq l(A+B)|z|^2 = l(A+B),
\end{align*}
which shows $l(A)\leq l(A+B)$.
\par
Similarly, let $w$ be the associated eigenvector of $m_l(A+B)$ with $|w| = 1$. 
Now Lemma \ref{lem-log}-2 again shows that $m_l(A+B)$ is the minimum eigenvalue of $S(A+B)$.
Thus we have
\begin{align*}
m_l(A+B) &= m_l(A+B)|w|^2 = w^T S(A+B)w = w^T S(A)w + w^T S(B)w\\
	&\geq w^T S(A)w \\
	&\geq m(A)|w|^2 = m_l(A),
\end{align*}
which shows $m_l(A)\leq m_l(A+B)$.
\end{proof}

\begin{lem}
\label{lem-max-min}
Let $A\in \mathbb{R}^{n\times n}$ be a matrix and $x\in \mathbb{R}^n$.
Then, under the standard Euclidean norm, the following inequality holds:
\begin{equation*}
m_l(A)|x|^2 \leq x^T A x \leq l(A) |x|^2.
\end{equation*}
\end{lem}

\begin{proof}
In general, $x^T A x = (x^T A x)^T = x^T A^T x$, and hence we have
\begin{equation*}
x^T Ax = \frac{1}{2} (x^T A x + x^T A^T x) = x^T S(A) x,\quad S(A) = \frac{1}{2}(A+A^T).
\end{equation*}
In general, the inequality $\lambda_{\min} |x|^2 \leq x^T B x \leq \lambda_{\max} |x|^2$ holds for any symmetric matrix $B$, where $\lambda_{\min}$ and $\lambda_{\max}$ are the minimum and the maximum eigenvalue of $B$, respectively. 
Apply this inequality to $B=S(A)$, and we have our statement by using Lemma \ref{lem-log}-2.
\end{proof}

\subsection{Rate conditions for flows and maps}
\label{section-rate}
In this subsection, we review results in \cite{CZ2015, CZ2016} concerning with the existence and smoothness of invariant manifolds, called {\em rate conditions} with a few modifications, which is partially discussed in \cite{Mat2} for proving the existence of slow manifolds for (\ref{fast-slow}).

\par
First consider the vector field\footnote
{
If the vector field depends on parameters $\epsilon$, we incorporate parameters into center variable $y$ with trivial evolution $\epsilon'=0$.
}:
\begin{align}
\label{flow-param}
&z' = F(z),\quad F : \mathbb{R}^{n_u + n_s + l} \to \mathbb{R}^{n_u + n_s + l},\\
\notag
&z = (a,b,y)^T\in \mathbb{R}^{n_u + n_s + l}, \quad F(z) = (F_a(z), F_b(z), F_y(z))^T.
\end{align}
The map $F$ is assumed to be $C^{k+1}$ for $k\geq 1$.

\begin{dfn}[Rate conditions for flows, cf. \cite{CZ2016}]\rm
\label{dfn-rate-flow}
Consider (\ref{flow-param}) and $D\subset \mathbb{R}^{n_u + n_s + l}$ be a $ch$-set.
For $M > 1$, let
\begin{align*}
\overrightarrow{\mu_{s,1}} &= \overrightarrow{\mu_{s,1}}(D) = \sup_{z\in D} \left\{ l\left(\frac{\partial F_b}{\partial b}(z) \right) + \frac{1}{M} \left\| \frac{\partial F_b}{\partial (a,y)}(z)\right\|\right\}, \\
\overrightarrow{\mu_{s,2}} &= \overrightarrow{\mu_{s,2}}(D) = \sup_{z\in D} \left\{ l\left(\frac{\partial F_b}{\partial b}(z) \right) + M \left\| \frac{\partial F_{(a,y)}}{\partial b}(z)\right\|\right\},\\ 
\overrightarrow{\xi_{u,1}} &= \overrightarrow{\xi_{u,1}}(D) = \inf_{z\in D}m_l\left( \frac{\partial F_a}{\partial a}(z) \right) - \frac{1}{M}\sup_{z\in D} \left\| \frac{\partial F_a}{\partial (b,y)}(z) \right\|,\\
\overrightarrow{\xi_{u,2}} &= \overrightarrow{\xi_{u,2}}(D) = \inf_{z\in D} \left\{ m_l\left( \frac{\partial F_a}{\partial a}(z) \right) - M\left\| \frac{\partial F_{(b,y)}}{\partial a}(z) \right\| \right\},\\
\overrightarrow{\mu_{ss,1}} &= \overrightarrow{\mu_{ss,1}}(D) = \sup_{z\in D} \left\{ l\left(\frac{\partial F_{(b,y)}}{\partial (b,y)}(z) \right) + M \left\| \frac{\partial F_{(b,y)}}{\partial a}(z)\right\|\right\}, \\ 
\overrightarrow{\mu_{ss,2}} &= \overrightarrow{\mu_{ss,2}}(D) = \sup_{z\in D} \left\{ l\left(\frac{\partial F_{(b,y)}}{\partial (b,y)}(z) \right) + \frac{1}{M}\left\| \frac{\partial F_a}{\partial (b,y)}(z)\right\|\right\},\\ 
\overrightarrow{\xi_{su,1}} &= \overrightarrow{\xi_{su,1}}(D) = \inf_{z\in D}m_l\left( \frac{\partial F_{(a,y)}}{\partial (a,y)}(z) \right) - M \sup_{z\in D} \left\| \frac{\partial F_{(a,y)}}{\partial b}(z) \right\|,\\
\overrightarrow{\xi_{su,2}} &= \overrightarrow{\xi_{su,2}}(D) = \inf_{z\in D} \left\{ m_l\left( \frac{\partial F_{(a,y)}}{\partial (a,y)}(z) \right) - \frac{1}{M}\left\| \frac{\partial F_b}{\partial (a,y)}(z) \right\| \right\}.
\end{align*}
We shall call these constants the {\em (local) rates} of $F$ in $D$.
\par
For $k\geq 1$, we say that the vector field $F$ satisfies the {\em rate condition of order $k$} in $D$ if, for all $j\in \{1,\cdots, k\}$,
\begin{align}
\label{rate-fss-1}
&\overrightarrow{\mu_{s,1}} < 0 < \overrightarrow{\xi_{u,1}},\\
\label{rate-fss-2}
&\overrightarrow{\mu_{ss,1}} < \overrightarrow{\xi_{u,1}},\quad \overrightarrow{\mu_{s,1}} < \overrightarrow{\xi_{su,1}},\\
\label{rate-fss-3}
&(j+1)\overrightarrow{\mu_{ss,1}} < \overrightarrow{\xi_{u,2}},\quad \overrightarrow{\mu_{s,2}} < (j+1)\overrightarrow{\xi_{su,1}},\\
\label{rate-fss-4}
&\overrightarrow{\mu_{ss,2}} < \overrightarrow{\xi_{u,1}}, \quad \overrightarrow{\mu_{s,1}} < \overrightarrow{\xi_{su,2}}.
\end{align}
\end{dfn}
Note that the rate condition of order $0$ is also discussed in \cite{Mat2} for the existence of (un)stable manifolds of slow manifolds.

\par
\bigskip
Similarly consider the map evolution\footnote
{
If the map depends on parameters $\epsilon$, we incorporate parameters into center variable $y$ with trivial evolution $\epsilon \mapsto \epsilon$.
}
\begin{align}
\label{map-param}
&(\mathbb{R}^{n_u + n_s + l}\ni) z = (a,b,y)^T \mapsto G(z) \in \mathbb{R}^{n_u + n_s + l},\\
\notag
&z\in \mathbb{R}^{n_u + n_s + l},\quad G(z) = (G_a(z), G_b(z), G_y(z))^T.
\end{align}
The map $G$ is assumed to be $C^{k+1}$ for $k\geq 1$.

\begin{dfn}[Rate conditions for maps, cf. \cite{CZ2015}]\rm
\label{dfn-rate-map}
Consider (\ref{map-param}) and $D\subset \mathbb{R}^{n_u + n_s + l}$ be a $ch$-set.
For $M > 1$, let
\begin{align*}
\mu_{s,1} &= \sup_{z\in D} \left\{ \left\|\frac{\partial G_b}{\partial b}(z) \right\| + \frac{1}{M} \left\| \frac{\partial G_b}{\partial (a,y)}(z)\right\|\right\}, \quad 
\mu_{s,2} = \sup_{z\in D} \left\{ \left\|\frac{\partial G_b}{\partial b}(z) \right\| + M \left\| \frac{\partial G_{(a,y)}}{\partial b}(z)\right\|\right\},\\ 
\xi_{u,1} &= \inf_{z\in D}m\left( \frac{\partial G_a}{\partial a}(z) \right) - \frac{1}{M}\sup_{z\in D} \left\| \frac{\partial G_a}{\partial (b,y)}(z) \right\|,\quad
\xi_{u,2} = \inf_{z\in D} \left\{ m\left( \frac{\partial G_a}{\partial a}(z) \right) - M\left\| \frac{\partial G_{(b,y)}}{\partial a}(z) \right\| \right\},\\
\mu_{cs,1} &= \sup_{z\in D} \left\{ \left\|\frac{\partial G_{(b,y)}}{\partial (b,y)}(z) \right\| + M \left\| \frac{\partial G_{(b,y)}}{\partial a}(z)\right\|\right\}, \quad 
\mu_{cs,2} = \sup_{z\in D} \left\{ \left\|\frac{\partial f_{(b,y)}}{\partial (b,y)}(z) \right\| + \frac{1}{M}\left\| \frac{\partial G_a}{\partial (b,y)}(z)\right\|\right\},\\ 
\xi_{cu,1} &= \inf_{z\in D}m\left( \frac{\partial G_{(a,y)}}{\partial (a,y)}(z) \right) - M\sup_{z\in D} \left\| \frac{\partial G_{(a,y)}}{\partial b}(z) \right\|,\quad
\xi_{cu,2} = \inf_{z\in D} \left\{ m\left( \frac{\partial G_{(a,y)}}{\partial (a,y)}(z) \right) - \frac{1}{M}\left\| \frac{\partial G_b}{\partial (a,y)}(z) \right\| \right\}.
\end{align*}
As in the case of vector fields, we shall call these constants the {\em (local) rates} of $G$ in $D$.
\par
We say that $G$ satisfies the {\em rate condition of order $k\geq 1$} if $\xi_{u,1}, \xi_{u,2}, \xi_{cu,1}$ and $\xi_{cu,2}$ are strictly positive, and for all $j\in \{1,\cdots, k\}$, the following inequalities hold true:
\begin{align}
\label{rate-map-1}
&\mu_{s,1} < 1 < \xi_{u,1},\\
\label{rate-map-2}
&\mu_{cs,1} < \xi_{u,1},\quad \mu_{s,1} < \xi_{cu,1},\\
\label{rate-map-3}
&(\mu_{cs,1})^{j+1} < \xi_{u,2},\quad \mu_{s,2} < (\xi_{cu,1})^{j+1},\\
\label{rate-map-4}
&\mu_{cs,2} < \xi_{u,1},\quad \mu_{s,1} < \xi_{cu,2}.
\end{align}
We say that $G$ satisfies the {\em rate condition of order $0$} if only (\ref{rate-map-1}) and (\ref{rate-map-2}) are satisfied.
\end{dfn}

Rate conditions with additional geometric conditions yield the existence and smoothness of invariant manifolds.
Let $\Phi$ be the flow generated by (\ref{flow-param}) and $\Phi_h = \Phi(h,\cdot)$ be the corresponding time-$h$ map.
Then we have the correspondence of rates between for flows and for time-$h$ maps, as stated in Proposition \ref{prop-rate-flow-map} below.
Following notations in \cite{CZ2015}, we shall use one of the pair of variables:
\begin{itemize}
\item ${\rm x} = a$, ${\rm y} = (b,y)$,
\item ${\rm x} = (a, y)$, ${\rm y} = b$.
\end{itemize}
For $M > 0$ and $h>0$, define
\begin{align*}
\overrightarrow{\xi_1(M)} &= \inf_{z\in D}m\left( \frac{\partial F_{\rm x}}{\partial {\rm x}}(z) \right) - M\sup_{z\in D} \left\| \frac{\partial F_{\rm x} }{\partial {\rm y}}(z) \right\|,\quad
\overrightarrow{\xi_2(M)} = \inf_{z\in D} \left\{ m\left( \frac{\partial F_{\rm x} }{\partial {\rm x}}(z) \right) - M\left\| \frac{\partial F_{\rm y}}{\partial {\rm x}}(z) \right\| \right\},\\
\overrightarrow{\mu_1(M)} &= \sup_{z\in D} \left\{ \left\|\frac{\partial F_{\rm y}}{\partial {\rm y}}(z) \right\| + M \left\| \frac{\partial F_{\rm y}}{\partial {\rm x}}(z)\right\|\right\}, \quad 
\overrightarrow{\mu_2(M)} = \sup_{z\in D} \left\{ \left\|\frac{\partial F_{\rm y}}{\partial {\rm y}}(z) \right\| + M\left\| \frac{\partial F_{\rm x} }{\partial {\rm y}}(z)\right\|\right\},\\
\xi_1(h,M) &= \inf_{z\in D}m\left( \frac{\partial \Phi_{\rm x}}{\partial {\rm x}}(h,z) \right) - M\sup_{z\in D} \left\| \frac{\partial \Phi_{\rm x} }{\partial {\rm y}}(h,z) \right\|,\\
\xi_2(h,M) &= \inf_{z\in D} \left\{ m\left( \frac{\partial \Phi_{\rm x} }{\partial {\rm x}}(h,z) \right) - M\left\| \frac{\partial \Phi_{\rm y}}{\partial {\rm x}}(h,z) \right\| \right\},\\
\mu_1(h,M) &= \sup_{z\in D} \left\{ \left\|\frac{\partial \Phi_{\rm y}}{\partial {\rm y}}(h,z) \right\| + M \left\| \frac{\partial \Phi_{\rm y}}{\partial {\rm x}}(h,z)\right\|\right\}, \\
\mu_2(h,M) &= \sup_{z\in D} \left\{ \left\|\frac{\partial \Phi_{\rm y}}{\partial {\rm y}}(h,z) \right\| + M \left\| \frac{\partial f\Phi_{\rm x} }{\partial {\rm y}}(h,z)\right\|\right\}.
\end{align*}

\begin{prop}[Correspondence of rate conditions. cf. Theorem 31 in \cite{CZ2016}]
\label{prop-rate-flow-map}
Let $M, M_1, M_2 > 0$. Then the following assertions hold true:
\begin{enumerate}
\item 
\begin{align}
\xi_1(h,M) &= 1+h \overrightarrow{\xi_1(M)} + O(h^2),\\
\label{rate-flow-map-xi2}
\xi_2(h,M) &= 1+h \overrightarrow{\xi_2(M)} + O(h^2),\\
\mu_1(h,M) &= 1+h \overrightarrow{\mu_1(M)} + O(h^2),\\
\mu_2(h,M) &= 1+h \overrightarrow{\mu_2(M)} + O(h^2).
\end{align}
\item If the inequality $\overrightarrow{\mu_2}(M_1) < (j+1)\overrightarrow{\xi_1}(M_2)$ holds for $j\geq 0$, there is a sufficiently small $h_0 > 0$ such that, for any $h\in (0,h_0)$, the following inequality holds:
\begin{equation*}
\mu_2(h,M_1) < \xi_1(h,M_2)^{j+1}.
\end{equation*}
Similarly, if the inequality $(j+1)\overrightarrow{\mu_1}(M_1) < \overrightarrow{\xi_2}(M_2)$ holds for $j\geq 0$, there is a sufficiently small $h_0 > 0$ such that, for any $h\in (0,h_0)$, the following inequality holds:
\begin{equation}
\label{rate-flow-map-stable1}
\mu_1(h,M_1)^{j+1} < \xi_2(h,M_2).
\end{equation}
\item If $\overrightarrow{\mu_1}(M_1) < \overrightarrow{\xi_1}(M_2)$ holds, there is a sufficiently small $h_0 > 0$ such that for any $h\in (0,h_0)$ the following inequality holds:
\begin{equation*}
\mu_1(h,M_1) < \xi_1(h,M_2).
\end{equation*}
\item If $\overrightarrow{\xi_1}(M) > 0$ holds, there is a sufficiently small $h_0 > 0$ such that for any $h\in (0,h_0)$ the  inequality $\xi_1(h,M) > 1$ holds.
\item If $\overrightarrow{\mu_1}(M) < 0$ holds, there is a sufficiently small $h_0 > 0$ such that for any $h\in (0,h_0)$ the  inequality $\mu_1(h,M) < 1$ holds.
\end{enumerate}
\end{prop}
\begin{proof}
All statements except (\ref{rate-flow-map-xi2}) and (\ref{rate-flow-map-stable1}) are exactly Theorem 31 in \cite{CZ2016}.
\par
The expansion (\ref{rate-flow-map-xi2}) can be proved by the same arguments in Statement 1 of Theorem 31 in \cite{CZ2016}.
We shall prove (\ref{rate-flow-map-stable1}).
Since
\begin{equation*}
\mu_1(h,M_1)^{j+1} = \left(1 + h \overrightarrow{\mu_1(M_1)} + O(h^2)\right)^{j+1} = 1+h(j+1)\overrightarrow{\mu_1(M_1)} + O(h^2),
\end{equation*}
from the assumption $(j+1)\overrightarrow{\mu_1}(M_1) < \overrightarrow{\xi_2}(M_2)$, we have
\begin{align*}
\mu_1(h,M_1)^{j+1} &= 1+h(j+1)\overrightarrow{\mu_1(M_1)} + O(h^2)\\
	&< 1 + h\overrightarrow{\xi_2(M_2)} + O(h^2) = \xi_2(h,M_2) + O(h^2),
\end{align*}
and the claim holds true for all sufficiently small $h > 0$.
\end{proof}
The proposition indicates that  rate conditions for flows yield those for time-$h$ maps with sufficiently small $h > 0$.
Combining the correspondence of invariant manifolds between for flows and for time-$h$ maps stated in Proposition \ref{prop-manifold-flow-map} below, all arguments for invariant manifolds for flows are reduced to the case for maps stated in \cite{CZ2015}.

\subsection{Summaries for normally hyperbolic invariant manifolds in \cite{CZ2015}}
\label{section-summary-CZ}
Here we gather central results about invariant manifold validations stated in \cite{CZ2015}.
\begin{dfn}[Center-(un)stable manifolds, \cite{CZ2015}]\rm
\label{dfn-cu-mfd-CZ}
Consider the map (\ref{map-param}).
Let $D\subset \mathbb{R}^{n_u + n_s + l}$ be a $ch$-set.
We define the {\em center-stable set} in $D$ as
\begin{equation*}
W^{cs} = \{z\in D\mid F^m(z)\in D\text{ for all }m\in \mathbb{N}\}.
\end{equation*}
Similarly, define the {\em center-unstable set} in $D$ as
\begin{equation*}
W^{cu} = \{z\in D\mid \text{ there is a full backward trajectory of }Z \text{ in }D\}.
\end{equation*}
Finally, define the maximal invariant set in $D$ as
\begin{equation*}
\Lambda^\ast = \{z\in D\mid \text{ there is a full trajectory of }Z \text{ in }D\}.
\end{equation*}
\end{dfn}

Next we state the following topological and geometric conditions, which is known as {\em covering relations} in e.g., \cite{ZG, ZCov}.

\begin{prop}
\label{prop-isol-cov}
Assume that $D = \overline{{\bf B}_{n_u}}\times \overline{{\bf B}_{n_s}}\times \overline{{\bf B}_l}$ is an isolating block for (\ref{flow-param}) such that
\begin{itemize}
\item $\partial {\bf B}_{n_u}\times \overline{{\bf B}_{n_s}}\times \overline{{\bf B}_l}$ is an exit;
\item $\overline{{\bf B}_{n_u}}\times \partial {\bf B}_{n_s}\times \overline{{\bf B}_l}$ is an entrance;
\item $\overline{{\bf B}_{n_u}}\times \overline{{\bf B}_{n_s}}\times \partial {\bf B}_l$ is either of an entrance or an exit,
\end{itemize}
and let $F_t(q) = \Phi(t,q)$.
If $t$ is sufficiently small, then $F_t$ satisfies \lq\lq covering condition"; namely,
\begin{enumerate}
\item There exists a continuous homotopy $h:[0,1]\times N_c\to \mathbb{R}^{n_u}\times \mathbb{R}^{n_s}$ satisfying
\begin{align*}
&h_0 = f_c,\\
&h([0,1],N_c^-)\cap M_c = \emptyset,\\
&h([0,1],N_c)\cap M_c^+ = \emptyset,
\end{align*}
where $h_\lambda = h(\lambda, \cdot)$ ($\lambda \in [0,1]$).
\item There exists a mapping $A:\mathbb{R}^{n_u}\to \mathbb{R}^{n_u}$ such that
\begin{equation}
\label{cov-degree}
\begin{cases}
h_1(p,q) = (A(p),0), &\\
A(\partial {\bf B}_{n_u} (0,1)) \subset \mathbb{R}^u \setminus \overline{{\bf B}_{n_u}}(0,1), & \\
\deg(A, \overline{{\bf B}_{n_u}}, 0)\not = 0
\end{cases}
\end{equation}
holds for $p\in \overline{{\bf B}_{n_u}}(0,1), q\in \overline{{\bf B}_{n_s}}(0,1)$.
\end{enumerate}
\end{prop}

\begin{proof}
Let $C : (x,y)\to (x,-y)$ and for $\alpha \in [0,1/2]$, let
\begin{equation*}
H_\alpha = (1-2\alpha)f + 2\alpha C.
\end{equation*}
For any $q\in \partial {\bf B}_{n_u} \times \overline{{\bf B}_{n_s}}$,
\begin{equation*}
(\pi_x H_\alpha(q) \mid \pi_x q) = (1-2\alpha)(\pi_x f(q) \mid \pi_x q) + 2\alpha(\pi_x q \mid \pi_x q) > 0,
\end{equation*}
and for any $q\in \overline{{\bf B}_{n_u}} \times \partial {\bf B}_{n_s}$,
\begin{equation*}
(\pi_y H_\alpha(q) \mid \pi_y q) = (1-2\alpha)(\pi_y f(q) \mid \pi_y q) - 2\alpha(\pi_y q \mid \pi_y q) < 0.
\end{equation*}
Let $\phi_\alpha(t,q)$ be the flow generated by $q' = H_\alpha(q)$.
Note that
\begin{equation*}
\phi_{1/2}(t,(x,y)) = (e^t x, e^{-t}y).
\end{equation*}
Fix a time $t$ being sufficiently small and define
\begin{equation*}
h_\alpha(x,y) = \begin{cases}
\phi_\alpha(t,q) & \alpha\in [0,1/2)\\
(e^t x, (2-2\alpha)e^{-t}y) & \alpha\in [1/2,1]
\end{cases}.
\end{equation*}
All conditions of covering condition follows from the definition of $h_\alpha$ and the isolation.
See \cite{CZ2015} for example.
\end{proof}

\begin{dfn}[Definition 13 in \cite{CZ2015}]\rm
\label{dfn-back-cone}
We say that the map $F$ satisfies {\em backward cone conditions} if the following condition holds:
If $Z_1, Z_2, F(Z_1), F(Z_2)\in D$ and $F(Z_1)\in C^s_{M}(F(Z_2))$, then we have $Z_1\in C^s_{M}(Z_2)$, where $C^s_{M}(Z_0)$ is the stable cone with the vertex $Z_0 = (a_0,b_0,y_0)$ defined by
\begin{equation*}
C^s_{M}(Z_0) = \left\{ (a,b,y)\in \mathbb{R}^{n_u+n_s+l} \mid \|b-b_0\|^2 \geq M^2 (\|a-a_0\|^2 + \|y-y_0\|^2)\right\}\footnote
{
The corresponding definition in \cite{CZ2015} is $C^s_{1/L}(Z_0) \equiv J_s(Z_0; L)= \left\{ L^2\|b-b_0\|^2 \geq (\|a-a_0\|^2 + \|y-y_0\|^2)\right\}$, in which case the constant $L$ is assumed to be $L<1$.
The constant $L$ is known as the {\em slope} of cone.
We choose our current definition of cones following arguments in \cite{Mat2}.
}.
\end{equation*}
\end{dfn}

The main result for the existence of smooth invariant manifolds is the following, which is stated in \cite{CZ2015} replacing $\overline{{\bf B}_l}$ by  an $l$-dimensional torus $\Lambda = (\mathbb{R}/\mathbb{Z})^l$ with slight modifications of all concepts stated in Section \ref{section-rate} and here.
\begin{prop}[Theorem 16 in \cite{CZ2015}]
Let $k\geq 1$, $R < \frac{1}{2}R_\Lambda$\footnote{
$R_\Lambda$ is a positive number associated with $\Lambda$. See Remark \ref{rem-CZ} for details.
}, and $f:D = \overline{{\bf B}_{n_u}(R)}\times \overline{{\bf B}_{n_s}(R)}\times \Lambda \to \mathbb{R}^{n_u}\times \mathbb{R}^{n_s}\times \Lambda$ be a $C^{k+1}$ map, where $\Lambda$ is an $l$-dimensional torus.
If $f$ satisfies rate conditions of order $k$ with covering conditions and backward cone conditions with $M$ satisfying $1/M \in (2R/R_\Lambda, 1)$, then $W^{cs}, W^{cu}$ and $\Lambda^\ast\equiv W^{cs}\cap W^{cu}$ are $C^k$ manifolds in $D$, which are the graphs of $C^k$ functions
\begin{equation*}
w^{cs}: \overline{{\bf B}_{n_s}(R)}\times \Lambda \to \overline{{\bf B}_{n_u}(R)},\quad w^{cu}: \overline{{\bf B}_{n_u}(R)}\times \Lambda \to \overline{{\bf B}_{n_s}(R)},\quad w^{c}:\Lambda \to \overline{{\bf B}_{n_u}(R)}\times \overline{{\bf B}_{n_s}(R)},
\end{equation*}
respectively. That is,
\begin{align*}
W^{cu} &=\{(\lambda, w^{cs}(y,\lambda),y)\mid \lambda\in \Lambda, y\in \overline{{\bf B}_{n_s}(R)}\},\\
W^{cs} &=\{(\lambda, x,w^{cu}(x,\lambda))\mid \lambda\in \Lambda, x\in \overline{{\bf B}_{n_u}(R)}\},\\
\Lambda^\ast &=\{(\lambda, w^c(\lambda))\mid \lambda\in \Lambda\}.
\end{align*}
Moreover, $f|_{W^{cu}}$ is an injection, $w^{cs}$ and $w^{cu}$ are Lipschitz with constants $1/M$, and $w^c$ is Lipschitz with the constant $\sqrt{2}/\sqrt{M^2-1}$.
The manifolds $W^{cs}$ and $W^{cu}$ intersect transversally, and $W^{cs}\cap W^{cu}=\Lambda^\ast$.
See Remark \ref{rem-CZ} for treatments of $\Lambda$.
\end{prop}
In the above result, we omitted statements about {\em invariant foliations} of manifolds because they are out of our focus in present arguments.

\begin{rem}\rm
Briefly speaking, inequalities (\ref{rate-map-1}) and (\ref{rate-map-2}) describe the invariance of cones, which yields the existence of $W^{cs}$, $W^{cu}$ and their invariant foliations.
The additional inequalities (\ref{rate-map-3}) and (\ref{rate-map-4}) show the $C^k$-smoothness of validated manifolds.
\end{rem}

\begin{rem}\rm
In \cite{CZ2016}, only the rate condition for center-unstable manifolds is considered.
In the above definition we also state the rate condition for center-stable manifolds: namely, $(j+1)\overrightarrow{\mu_{cs,1}} < \overrightarrow{\xi_{u,2}}$ and $\overrightarrow{\mu_{s,1}} < \overrightarrow{\xi_{cu,2}}$.
\end{rem}

The key point of the existence and smoothness of invariant manifolds for flows is to reduce the problem into those for time-$t$ maps for sufficiently small $t>0$.
The reduction is realized by the following.

Firstly, the following proposition shows that the coincidence of center-(un)stable manifolds for flows and those for time-$t$ maps.

\begin{prop}[Correspondence of center-(un)stable manifolds. cf. Proof of Theorem 30 in \cite{CZ2016}]
\label{prop-manifold-flow-map}
Let $\Phi$ be a flow on $\mathbb{R}^{n_u+n_s+l}$.
For $h>0$, let $\Phi_h = \Phi(h,\cdot)$.
The sets $W^{cu}(\Phi)$ and $W^{cu}(\Phi_h)$ denote the center-unstable manifolds for the flow $\Phi$ and the map $\Phi_h$, respectively.
Similarly, the sets $W^{cs}(\Phi)$ and $W^{cs}(\Phi_h)$ denote the center-stable manifolds for the flow $\Phi$ and the map $\Phi_h$, respectively.
\par
Let $N^u = \overline{{\bf B}_{n_u}}\times \overline{{\bf B}_{n_s}} \times \overline{{\bf B}_l}$ be an isolating block for $\Phi$ with the entrance $N^{u,+} = \overline{{\bf B}_{n_u}}\times \partial {\bf B}_{n_s} \times \overline{{\bf B}_l}$ and the exit $N^{u,-} = \partial {\bf B}_{n_u}\times \overline{{\bf B}_{n_s}} \times \partial {\bf B}_l$.
Then there is a positive number $h_0 > 0$ such that $W^{cu}(\Phi)\cap N^u = W^{cu}(\Phi_h)\cap N^u$ holds for all $h\in (0,h_0)$.
\par
Similarly, let $N^s = \overline{{\bf B}_{n_u}}\times \overline{{\bf B}_{n_s}} \times \overline{{\bf B}_l}$ be an isolating block for $\Phi$ with the entrance $N^{s,+} = \overline{{\bf B}_{n_u}}\times \partial {\bf B}_{n_s} \times \partial {\bf B}_l$ and the exit $N^{s,-} = \partial {\bf B}_{n_u}\times \overline{{\bf B}_{n_s}} \times \overline{{\bf B}_l}$.
Then there is a positive number $h_0 > 0$ such that $W^{cs}(\Phi)\cap N^u = W^{cs}(\Phi_h)\cap N^u$ holds for all $h\in (0,h_0)$.
\end{prop}

\begin{proof}
The first assertion is discussed in the proof of Theorem 30 in \cite{CZ2016}.
Although the proof of the second assertion is basically the same as the first, we state the proof of the second assertion for readers who are not familiar with arguments in this direction.
\par
The inclusion $W^{cs}(\Phi_h)\supset W^{cs}(\Phi)$ is obvious for any $h > 0$.
We then prove that, for suitable small $h > 0$, $W^{cs}(\Phi_h)\subset W^{cs}(\Phi)$. 
We shall rewrite $N^s\equiv N$ for simplicity.
Since $N$ is an isolating block, then the exit $N^-$ is compact and hence there is a $\delta > 0$ such that
\begin{equation}
\label{exit}
\Phi(s,z)\not \in N\quad \text{ for all }s\in (0,\delta] \text{ and } z\in N^-.
\end{equation}
Choose $h < \delta$.
Such a choice will prove our claim; namely, for any $z\in W^{cs}(\Phi_h)$ we can prove $\Phi(t,z)\in N$ for any $t > 0$.
\par
Assume that $z\in W^{cs}(\Phi_h)$. 
Then, for any $m\in \mathbb{N}$, we have
\begin{equation}
\label{inclusion-cs}
\Phi_h^m(z) = \Phi(mh, z)\in N.
\end{equation}
Assume further that for some $t > 0$, $\Phi(t,z)\not \in N$. 
By (\ref{inclusion-cs}), we have $mh < t <(m+1)h$ for some $n\in \mathbb{N}$.
Since $N$ is an isolating block, the only possibility to leave $N$ is that the trajectory cross the exit $N^-$.
We thus know that, for some $\tau^\ast \in (mh, t)$, $z^\ast\equiv \Phi(\tau^\ast,z)\in N^-$.
We then see that
\begin{equation*}
\Phi((m+1)h -\tau^\ast, z^\ast) = \Phi((m+1)h -\tau^\ast, \Phi(\tau^\ast,z)) = \Phi((m+1)h,z)\in N
\end{equation*}
from the assumption $z\in W^{cs}(\Phi_h)$. 
But it contradicts (\ref{exit}) by taking $s = (m+1)h -\tau^\ast \in (0,\delta)$.
We thus have, $\Phi(t,z)\not \in N$ for any $t > 0$, which indicates $z\in W^{cs}(\Phi)$ and hence $W^{cs}(\Phi_h)\subset W^{cs}(\Phi)$.
\end{proof}

Secondly, the covering condition for time-$t$ maps with sufficiently small $t > 0$ is derived from isolating blocks, as stated in Proposition \ref{prop-isol-cov}.
Finally, we can prove that backward cone conditions for time-$t$ maps (Definition \ref{dfn-back-cone}) can be automatically constructed by the rate condition of order $0$ for flows.
We see this consequence in the next section.
\par
As a consequence, Propositions \ref{prop-rate-flow-map} and \ref{prop-manifold-flow-map} as well as the above observations reduce problems concerning with center-(un)stable manifolds for flows to those for maps.
Therefore, the rate condition in Definition \ref{dfn-rate-flow} gives the existence as well as their smoothness of invariant manifolds for flows.

\begin{rem}\rm
\label{rem-CZ}
We gather several comments about discussions in \cite{CZ2015, CZ2016} and our present focus.
In \cite{CZ2015}, the center variable is assumed to belong to an $l$-dimensional closed manifold $\Lambda$ such as a torus.
In this case, we have to care about a good chart in terms of, say, a covering map $\phi : \mathbb{R}^l \to \Lambda$.
Constants for rate conditions (Definition \ref{dfn-rate-map}) are then considered for both original $C^{k+1}$ map $f : N\to \mathbb{R}^{n_u}\times \mathbb{R}^{n_s}\times \Lambda$ and that defined on a set in the same good chart given by
\begin{equation*}
P(q) = \{z\in N\mid \|\pi_y z- \pi_y q\|\leq R_\Lambda / 2 \},
\end{equation*}
where $R_\Lambda > 0$ is such that $\phi\mid_{{\bf B}_l(y,R_\Lambda)}$ is homeomorphic onto its image for each $y\in \Lambda$. 
A difference arises in the definition of constants $\xi_{u,i}$ and $\xi_{u,i,P}$, $i=1,2$, in \cite{CZ2015}.
Note that constants $\xi_{u,i}$ in the current definition corresponds to $\xi_{u,i,P}$ in \cite{CZ2015}\footnote
{
Our stated conditions gives stronger ones than \cite{CZ2015}, since $\xi_{u,i,P} \leq \xi_{u,i}$ holds in general.
}.
In our current setting, the set $\Lambda$ is assumed to be an $h$-set $Y\subset \mathbb{R}^l$, which leads to simpler treatments of charts.
\par
On the other hand, we have to care about treatments of isolating blocks and center-(un)stable manifolds of invariant manifolds for flows when we apply a series of arguments with $\Lambda = Y\subset \mathbb{R}^l$ being an $h$-set or compact manifold with boundary.
In \cite{CZ2016}, isolating blocks of the form $D = \overline{{\bf B}_{n_u}}\times \overline{{\bf B}_{n_s}}\times \Lambda$ with $\Lambda$ being a torus do not assume transversal intersections between $\overline{{\bf B}_{n_u}}\times \overline{{\bf B}_{n_s}}\times \partial \Lambda$ and flow, in which case there is no problem since $\partial \Lambda = \emptyset$.
If we apply the same arguments as \cite{CZ2015, CZ2016} with replacements of $\Lambda$ by an $h$-set $Y$, however, we need isolation arguments in center variables.
The same kind of problems appear in treatments of center manifolds, since center manifolds as graphs defined on an $h$-set $Y$ are neither positively nor negatively invariant in general (compare with Definition \ref{dfn-cu-mfd-CZ}).
We then modify the original vector field on an extended $h$-set so that the same arguments as \cite{CZ2015, CZ2016} can be applied to the modified vector field on the extended $h$-set. 
The key requirement is isolation of the extended $h$-set in the center (namely, slow) direction with respect to modified flow, which essentially concerns with the existence of center-(un)stable manifolds.
In fact, such a modification for fast-slow systems already appears in \cite{Jones, Mat2}.
In the next section, we state the concrete modification and complete arguments about smoothness of slow manifolds.
\end{rem}

\section{Validating the existence and smoothness of slow manifolds}
\label{section-inv-mfd}
Here we review a verification theorem of slow manifolds as well as their stable and unstable manifolds stated in \cite{Mat2}, which provides sufficient conditions to validate not only the critical manifold $S_0$ but also the perturbed slow manifold $S_\epsilon$ of (\ref{fast-slow})$_\epsilon$ {\em for all $\epsilon \in (0,\epsilon_0]$ in given regions}. 
We also add the smoothness arguments of $S_\epsilon$, which is an application of arguments in Section \ref{section-preliminary}.
\par
Recall that Fenichel's results, which are ones of the origin of geometric singular perturbation theory (Remark \ref{rem-Fenichel} below), assume normal hyperbolicity and graph representation of the critical manifold $S_0$ for (\ref{layer}).
These assumptions are nontrivial, but very essential to prove the persistence. 
Our verification theorem contains verification of both normal hyperbolicity and graph representation of $S_0$.

The main idea is based on discussions in \cite{Jones}. 
For technical reasons, we use a multiple of $\epsilon$ as the new auxiliary variable. 
We set $\epsilon = \eta \sigma$ and $\sigma:= \epsilon_0 > 0$, where $\epsilon_0$ is a given positive number. 
We add the equation $\eta' = 0$ to (\ref{fast-slow})$_\epsilon$. 
Furthermore, we consider the following system instead of (\ref{fast-slow})$_\epsilon$ for simplicity: 
\begin{equation}
\label{abstract-form}
\begin{cases}
a' = A(y)a + F_1(x,y,\epsilon) & \\
b' = B(y)b + F_2(x,y,\epsilon) & \\
y' = \epsilon g(x,y,\epsilon) & \\
\eta' = 0 &
\end{cases}.
\end{equation}
Here $A(y)$ denotes the $u\times u$ matrix which all eigenvalues have positive real part 
and $B(y)$ denotes the $s\times s$ matrix which all eigenvalues have negative real part\footnote
{
In \cite{Mat2}, matrices $A$ and $B$ are assumed to be (locally) constant.
For practical validation of slow manifolds in this section, we consider (\ref{abstract-form}) with locally constant matrices $A$ and $B$ around numerical equilibria of $\dot x = f(x,y,0)$.
In Section \ref{section-eigenpairs} we consider $y$-dependent matrices $A(y)$ and $B(y)$.
}.
This formulation is natural when the construction of fast-saddle-type blocks stated in Section \ref{section-preliminary} is taken into account.

Let $N$ be a fast-saddle type block for (\ref{fast-slow}). 
Section \ref{section-block} implies that the coordinate representation, $N_c$, is given by (\ref{fast-block}) (or (\ref{fast-block-2})), which is directly obtained from the system (\ref{abstract-form}). 
A fast-saddle-type block $N$ has the form (\ref{fast-block}), which has $a$-coordinate, $b$-coordinate and $y$-coordinate following (\ref{abstract-form}). 
With this in mind, we put  following notations.

\begin{notation}\rm
Let $\pi_a$, $\pi_b$, $\pi_y$, $\pi_{a,b}$, $\pi_{a,y}$ and $\pi_{b,y}$ be the projection onto the $a$-, $b$-, $y$-, $(a,b)$-, $(a,y)$- and $(b,y)$-coordinate in $N$, respectively. If no confusion arises, we drop the phrase \lq\lq in $N$" in their notations.

We identify nonlinear terms $F_1(x,y,\epsilon)$, $F_2(x,y,\epsilon)$ and $g(x,y,\epsilon)$ with $F_1(a,b,y,\epsilon)$, $F_2(a,b,y,\epsilon)$ and $g(a,b,y,\epsilon)$, respectively, via an affine transform $x(\in \mathbb{R}^n)\mapsto (a,b)\in \mathbb{R}^{n_u+n_s}$.

For a squared matrix $A(y)$ with ${\rm Spec}(A(y)) \subset \{\lambda \in \mathbb{C}\mid \re \lambda > 0\}$, $\lambda_A > 0$ denotes a positive number such that
\begin{equation}
\label{bound-unst-ev}
\lambda_A < \re \lambda,\quad \forall \lambda \in {\rm Spec}(A(y)).
\end{equation}
Similarly, for a squared matrix $B(y)$ with ${\rm Spec}(B(y)) \subset \{\lambda \in \mathbb{C}\mid \re \lambda < 0\}$, $\lambda_B < 0$ denotes a negative number such that
\begin{equation}
\label{bound-st-ev}
\lambda_B > \re \lambda,\quad \forall \lambda \in {\rm Spec}(B(y)).
\end{equation}

Finally, let ${\rm dist}(\cdot, \cdot)$ be the distance between compact sets $N_1, N_2\subset \mathbb{R}^{n+l}$ given by ${\rm dist}(N_1,N_2) = \inf_{z_1\in N_1, z_2\in N_2}|z_1-z_2|$.
\end{notation}

The basic concept for verifying the existence and smoothness of slow manifolds is {\em rate conditions} discussed by Fenichel \cite{F1973, F1977} for the modified vector field of (\ref{abstract-form}), which compares the expanding and decay rates of variational trajectories along invariant manifolds. 
The (generalized) {\em Lyapunov-type numbers} are considered there and these numbers estimate the smoothness of (normally hyperbolic) invariant manifolds as well as their invariant foliations.
Here we apply arguments based on \cite{CZ2015, CZ2016}, which are reviewed in Section \ref{section-preliminary} and are in the same spirit as Fenichel's arguments, to slow manifolds for fast-slow systems.
The approximate diagonal system (\ref{abstract-form}) relates to (\ref{flow-param}) in the following correspondence.
The variable $z$ in (\ref{flow-param}) corresponds to $Z = (z,\eta)^T \equiv ((a,b,y)^T,\eta)^T$, where $a$ is the (fast-)unstable variable, $b$ is the (fast-)stable variable and $y$ is the slow variable\footnote
{
In fast-slow systems, the multiple time scale parameter $\epsilon$ can be considered as a component of the center variable.
The center variable $y$ in Sections \ref{section-rate} and \ref{section-summary-CZ} therefore corresponds to the pair $(y,\epsilon)$, or $(y,\eta)$, of the slow variable and the multiple time scale parameter in this section.
\par
In (\ref{flow4fss-rate}), we do not consider other parameter dependence of systems explicitly.
In our setting, several components of slow variable $y$ can be considered as parameters.
In such a case, parameters $\tilde y_{i_1}=\epsilon, \tilde y_{i_2}, \cdots, \tilde y_{i_p}$ can be considered to evolve following the trivial vector field $\frac{d}{dt}\tilde y_{i_j} = 0$.
Whenever the system (\ref{fast-slow}) contains parameters, the above treatment enables us to apply arguments in Sections \ref{section-preliminary}, \ref{section-inv-mfd} and later.
}, in which case (\ref{abstract-form}) has the form $Z' = F(Z)$, where
\begin{equation}
\label{flow4fss-rate}
F(Z) = (F_a(Z), F_b(Z), F_y(Z), F_\eta(Z))^T,\quad 
\begin{pmatrix}
F_a(Z)\\
F_b(Z)\\
F_y(Z)\\
F_\eta(Z)
\end{pmatrix}
=
\begin{pmatrix}
A(y)a + F_1(a,b,y,\epsilon) \\
B(y)b + F_2(a,b,y,\epsilon) \\
\epsilon g(a,b,y,\epsilon)\\
0
\end{pmatrix}.
\end{equation}

\begin{dfn}[Rate condition for fast-slow systems]\rm
\label{dfn-rate-fast-slow}
Define the {\em (local) rates for fast-slow system (\ref{fast-slow})} as those in Definition \ref{dfn-rate-flow} for $F$ in (\ref{flow4fss-rate}) with the coordinate $(a,b,y)$.
We say that (\ref{fast-slow}) satisfies the {\em (local) rate condition of order $k$} (in a $ch$-set $D=N\times [0,\epsilon_0]\subset \mathbb{R}^{n_u+n_s+l+1}$) if $F$ satisfies the rate condition of order $k$ in terms of local rates in Definition \ref{dfn-rate-flow}.
\end{dfn}

\par

\begin{rem}\rm
As in \cite{CZ2015}, inequalities (\ref{rate-fss-1}) and (\ref{rate-fss-2}) show the existence of $W^s(S_\epsilon)$ and $W^u(S_\epsilon)$.
Inequalities (\ref{rate-fss-3}) describe $C^j$-smoothness of $W^s(S_\epsilon)$ and $W^u(S_\epsilon)$.
Inequalities (\ref{rate-fss-1}) and (\ref{rate-fss-4}) yield 
\begin{equation*}
(j+1)\overrightarrow{\mu_{s,1}} < \overrightarrow{\xi_{su,2}},\quad \overrightarrow{\mu_{ss,2}} < (j+1) \overrightarrow{\xi_{u,1}},
\end{equation*}
showing the existence of (local) invariant foliations and $C^j$-smoothness fibers of $W^s(S_\epsilon)$ and $W^u(S_\epsilon)$.
\end{rem}
Unless otherwise noted, the terminology \lq\lq rate condition" always means that in Definition \ref{dfn-rate-fast-slow} in the context of fast-slow systems.

\bigskip
We would like to apply a series of results in Section \ref{section-CZ} to $F$ in a fast-saddle-type block $N$ and $\epsilon \in [0,\epsilon_0]$.
However, we cannot directly apply these results to $N$ because $N$ is not actually an isolating block.
More precisely, the flow $\varphi_\epsilon$ does not always intersect the boundary in the slow direction, $N_c^{slow} = \overline{{\bf B}_{n_u}}\times \overline{{\bf B}_{n_s}}\times \partial {\bf B}_l$, transversely.
Nevertheless, we can prove the smoothness of $W^\alpha(S_\epsilon),\ (\alpha = u,s)$ under careful treatments of discussions in \cite{CZ2015}.
Before stating our result in this section, we put the following assumption, which covers direct applications with rigorous numerics.
\begin{ass}
The $y$-component of fast-saddle-type blocks, $Y_N\equiv \pi_y N$, is an interval set, namely, 
\begin{equation}
\label{Y_N}
Y_N = \prod_{i=1}^l [y_i^-, y_i^+],
\end{equation}
which obviously keeps the structure of $N$ as an $h$-set.
Let $m_N(y_i) := \frac{1}{2}(y_i^-, y_i^+)$ be the middle point of the interval $[y_i^-,y_i^+]$
\end{ass}

We then have the following result.
\begin{thm}
\label{thm-smooth-slow-mfd}
Consider (\ref{abstract-form}), where $F_1, F_2, g$ are $C^{k+1}$ for all variables including $\epsilon$.
Let $N = \overline{{\bf B}_{n_u}} \times \overline{{\bf B}_{n_s}} \times Y_N\subset \mathbb{R}^{n_u+n_s+l}$ be a fast-saddle-type block constructed in Section \ref{section-block} being of the form (\ref{Y_N}).
Assume that the vector field $F$ satisfies the rate condition of order $k$ in $N\times [0,\epsilon_0]$.
Then, for each $\epsilon\in [0,\epsilon_0]$, $N$ contains locally invariant manifolds $S_\epsilon$, $W^\alpha(S_\epsilon),\ \alpha = u,s$, which form
\begin{align*}
S_\epsilon &= W^u(S_\epsilon)\cap W^s(S_\epsilon)\cap N,\\
\{W^u(S_\epsilon)\}_{\epsilon\in [0,\epsilon_0]} &= \{(a,b,y,\epsilon)\in N\times [0,\epsilon_0] \mid b = h^u(a,y,\epsilon), a\in \overline{{\bf B}_{n_u}}, y\in Y_N\},\\
\{W^s(S_\epsilon)\}_{\epsilon\in [0,\epsilon_0]} &= \{(a,b,y,\epsilon)\in N\times [0,\epsilon_0] \mid a = h^s(b,y,\epsilon), a\in \overline{{\bf B}_{n_s}}, y\in Y_N\}.
\end{align*}
Moreover, all these manifolds are $C^k$.
That is, functions $h^u$ and $h^s$ determining $W^u$ and $W^s$ are $C^k$ functions.
In particular, for any $\epsilon \in [0,\epsilon_0]$, the validated slow manifold has a graph representation
\begin{equation*}
S_\epsilon \cap N = \{(x,y)\in Y\mid x = h_\epsilon(y),\ y\in Y_N\}
\end{equation*}
such that the graph $h_\epsilon$ is $C^k$ on $Y_N$.
\end{thm}

\begin{proof}
We only prove of the existence and smoothness for $W^s(S_\epsilon)$.
The existence and smoothness of $W^u(S_\epsilon)$ follows from similar arguments.
The smoothness of $S_\epsilon$ follows from the fact that $S_\epsilon$ is the intersection of $C^k$-manifolds $W^u(S_\epsilon)$ and $W^s(S_\epsilon)$, and that the intersection is transversal.
\par
\begin{description}
\item[Step 1.] Modified vector fields and isolation of blocks.
\end{description}

First of all, we slightly extend $N$ to $\tilde N$, where
\begin{equation*}
\tilde N = \overline{{\bf B}_{n_u}}\times \overline{{\bf B}_{n_s}}\times \prod_{i=1}^l [\tilde y_i^-, \tilde y_i^+]\equiv \overline{{\bf B}_{n_u}}\times \overline{{\bf B}_{n_s}}\times Y_{\tilde N},
\end{equation*}
and $(\tilde y_i^-, \tilde y_i^+) \equiv \Int [\tilde y_i^-, \tilde y_i^+] \supset [y_i^-, y_i^+]$ for all $i=1,\cdots, l$.
Taking $Y_{\tilde N}$ smaller if necessary, we may assume that $|\tilde y_i^+ - m_N(y_i)| = |\tilde y_i^- - m_N(y_i)|$ for all $i \in \{1,\cdots, n\}$.
\par
Next, we modify the vector field (\ref{abstract-form}) of the form
\begin{equation}
\label{abstract-form-modify}
\begin{cases}
a' = A(y)a + F_1(x,y,\epsilon) & \\
b' = B(y)b + F_2(x,y,\epsilon) & \\
y' = \epsilon g(x,y,\epsilon) + \delta \rho(y)n_y & \\
\eta' = 0 &
\end{cases}
\end{equation}
so that the vector field is inflowing invariant with respect to the slow-boundary $\tilde N^{slow} \equiv \overline{{\bf B}_{n_u}}\times \overline{{\bf B}_{n_s}}\times \partial Y_{\tilde N}$.
In other words, $\tilde N^{slow}$ should be a subset of the entrance of $\tilde N$ for (\ref{abstract-form-modify})
\footnote
{
In the case of $W^u(S_\epsilon)$, $\tilde N^{slow}$ should be a subset of the exit of $\tilde N$.
Namely, the vector field should be overflowing invariant with respect to $\tilde N^{slow}$.
}.
\par
More precisely, let $\rho_i : \mathbb{R}\to \mathbb{R}$ be a $C^\infty$ function satisfying
\begin{equation*}
\rho_i(y_i) = \begin{cases}
0 & \text{$y_i \leq y_i^+ - m_N(y_i)$}\\
1 & \text{$y_i \geq \tilde y_i^+ - m_N(y_i)$}
\end{cases},\quad \rho_i'(y_i) \geq 0\text{ for all }y_i\in \mathbb{R}.
\end{equation*}
We then construct a function $\tilde \rho_i : \mathbb{R}\to \mathbb{R}$ as follows:
\begin{equation*}
\tilde \rho_i (y_i) = 
\begin{cases}
\rho_i(y_i-m_N(y_i)), & \text{$y_i \geq m_N(y_i)$,} \\
\rho_i(-y_i+m_N(y_i)), & \text{$y_i \leq m_N(y_i)$.}
\end{cases}
\end{equation*}
Finally, set 
\begin{equation*}
\rho(y) := (-\sgn (y_1-m_N(y_1)) \tilde \rho_1 (y_1),\cdots, -\sgn (y_l-m_N(y_l)) \tilde \rho_l (y_l) )^T,\quad y\in \mathbb{R}^l.
\end{equation*}
Note that the function $\tilde \rho$ is $C^\infty$ on $\mathbb{R}$ and that $n_y = \pm e_i$ (the standard $i$-th unit vector) for $y$ with $y_i = \tilde y_i^\pm$, which indicate that the slow boundary $\tilde N^{slow}$ is indeed inflowing invariant for sufficiently large $\delta > 0$ and hence $\tilde N$ is an isolating block for the modified system (\ref{abstract-form-modify})\footnote{
In the case of $W^u(S_\epsilon)$, replace $\rho$ by $(\sgn (y_1-m_N(y_1))\tilde \rho_1 (y_1),\cdots, \sgn (y_l-m_N(y_l))\tilde \rho_l (y_l) )^T$.
}.

Arguments discussed in \cite{Mat2} (cf. \cite{Jones}) indicate that, if necessary choosing $\tilde N$ being sufficiently small 
so that $N\subset \tilde N$, the new compact set $\tilde N$ is an isolating block for the modified vector field (\ref{abstract-form-modify}) with the exit $\tilde N^-$ and the entrance $\tilde N^+$ given as follows\footnote{
In the case of $W^u(S_\epsilon)$, the exit $\tilde N^-$ is $\partial {\bf B}_{n_u} \times \overline{{\bf B}_{n_s}} \times Y_{\tilde N} \cup \overline{{\bf B}_{n_u}} \times \overline{{\bf B}_{n_s}}\times \partial Y_{\tilde N}$,  and the entrance $\tilde N^+$ is $\overline{{\bf B}_{n_u}} \times \partial {\bf B}_{n_s}\times Y_{\tilde N}$.
}:
\begin{align*}
\tilde N^- &=  \partial {\bf B}_{n_u} \times \overline{{\bf B}_{n_s}}\times Y_{\tilde N},\\
\tilde N^+ &= \overline{{\bf B}_{n_u}} \times \partial {\bf B}_{n_s}\times Y_{\tilde N} \cup \overline{{\bf B}_{n_u}} \times \overline{{\bf B}_{n_s}}\times \partial Y_{\tilde N}.
\end{align*}
In particular, the modified vector field (\ref{abstract-form-modify}) satisfies the covering condition in $\tilde N$ stated in Proposition \ref{prop-isol-cov}.

\begin{description}
\item[Step 2.] Rate conditions in $N\times [0,\epsilon_0]$ and $\tilde N\times [0,\epsilon_0]$.
\end{description}
Next, we discuss the relationship of local rates between on $N\times [0,\epsilon_0]$ and on $\tilde N\times [0,\epsilon_0]$.
The Jacobian matrix of the modified vector field (\ref{abstract-form-modify}) at $\tilde Z = (\tilde z, \tilde \eta) = ((\tilde a,\tilde b,\tilde y)^T,\tilde \eta)^T$ is 
\begin{equation*}
J(\tilde Z) = \begin{pmatrix}
A(y) + (\partial F_1 / \partial a)(Z) & (\partial F_1 / \partial b)(Z) & (\partial F_1 / \partial y)(Z) & (\partial F_1 / \partial \eta)(Z)\\
(\partial F_2 / \partial a)(Z) & B(y) + (\partial F_2/\partial b)(Z) & (\partial F_2 / \partial y)(Z)  & (\partial F_2 / \partial \eta)(Z) \\
\epsilon (\partial g / \partial a)(Z) & \epsilon (\partial g / \partial b)(Z) & \epsilon (\partial g / \partial y)(Z) + \delta J_y \rho(y) & \sigma g(Z) + \epsilon (\partial g / \partial \eta)(Z)\\
O & O & O & O
\end{pmatrix}_{Z = \tilde Z},
\end{equation*}
where 
\begin{equation*}
J_y \rho(y) = \diag \left( - \frac{d\rho_1}{dy_1}(y_1),\cdots, - \frac{d\rho_l}{dy_l}(y_l)\right).
\end{equation*}
By our construction of $\rho$, the matrix
\begin{equation*}
J^\rho (Z)\equiv
\begin{pmatrix}
O & O & O & O\\
O & O & O & O\\
O & O & \delta J_y \rho(y) & O\\
O & O & O & O
\end{pmatrix}
\end{equation*}
is negative semidefinite for all $z\in \tilde N$.
Applying Lemma \ref{lem-monotome-log} to $A = J(Z)$ and $B = -J^\rho(Z)$, we know that the rate condition with respect to the modified vector field (\ref{abstract-form-modify}) in $\tilde N\times [0,\epsilon_0]$ is always satisfied under the rate condition of $F$ in $N\times [0,\epsilon_0]$.
Indeed, the local rates which are changed with the replacement of $N$ by $\tilde N$ in Definition \ref{dfn-rate-fast-slow} are $\overrightarrow{\mu_{ss,1}}$ and $\overrightarrow{\mu_{su,2}}$\footnote
{
Actually, constants $\overrightarrow{\mu_{su,1}}$ and $\overrightarrow{\mu_{ss,2}}$ also have effects on modifications to (\ref{abstract-form-modify}). However, these constants are only used in discussions about $W^u(S_\epsilon)$.
}.
Since the original vector field $F$ is $C^{k+1}$, then the inequalities (\ref{rate-fss-1}) - (\ref{rate-fss-4}) replacing $N$ by $\tilde N$ still hold true by choosing $\tilde N$ sufficiently close to $N$.
We shall write corresponding constants as $\overrightarrow{\mu_{ss,1; \tilde N}}$, and so on.
Then, with the help of Lemma \ref{lem-monotome-log}, we have 
\begin{align*}
\overrightarrow{\mu_{ss,1;\tilde N}} &=  \sup_{Z\in \tilde N\times [0,\epsilon_0]} \left\{ l\left(\frac{\partial F_{(b,y)}}{\partial (b,y)}(Z) \right) + M\left\| \frac{\partial F_{(b,y)}}{\partial a}(Z)\right\|\right\}\\ 
	&\geq  \sup_{Z\in \tilde N\times [0,\epsilon_0]} \left\{ l\left(\frac{\partial F_{(b,y)}}{\partial (b,y)}(Z) + J^\rho (Z) \right) + M \left\| \frac{\partial F_{(b,y)}}{\partial a}(Z)\right\|\right\},\\
\overrightarrow{\mu_{su,2;\tilde N}} &= \inf_{Z\in \tilde N\times [0,\epsilon_0]} \left\{ m_l\left( \frac{\partial F_{(a,y)}}{\partial (a,y)}(Z) \right) - \frac{1}{M}\left\| \frac{\partial F_b}{\partial (a,y)}(Z) \right\| \right\}\\
	&\leq \inf_{Z\in \tilde N\times [0,\epsilon_0]} \left\{ m_l\left( \frac{\partial F_{(a,y)}}{\partial (a,y)}(Z) + J^\rho (Z) \right) - \frac{1}{M}\left\| \frac{\partial F_b}{\partial (a,y)}(Z) \right\| \right\}.
\end{align*}
The rightmost constants in the above inequalities are corresponding ones for (\ref{abstract-form-modify}) in $\tilde N$.
Note that these inequalities are all necessities for existence and smoothness of $W^s(S_\epsilon)$.
As a consequence, the rate condition in $N$ yields the rate condition in $\tilde N$, provided $\tilde N$ is chosen sufficiently small so that $N\subset \tilde N$. 

\begin{description}
\item[Step 3.] Cone conditions in $\tilde N\times [0,\epsilon_0]$.
\end{description}

Finally, we discuss the existence of cones in $\tilde N\times [0,\epsilon_0]$.
To this end, we consider a series of arguments about cones in $N\times [0,\epsilon_0]$.

\begin{prop}[Unstable $M$-cone, cf. \cite{Mat2}]
\label{prop-unst-m-cone}
Consider (\ref{abstract-form}). 
Let $N\subset \mathbb{R}^{n+l}$ be a $ch$-set and fix $M > 1$. 
Assume that the following inequalities hold in $N\times [0,\epsilon_0]$: 
\begin{align}
\label{ineq-m-graph-unstable}
&\overrightarrow{\xi_{u,1}} > 0,\\
\label{ineq-m-cone-unstable}
&\overrightarrow{\xi_{u,1}} - \overrightarrow{\xi_{ss,1}} > 0.
\end{align}
Then, letting a function $Q^u_M(t):= |\Delta a(t)|^2 - M^2|\Delta \zeta(t)|^2$, ${Q^u_M}'(t) > 0$ holds for all points in $N\times [0,\epsilon_0]$ with $Q^u_M(t) \geq 0$, where $\zeta = (b,y,\eta)$.
\end{prop}

Before going back to the proof in Step 3, we derive the analogous statements for dynamics in stable direction.

\begin{prop}[Stable $M$-cone]
\label{prop-st-m-cone}
Consider (\ref{abstract-form}). Let $N\in \mathbb{R}^{n+l}$ be a $ch$-set and fix $M > 1$.  Assume that the following inequalities hold in $N\times [0,\epsilon_0]$:
\begin{align}
\label{ineq-m-graph-stable}
&0 > \overrightarrow{\xi_{s,1}},\\
\label{ineq-m-cone-stable}
&\overrightarrow{\xi_{su,1}} - \overrightarrow{\xi_{s,1}} > 0.
\end{align}
Then, defining a function $Q^s_M(t):= |\Delta b(t)|^2 - M^2|\Delta \nu(t)|^2$, ${Q^s_M}'(t) > 0$ with the time-reversal flow holds for all points on $N\times [0,\epsilon_0]$ with $Q^s_M(t) \geq 0$, where $\nu = (a,y,\eta)$.
\end{prop}

We go back to the proof in Step 3.
Arguments in Step 2 show that inequalities (\ref{ineq-m-graph-unstable}), (\ref{ineq-m-cone-unstable}), (\ref{ineq-m-graph-stable}) and (\ref{ineq-m-cone-stable}), estimated in $N\times [0,\epsilon_0]$, lead to the same inequalities in $\tilde N\times [0,\epsilon_0]$.

\begin{description}
\item[Step 4.] Final arguments.
\end{description}
Notice that inequalities (\ref{ineq-m-graph-unstable}), (\ref{ineq-m-cone-unstable}), (\ref{ineq-m-graph-stable}) and (\ref{ineq-m-cone-stable}) are contained in the rate condition (or order $0$) in $N\times [0,\epsilon_0]$.
As a consequence, under the rate condition of order $k$ in $N\times [0,\epsilon_0]$, we have the following statements:
\begin{description}
\item[(From Step 1)] the covering condition on $\tilde N\times [0,\epsilon_0]$ for the time-$h$ map $\Phi_h$ generated by the flow of (\ref{abstract-form-modify}).
\item[(From Step 2)] the local rate condition in $\tilde N\times [0,\epsilon_0]$ for (\ref{abstract-form-modify}) and, by Proposition \ref{prop-rate-flow-map}, for the time-$h$ map $\Phi_h$ generated by the flow of (\ref{abstract-form-modify}).
\item[(From Step 3)] Cone conditions in $\tilde N\times [0,\epsilon_0]$. An alternative condition for the time-$h$ map $\Phi_h$; namely, the backward cone condition (Definition \ref{dfn-back-cone}), follows from Lemma \ref{lem-cone-flow-map} below. 
\end{description}

A series of arguments in Section \ref{section-CZ} to (\ref{abstract-form-modify}) in $\tilde N\times [0,\epsilon_0]$ show that the positively invariant manifold $\tilde{W^s} \subset \tilde N\times [0,\epsilon_0]$ is $C^k$.
Since the restriction of $\tilde{W^s}$ to $N\times \{\epsilon\}$ is exactly $W^s(S_\epsilon)$, then we know that $W^s(S_\epsilon)$ is also $C^k$ and the proof is completed.
\end{proof}


\begin{dfn}[$M$-cone conditions]\rm
\label{dfn-m-cone}
We shall call inequalities (\ref{ineq-m-graph-unstable}) and (\ref{ineq-m-cone-unstable}) the {\em unstable $M$-cone condition in $N$}.
Similarly, we shall call inequalities (\ref{ineq-m-graph-stable}) and (\ref{ineq-m-cone-stable}) the {\em stable $M$-cone condition in $N$}. 
When these conditions are satisfied, {\em the unstable $M$-cone and the stable $M$-cone with the vertex $Z_0 = (a_0, b_0, y_0, \eta_0)$} (in the $(a,b,y)$-coordinate) are given as follows, respectively:
\begin{align*}
C_M^u(Z_0) := \{(a,b,y,\eta)\mid |a - a_0|^2 \geq  M^2(|b - b_0|^2 + |y - y_0|^2 + |\eta - \eta_0|^2)\},\\
C_M^s(Z_0) := \{(a,b,y,\eta)\mid |b - b_0|^2 \geq  M^2(|a - a_0|^2 + |y - y_0|^2 + |\eta - \eta_0|^2)\}. 
\end{align*}
\end{dfn}

Propositions \ref{prop-unst-m-cone} and \ref{prop-st-m-cone} show the transversality of flows on the boundary of cones as well as invariance.
In particular, these properties for stable $M$-cones yield the backward cone condition for time-$t$ maps with sufficiently small $t>0$.
\begin{lem}
\label{lem-cone-flow-map}
Let $N\subset \mathbb{R}^{n_u + n_s + l}$ be a $ch$-set.
Assume that the stable $M$-cone condition holds in $N\times [0,\epsilon_0]$.
Then, for sufficiently small $t > 0$, the time-$t$ map $\Phi_t$ satisfies the backward cone condition in $N\times [0,\epsilon_0]$ in the sense of 
Definition \ref{dfn-back-cone}.
\end{lem}

\begin{proof}
Let $Z_1, Z_2 \in N\times [0,\epsilon_0]$ be points such that $\Phi_{[0,h]}(Z_1), \Phi_{[0,h]}(Z_2) \subset N\times [0,\epsilon_0]$ hold for sufficiently small $h>0$.
Assume that $\Phi_h(Z_1)\in C^s_M(\Phi_h(Z_2))$. 
This implies $Q^s_M(t)= |\Delta b(t)|^2 - M^2|\Delta \nu(t)|^2 \geq 0$, where
\begin{align*}
\Delta b(t) &= \pi_b \Phi_{t+h} (Z_1) - \pi_b \Phi_{t+h} (Z_2),\\
\Delta \nu(t) &= \begin{pmatrix}
\pi_a \Phi_{t+h} (Z_1) - \pi_a \Phi_{t+h} (Z_2)\\
\pi_y \Phi_{t+h} (Z_1) - \pi_y \Phi_{t+h} (Z_2)
\end{pmatrix}.
\end{align*}
If $Q^s_M(0) > 0$, then $Q^s_M(-t) > 0$ still holds for sufficiently small $t > 0$.
The remainder is the case $Q^s_M(0) = 0$.
The stable cone condition indicates that $dQ^s_M / d(-t) > 0$ on the boundary $Q^s_M(0) = 0$.
Therefore, for any point $Z_1$ on $(N\times [0,\epsilon_0])\cap \{Q^s_M(0) = 0\}$, there is a positive number $t_{Z_1} > 0$ such that $Q^s_M(-t_{Z_1}) > 0$.
Since $(N\times [0,\epsilon_0])\cap \{Q^s_M(0) \geq 0\}$ is compact, then there is $t_0 \in (0,h)$ such that $Q^s_M(-t_0) > 0$ holds for any $Z_1, Z_2 \in N\times [0,\epsilon_0]$ satisfying $\Phi_h(Z_1)\in C^s_M(\Phi_h(Z_2))$.
Setting $F(z) := \Phi_{h-t_0}(z)$, we know that the above arguments yield the backward cone condition in the sense of Definition \ref{dfn-back-cone}.
\end{proof}

\par
\bigskip
Let $D$ and $\hat D$ be fast-saddle-type blocks given by (\ref{fast-block}) and (\ref{fast-block-2}), respectively.
Theorem \ref{thm-smooth-slow-mfd} says that, under cone conditions, the slow manifold $S_\epsilon$ is contained in the smaller block $D$. 
Obviously $S_\epsilon$ is also contained in $\hat D$ since $D\subset \hat D$. 
Moreover, $S_\epsilon$ is uniquely determined in $\hat D$.
If $d_0 \leq {\rm dist}(\partial (\pi_{a,b}\hat D), \pi_{a,b}D)$, then our observations imply that the distance between $S_\epsilon$ and $\partial \hat D$ in fast components is greater than $d_0$.
Summarizing these arguments, we have the following result, which is a simpler one than \cite{Mat2} and the key result for constructing tubular neighborhoods of $S_\epsilon$.

\begin{cor}[Fast-saddle-type blocks with spaces. cf. \cite{Mat2}]
\label{cor-block-space}
Consider (\ref{ode-ab-coord}). Let $D, \hat D\subset \mathbb{R}^{n+l}$ be fast-saddle-type blocks for (\ref{ode-ab-coord}) such that the coordinate representations $D_c$ and $\hat D_c$ are actually given by (\ref{fast-block}) and (\ref{fast-block-2}), respectively, for a given pair of positive numbers $\{\eta^u, \eta^s\}$. 
Assume that stable and unstable cone conditions hold in $\hat D_c$. Then the same statements as Theorem \ref{thm-smooth-slow-mfd} holds in $\hat D_c$. Moreover, the distance between $\partial \hat D_c$ and the validated slow manifold $S_\epsilon$ is estimated by
\begin{align*}
&{\rm dist}(\pi_a(\partial \hat D_c), \pi_a(W^s(S_\epsilon))) \geq \eta^u,\\
&{\rm dist}(\pi_b(\partial \hat D_c), \pi_b(W^u(S_\epsilon))) \geq \eta^s\quad \text{ for }\epsilon \in [0,\epsilon_0].
\end{align*}
\end{cor}

The main feature of Corollary \ref{cor-block-space} is that slow manifolds as well as their stable and unstable manifolds in Fenichel's theorems are validated in given blocks with an explicit range $\epsilon\in [0,\epsilon_0]$.
Our criteria can be explicitly validated with rigorous numerics, which can be seen in \cite{Mat2}.

\begin{rem}[Fenichel's original invariant manifold theorems]\rm
\label{rem-Fenichel}
A series of invariant manifold theorems in Fenchel's theory (e.g., \cite{F1979}) mainly consists of the following two parts.
Consider (\ref{abstract-form}) and let $S_0\subset \{f(x,y,0)=0\}$ be a subset. 
We assume
\begin{description}
\item[(F)] $S_0$ is given by the graph of the $C^\infty$ function $h_0(y)$ for $y\in Y$, where the set $Y$ is a compact, simply connected domain whose boundary is an $(l-1)$-dimensional $C^\infty$ submanifold. 
$S_0$ is normally hyperbolic.
Finally, under a suitable nonlinear transformation, $S_0$ is given by $\{(x,y,0)\mid x=h_0(y) = 0,y\in Y\}$.
\end{description}
Then, for sufficiently small $\epsilon > 0$ and small $\Delta > 0$, and an appropriate coordinate system $(a,b,y,\epsilon)$,
\begin{description}
\item[1.] there are smooth functions $h^s(b,y,\epsilon)$ and $h^u(a,y,\epsilon)$ such that
\begin{align*}
W_{loc}^s(S_\epsilon) &:= \{(a,b,y,\epsilon)\mid a=h^s(b,y,\epsilon), |a|\leq \Delta, |b|\leq \Delta, y\in Y\},\\
W_{loc}^u(S_\epsilon) &:= \{(a,b,y,\epsilon)\mid b=h^u(a,y,\epsilon), |a|\leq \Delta, |b|\leq \Delta, y\in Y\},\\
S_\epsilon &:= W_{loc}^s(S_\epsilon) \cap W_{loc}^u(S_\epsilon)
\end{align*}
are locally invariant (see \cite{Jones}).
\item[2.] $W_{loc}^s(S_\epsilon)$ and $W_{loc}^u(S_\epsilon)$ admit invariant foliations. Namely, for each $p\in S_\epsilon$, its stable and unstable manifolds $W_{loc}^s(p)$ and $W_{loc}^u(p)$, respectively, are locally invariant.
\end{description}
In particular, $W_{loc}^s(S_\epsilon)$ and $W_{loc}^u(S_\epsilon)$ admit fiber bundle structures over $S_\epsilon$.
These bundles are constructed in a tubular neighborhood $\{|a|\leq \Delta, |b|\leq \Delta, y\in Y\}$ of $S_0$.
The assumption $S_0 = \{x = h_0(y) = 0\}$ is often imposed in abstract settings.
Moreover, in the {\em Fenichel's normal form}\footnote
{
In (\ref{Fenichel-normal-form}), $\Lambda$ and $\Gamma$ are $C^k$ smooth, $n_u\times n_u$ and $n_s \times n_s$ dimensional matrix-valued functions, respectively, $h$ is a $C^k$ smooth function in $\mathbb{R}^l$, and ${\bf H}$ is a $C^k$ rank there tensor, for all $k < \infty$, with $\otimes$ denoting the tensor product.
In component-wise notation, the $y$ component of the vector field is expressed as $y_i' = h_i(y,\epsilon) + \sum_{u=1}^{n_u}\sum_{s=1}^{n_s} H_{ius} a_u b_s$.
}
 (e.g., \cite{JK})
\begin{align}
\notag
a' &= \Lambda(a,b,y,\epsilon)a\\
\label{Fenichel-normal-form}
b' &= \Gamma(a,b,y,\epsilon)b\\
\notag
y' &= h(y,\epsilon) + {\bf H}(a,b,y,\epsilon)\otimes a \otimes b,
\end{align}
the assumption $S_\epsilon = \{x = h_\epsilon(y) = 0\}$ is also imposed for sufficiently small $\epsilon > 0$.
In concrete systems, they are intrinsically two nontrivial problems; one is the concrete form for realizing such assumptions, and another is the choice of radius $\Delta$ of tubular neighborhoods so that they are fast-saddle type blocks.
Corollary \ref{cor-block-space} provides lower bounds of $\Delta$.
\end{rem}

\section{Validating continuous family of eigenpairs}
\label{section-eigenpairs}

As shown in Section \ref{section-preliminary}, fast-saddle-type blocks and cone conditions validate slow manifolds as the graphs of Lipschitzian functions.
This validation leads to considerations of computing eigenpairs of the linearized matrix $f_x(h_\epsilon(y),y,\epsilon)$ on the validated slow manifold $S_\epsilon = \{x=h_\epsilon(y)\}$.
As seen in construction of blocks, eigenpairs of $f_x(h_\epsilon(y),y,\epsilon)$ on $S_\epsilon$ are essential to obtain concrete transformation of coordinate systems from the original ones to the simpler ones like mentioned in Remark \ref{rem-Fenichel} along with $S_\epsilon$ itself.
\par
In this section we consider the eigenvalue problem
\begin{equation}
\label{ev-prob-param}
A(y)u(y) = \lambda(y)u(y)
\end{equation}
of continuous real matrix-valued functions $A(y)$.
In particular, we study this problem for validating the family of eigenpairs $(\lambda(y), u(y))$ depending continuously on $y$ with computer assistance.
\par
Assume that we have {\em approximate} eigenpairs $\{(\tilde \lambda_j, \tilde u_j)\}_{j=1}^n$ of $A$.
We want to validate {\em rigorous} eigenpairs of $A$ by means of $\{(\tilde \lambda_j, \tilde u_j)\}_{j=1}^n$ even in the case that $A$ depends continuously on parameters $y$.
We apply the Newton-like iteration method to (\ref{ev-prob-param}).
A benefit of such an iteration approach is that we can easily extend it to the interval Newton or the Krawczyk method (e.g., \cite{Tbook}), which enables us to validate a continuous parameter family of eigenpairs.
\par
First consider real eigenpairs.
Our aim is the family of eigenpairs $\{\lambda_i(y),u_i(y)\}_{i=1}^n$ which varies continuously, {\em not only locally but globally} on $y\in Y$.
We then consider the following formulations for applying problems to (not necessarily small) $Y$:
\begin{equation}
\label{ev-prob-real}
F_{\mathbb{R}}(u,\lambda;y) := \begin{pmatrix}
A(y)u - \lambda u\\
|u|^2 - 1
\end{pmatrix}
=\begin{pmatrix}
0\\
0
\end{pmatrix},\quad y\in Y,\quad \lambda\in \mathbb{R},\quad u\in \mathbb{R}^n,
\end{equation}
where $Y$ is a compact, contractible subset of $\mathbb{R}^l$.
Without the loss of generality, we assume that $Y$ is an interval set $[Y]$, namely, the interval hull of a set.
The $n\times n$ matrix-valued function $A(y)$ is assumed to be continuous with respect to $y\in Y$. 
\par
\bigskip
Next consider the problem involving complex eigenvalues.
An expected formulation for complex eigenpairs from real ones will be the following system:
\begin{equation*}
\tilde F_{\mathbb{C}}(u,\lambda;y) := \begin{pmatrix}
A(y)u - \lambda u\\
|u|^2 - 1
\end{pmatrix}
=\begin{pmatrix}
0\\
0
\end{pmatrix},\quad y\in Y,\quad \lambda\in \mathbb{C},\quad u\in \mathbb{C}^n,
\end{equation*}
which actually fails to provide unique determination of eigenvectors.
Indeed, given an eigenpair $(\lambda,u)\in \mathbb{C}^{1+n}$ with $|u|=1$, where $u$ is complex, then $(\lambda, e^{i\theta}u)\in \mathbb{C}^{1+n}$ is also an eigenpair and $|e^{i\theta}u| = 1$ for any $\theta \in [0,2\pi)$.
In other words, we need add a restriction of $\theta$ so that an eigenvector $u$ locates a certain direction. 
To this end, we add an equation $\im (u_1) = 0$ to the equation $\tilde F_\mathbb{C} = 0$; namely, the first component of $u$ should be real. 
The new equation is
\begin{equation*}
\hat F_{\mathbb{C}}(u,\lambda;y) := \begin{pmatrix}
A(y)u - \lambda u\\
|u|^2 - 1\\
\im u_1
\end{pmatrix}
=\begin{pmatrix}
0\\
0\\
0
\end{pmatrix},\quad y\in Y,\quad \lambda\in \mathbb{C},\quad u\in \mathbb{C}^n.
\end{equation*}
We further divide the original problem into the real part and the imaginary part, which is our system for complex eigenpairs:
\begin{equation}
\label{ev-prob-complex}
F_{\mathbb{C}}(u,\lambda;y) := \begin{pmatrix}
A(y)u^r - (\lambda^r u^r - \lambda^i u^i)\\
A(y)u^i - (\lambda^r u^i + \lambda^i u^r )\\
\sum_{j=1}^n ((u_j^r)^2 + (u_j^i)^2) - 1\\
u_1^i
\end{pmatrix}
= \begin{pmatrix}
0\\
0\\
0\\
0
\end{pmatrix},
\end{equation}
where $u = u^r + iu^i$ and $\lambda = \lambda^r + i\lambda^i$.
The linearized matrix of $F_\mathbb{C}$ at $(u^0,\lambda^0, y_0)$ is
\begin{equation}
\label{A_C}
\hat A_{\mathbb{C}} = 
\begin{pmatrix}
A - (\lambda^0)^r I_n & (\lambda^0)^i I_n &  -(u^0)^r & (u^0)^i\\
- (\lambda^0)^i I_n & A - (\lambda^0)^r I_n &  -(u^0)^i & -(u^0)^r\\
2u^i & 2u^r & 0 & 0\\
0 & e_1^T & 0 & 0\\
\end{pmatrix},
\end{equation}
where $e_j^T = (0,\cdots, 0,1,0,\cdots, 0)$ is the unit vector whose $j$-th component is the only nontrivial element.
In the present paper, we apply the Krawczyk iterations to validating $y$-continuous family of solutions.
\par
Let $y_0\in [Y]$ be fixed and $\{\lambda(y_0),u(y_0)\} \equiv \{\lambda^0,u^0\}$ be an (approximate) eigenpair.
First we focus on real eigenvalues and associated eigenvectors.
Define then the Krawczyk-type operator $K_{\mathbb{R}}:\mathbb{R}^{n+1}\to \mathbb{R}^{n+1}$ associated with (\ref{ev-prob-real}) as
\begin{align}
\notag
K_{\mathbb{R}}([u], [\lambda]) &= K_{\mathbb{R}}([u], [\lambda]; y_0, u^0, \lambda^0)\\
\label{Krawczyk-ev}
&:= 
\begin{pmatrix}
u^0 \\ \lambda^0
\end{pmatrix} - C_\mathbb{R}\begin{pmatrix}
(A(y_0)- \lambda^0 I) u^0 \\ 
|u^0|^2-1
\end{pmatrix} + (I-C_\mathbb{R}\cdot DF_{\mathbb{R}}([u],[\lambda] ;[Y])) \cdot \begin{pmatrix}
[u] - u^0 \\ [\lambda] - \lambda^0
\end{pmatrix},
\end{align}
where we use the following notations:
\begin{itemize}
\item $C_\mathbb{R}$ is a nonsingular matrix close to the inverse of
\begin{equation*}
\hat A_\mathbb{R} (y_0) = \begin{pmatrix}
A(y_0) - \lambda^0 I & -u^0\\
2(u^0)^T & 0
\end{pmatrix},
\end{equation*}
which $\hat A_\mathbb{R}(y_0)$ is expected to be nonsingular.
\item $[u]\subset \mathbb{R}^n$ is an interval set in the $u$-component containing $u^0$.
\item $[\lambda]\subset \mathbb{R}$ is an interval in the $\lambda$-component containing $\lambda^0$.
\end{itemize}
The Krawczyk-type operator $K_{\mathbb{C}}:\mathbb{R}^{2(n+1)}\to \mathbb{R}^{2(n+1)}$ associated with the complex eigenpair is defined in the similar manner to $K_{\mathbb{R}}$ as follows:
\begin{align}
\notag
K_{\mathbb{C}}([u],[\lambda]) &= K_{\mathbb{C}}([u],[\lambda] ; y_0,u^0, \lambda^0)\\
\label{Krawczyk-ev-complex}
&:= 
\begin{pmatrix}
(u^0)^r \\ (u^0)^i \\ (\lambda^0)^r \\ (\lambda^0)^i
\end{pmatrix} - C_{\mathbb{C}}
F_{\mathbb{C}}(u^0,\lambda^0,y;p)
 + (I-C_{\mathbb{C}}\cdot DF_{\mathbb{C}}([u],[\lambda]; [Y])) \cdot \begin{pmatrix}
[u^r] - (u^0)^r \\ 
[u^i] - (u^0)^i \\ 
[\lambda^r] - (\lambda^0)^r\\
[\lambda^i] - (\lambda^0)^i
\end{pmatrix},
\end{align}
where we use the following notations:
\begin{itemize}
\item $C_\mathbb{C}$ is a nonsingular matrix close to the inverse of $\hat A_{\mathbb{C}}$ in (\ref{A_C}), which is expected to be nonsingular.
\item $[u^r]\subset \mathbb{R}^n$ is an interval set in the $u^r$-component containing $(u^0)^r$.
\item $[u^i]\subset \mathbb{R}^n$ is an interval set in the $u^i$-component containing $(u^0)^i$.
\item $[\lambda^r]\subset \mathbb{R}$ is an interval in the $\lambda^r$-component containing $(\lambda^0)^r$.
\item $[\lambda^i]\subset \mathbb{R}$ is an interval in the $\lambda^i$-component containing $(\lambda^0)^i$.
\end{itemize}
If the matrix $C_\mathbb{K}$, $\mathbb{K} = \mathbb{R}$ or $\mathbb{C}$, is nonsingular, then $K_\mathbb{K}$ is well-defined.
Regularity of $C_\mathbb{K}$ corresponds to the simpleness of eigenvalues, which is shown in the following lemma.
\begin{lem}
\label{lem-nonsing}
For an eigenpair satisfying $F_{\mathbb{R}}(u,\lambda;y_0) = 0$, $\lambda\in \mathbb{R}$ is simple if and only if
\begin{equation*}
J_\mathbb{R}(u,\lambda) := \begin{pmatrix}
A-\lambda I_n & -u\\
2u^T & 0
\end{pmatrix}
\end{equation*}
is nonsingular, where $A=A(y_0)$.
\par
Similarly, for a complex eigenpair $(\lambda,u)\in \mathbb{C}^{n+1}$ satisfying $F_{\mathbb{C}}(u,\lambda;y_0) = 0$, $\lambda\in \mathbb{C}$ is simple if and only if
\begin{equation*}
J_{\mathbb{C}}(u,\lambda) := \begin{pmatrix}
A-\lambda^r I_n & \lambda^i I_n & -u^r & u^i\\
-\lambda^i I_n & A-\lambda^r I_n  & -u^r & u^i\\
2(u^r)^T & 2(u^i)^T & 0 & 0\\
0 & e_1^T & 0 & 0\\
\end{pmatrix}
\end{equation*}
is nonsingular.
\end{lem}
\begin{proof}
See Lemma 2 in \cite{Y1980} for $\lambda \in \mathbb{R}$ and $J_{\mathbb{R}}$.
The proof for $\lambda \in \mathbb{C}$ and $J_{\mathbb{C}}$ is similar, but we give a precise proof here.
\par
Assume first that $J=J_\mathbb{C}$ is singular. 
Then there is a nonzero vector $v = (z^r,z^i,\alpha^r,\alpha^i)^T$ such that $Jv = 0$, which implies
\begin{align*}
Az^r - (\lambda^r z^r - \lambda^i z^i) - (\alpha^r u^r - \alpha^i u^i) &= 0,\\
Az^i - (\lambda^i z^r + \lambda^r z^i) - (\alpha^r u^i + \alpha^i u^r) &= 0,
\end{align*}
i.e., $Az-\lambda z = \alpha u$, where $z = z^r+iz^i$ and $\alpha = \alpha^r + i\alpha^i$.
We also have $z^ru^r + z^iu^i =0$, which indicates ${\rm Re}\langle z,u\rangle = 0$ with the standard inner product $\langle \cdot, \cdot\rangle$ on $\mathbb{C}^n$.
Moreover, the vector $w = (-z^i,z^r,\alpha^i,-\alpha^r)^T$ is nonzero and leads the same conclusions as $v$ with ${\rm Im}\langle z,u\rangle = 0$ instead of ${\rm Re}\langle z,u\rangle = 0$.
There are two cases for these statements; the first is $\alpha = 0$ and the second is $\alpha \not = 0$.
The first case means $z\not = 0$, $Az=\lambda z$ and $\langle z,u\rangle = 0$, which implies that $z$ is another eigenvector of $\lambda$ independent of $u$.
The second case means $Az-\lambda z = \alpha u\not = 0$ and $(A- \lambda I_n)^2 z = \alpha (A- \lambda I_n) u = 0$ since $u$ is the eigenvector of $\lambda$.
Therefore, $\lambda$ is at least double eigenvalue and hence $\lambda$ is not simple.
\par
Conversely, suppose that $\lambda$ is not simple. 
Assume further that there is an eigenvector $\tilde z$ of $\lambda$, which can be chosen so that $\langle \tilde z,u\rangle = 0$.
In particular, $\tilde z\not = 0$ and $A\tilde z = \lambda \tilde z$ hold.
Now there is an element $\theta_1\in [0,2\pi)$ such that the first component of $e^{i\theta_1}\tilde z$ is real.
From the linearity of eigenvalue problems,  $e^{i\theta_1} \tilde z\not = 0$ and $A(e^{i\theta_1} \tilde z) = e^{i\theta_1}A \tilde z  = e^{i\theta_1} \lambda \tilde z = \lambda (e^{i\theta_1}\tilde z)$ still hold.
Thus the vector $\tilde v := ((e^{i\theta_1}\tilde z)^r, (e^{i\theta_1}\tilde z)^i, 0,0)^T$ satisfies $J\tilde v = 0$ and hence $J$ is singular.
If there is no second eigenvector, then there is a vector $\hat z\not = 0$ such that $(A-\lambda I_n)\hat z = u$ and $\langle \hat z,u\rangle = 0$ and that the first component of $e^{i\hat \theta_1}\hat z$ is real for some $\hat \theta_1\in [0,2\pi)$, since $\lambda$ is not simple.
Hence the vector $\hat v := ((e^{i\hat\theta_1}\hat z)^r, (e^{i\hat\theta_1}\hat z)^i, 1,0)^T$ is nonzero and satisfies $J\hat v = 0$, which implies that $J$ is singular.
\end{proof}

Given an initial region ${\bf X}_0 \equiv ([u]_0, [\lambda]_0)\subset \mathbb{R}^{n+1}$ or $([u^r]_0, [u^i]_0, [\lambda^r]_0, [\lambda^i]_0)\subset \mathbb{R}^{2(n+1)}$, define the sequence of interval sets
\begin{equation*}
{\bf X}_{k+1} := K({\bf X}_k,[Y])\cap {\bf X}_k,\quad k=0,1,2,\cdots,
\end{equation*}
where $K=K_{\mathbb{K}}$.
As a consequence of ordinary Krawczyk iterations, we obtain the following result.
\begin{prop}
\label{prop-Krawczyk}
Let $\mathbb{K} = \mathbb{R}$ or $\mathbb{C}$.
Assume that $K({\bf X},[Y]) = K_{\mathbb{K}}({\bf X},[Y])$ in (\ref{Krawczyk-ev}) or (\ref{Krawczyk-ev-complex}) is well-defined, namely, $K({\bf X},y)$ is well-defined for all $y\in [Y]$.
Further assume that one of components of $[u]_0$ in ${\bf X}_0$ does not contain $0\in \mathbb{R}$.
Then the following statements hold: for some $k\geq 0$,
\begin{enumerate}
\item if ${\bf X}_k$ contains a zero $x^\ast$ of $F_{\mathbb{K}}(\cdot, y)$, then so does $K({\bf X}_{k+1},y)\cap {\bf X}_k$;
\item for each $y\in Y$, if $K({\bf X}_k,y)\cap {\bf X}_k = \emptyset$, then ${\bf X}_k$ contains no zeros of $F_{\mathbb{K}}(\cdot, y)$;
\item If $K({\bf X}_k,[Y])\subset \Int ({\bf X}_k)$, then for all $y\in [Y]$, ${\bf X}_k$ contains exactly one zero of $F_{\mathbb{K}}$. 
Moreover, the family of zeros $\{x^\ast(y) = (\lambda^\ast(y),u^\ast(y))\}$ depends continuously on $y\in [Y]$.
As a consequence, we have a continuous family $\{\lambda^\ast(y),u^\ast(y)\}$ of an eigenpair of $A(y)$ on $[Y]$.
Suppose further that the validated eigenvalue $\lambda(y)$ is simple for all $y\in [Y]$. 
Then the associated eigenvector $u(y)$ is uniquely determined.
\end{enumerate}
\end{prop}
\begin{proof}
All the statements with fixed $y\in Y$ are just consequences of Krawczyk iterating method (e.g., Theorem 5.9 in \cite{Tbook}).
The rest is the proof of the third statement.
The assumption $K({\bf X},[Y])\subset \Int ({\bf X})$ implies that $K({\bf X},y)\subset \Int ({\bf X})$ holds for each $y\in Y$.
The ordinary Krawczyk iteration results thus holds true for all $y\in K$, namely, we have a family $\{x(y)\mid y\in Y\}\subset {\bf X}$ such that $x(y)$ is the unique zero of $F_{\mathbb{K}}(\cdot, y)$ for each $y\in Y$.
Zeros $\{x(y)\}$ of $F_{\mathbb{K}}$ corresponds to the eigenpair $(u(y),\lambda(y))$. The continuous dependence of $x(y)$ follows from the continuous dependence of eigenpairs of continuous matrix-valued functions on $y$. 
The second equation $|u(y)|^2 = 1$ restricts the norm of eigenfuctions, and hence $u(y)$ is uniquely determined up to signatures.
Now the signature of $u(y)$ is also uniquely determined, since one of components $[u]_0\ni u(y)$ does not contain $0$. 
This fact shows the uniqueness of $u(y)$ for all $y\in [Y]$.
\end{proof}

In the next section, we apply Proposition \ref{prop-Krawczyk} to $A(y) = f_x(h_\epsilon(y),y,\epsilon)$ with $y$ and $\epsilon$ as parameters, and obtaining a continuous family of eigenvectors at normally hyperbolic slow manifolds $S_\epsilon = \{(x,y,\epsilon)\mid x=h_\epsilon(y)\}$.
As an immediate consequence, we obtain vector bundles over $S_\epsilon$.
Using these bundles, we validate tubular neighborhoods of $S_\epsilon$ with explicit radii.

\section{Validations of tubular neighborhoods of slow manifolds}
\label{section-main}
We have already seen in Sections \ref{section-block} and \ref{section-inv-mfd} that fast-saddle-type blocks and cone conditions validate slow manifolds with their stable and unstable manifolds.
Note that fast-saddle-type blocks in practical validations are rectangular up to affine transformations and hence they are candidates of tubular neighborhoods.
The essence for constructing blocks is {\em the choice of an appropriate coordinate}.
This is locally realized as shown in Section \ref{section-block}, but it is a non-trivial question if such a coordinate, possibly depending on slow variables, can be chosen continuously and globally.
The problem relates to {\em the whole continuations of bases on fibers} regarding a compact domain of slow variables as a base space of vector bundles.
\par
Validations of eigenpairs of $f_x(h_0(y),y,0)$ along the critical manifold $S_0$ on an interval set $Y$ give bases of vector bundles over $S_0$ below. 
As Fenichel's theorems show, normally hyperbolic invariant manifold $S_0\subset \{f(x,y,0)=0\} = \{x=h_0(y)\}$ for (\ref{fast-slow})$_0$ admits an invariant foliation of its stable and unstable manifold:
\begin{equation*}
W^s(S_0) = \bigcup_{p\in S_0}W^s(p),\quad W^u(S_0) = \bigcup_{p\in S_0}W^u(p).
\end{equation*}
It follows that, for each $p\in S_0$, eigenvectors $\{v^s_1(p),\cdots, v^s_{n_s}(p)\}$ of $f_x(h_0(y),y,0)$ associated with eigenvalues with negative real part generate a basis of $T_pW^s(p)$ with $p=(h_0(y),y)$.
Similarly, eigenvectors $\{v^u_1(p),\cdots, v^u_{n_u}(p)\}$ of $f_x(h_0(y),y,0)$ associated with eigenvalues with positive real part generate a basis of $T_pW^u(p)$.
Since eigenvectors vary continuously on $p\in S_0$, the collections
\begin{equation*}
V^s_0 \equiv \bigcup_{p\in S_0}T_pW^s(p)\quad \text{ and }\quad V^u_0 \equiv \bigcup_{p\in S_0}T_pW^u(p)
\end{equation*}
become vector bundles over $S_0$.
In other words, families of eigenvectors $\{v^s_i(p)\mid i=1,\cdots, n_s, p\in S_0\}$ and $\{v^u_i(p)\mid i=1,\cdots, n_u, p\in S_0\}$ generate vector bundles over $S_0$.
These bundles can be actually constructed by arguments in Section \ref{section-eigenpairs} with rigorous numerics, as shown below.
Moreover, the validation procedure in Section \ref{section-eigenpairs} gives eigenvectors $\{v^s_{\epsilon,1}(p),\cdots, v^s_{\epsilon,n_s}(p)\}$ and $\{v^u_{\epsilon,1}(p),\cdots, v^u_{\epsilon,n_u}(p)\}$ of the linearized matrix $f_x(h_\epsilon(y),y,\epsilon)$ at $p=(h_\epsilon(y),y)\in S_\epsilon$ for $\epsilon > 0$.
These families of eigenvectors generate vector bundles
\begin{equation*}
V^s_\epsilon \equiv \bigcup_{p\in S_\epsilon}V^s_\epsilon(p)\quad \text{ and }\quad V^u_\epsilon \equiv \bigcup_{p\in S_\epsilon}V^u_\epsilon(p)
\end{equation*}
where $V^s_\epsilon(p) = {\rm span}\{v^s_{\epsilon,1}(p),\cdots, v^s_{\epsilon,n_s}(p)\}$ and $V^u_\epsilon(p) = {\rm span}\{v^u_{\epsilon,1}(p),\cdots, v^u_{\epsilon,n_u}(p)\}$.

\par
\bigskip
In this section, we firstly provide algorithms for constructing vector bundles $V^\alpha_\epsilon\ (\ast = s,u)$ over $S_\epsilon$ as well as validating $S_\epsilon$ itself.
Secondly, we apply such validated bundles and the procedure in Section \ref{section-block} to constructing enclosures of $S_\epsilon$ with explicit and uniform lower bounds of radius.
We then show that, for each $\epsilon \in [0,\epsilon_0]$, these enclosures contain tubular neighborhoods centered at $S_\epsilon$ with certain radii, which is our validation methodology of tubular neighborhoods of $S_\epsilon$.
The numerical validation of tubular neighborhoods  leads geometrically simple settings for considering dynamics around slow manifolds like the Exchange Lemma (e.g., \cite{JKK,JK}).
Thirdly, we extend the idea of tubular neighborhood validations to cones.
Note that these methodologies can be incorporated with rigorous numerics to various concrete fast-slow systems.
Also note that, if we add rate conditions in Definition \ref{dfn-rate-fast-slow} in our validations, we can validate various neighborhoods of {\em smooth} slow manifolds, which are usually considered as tubular neighborhoods of invariant manifolds (e.g., \cite{BLZ2000}).

\subsection{Vector bundles $V^\alpha_\epsilon$}
\label{section-bundle-validate}
First we validate vector bundles over slow manifolds.
Our procedure of $V^\alpha_\epsilon$ stated in the beginning of this section consists of the following validation processes:
\begin{itemize}
\item Validate slow manifolds $S_\epsilon$ with the graph representation $x = h_\epsilon(y)$ on $Y$;
\item Validate eigenvectors of $f_x(h_\epsilon(y),y,\epsilon)$ at each point on $S_\epsilon$.
\end{itemize}

In order to compute eigenvectors of $f_x(h_\epsilon(y),y,\epsilon)$, we have to know where the slow manifold $S_\epsilon$ is.
Fast-saddle-type blocks for all $y\in Y\subset \mathbb{R}^l$ and cone conditions validate neighborhoods of slow manifolds with graph representations on $Y$ as shown in Sections \ref{section-block} and \ref{section-inv-mfd}.
Eigenvectors of $f_x(h_\epsilon(y),y,\epsilon)$ can be thus validated by the procedure discussed in Section \ref{section-eigenpairs} with these blocks.
Iterating these steps with details below, we can validate $S_\epsilon$ and vector bundles over $S_\epsilon$ at the same time.
In any cases, we need to validate small neighborhoods of slow manifolds as {\em seeds} in our iteration steps.
We then use the following terminologies as simplified notations.

\begin{dfn}[Seeds]\rm
Construct an affine fast-saddle-type block following the procedure in Section \ref{section-block}: in particular, $D$ in (\ref{fast-block}).
We shall say an affine fast-saddle-type block $D$, or equivalently $TD$ with an affine transformation $T$ and $\eta^\alpha \equiv 0$ in Section \ref{section-block} satisfying cone conditions a {\em seed}.
\end{dfn}
Seeds describe enclosures of $S_\epsilon$ in the present validation steps.
The detailed algorithm for constructing $S_\epsilon$ and $V^\alpha_\epsilon$ is the following.

\begin{alg}[Construction of $S_\epsilon$ and $V^\alpha_\epsilon$]
\label{alg-bundle}
Let $Y\subset \mathbb{R}^l$ be a compact, contractible set and $\epsilon_0 > 0$ be a given positive number.
Divide $Y$ into small pieces of interval sets $Y=\bigcup_{j=1}^{m_0}Y_j$.
For each $j=1,\cdots, m_0$,
\begin{enumerate}
\item Construct a seed $T_jD_j$ on $Y_j$, with an affine transformation $T_j$, containing a branch of slow manifolds for $\epsilon\in [0,\epsilon_0]$ (cf. Section \ref{section-block}).
Here keep the information of numerical (right) eigenpairs $\{\lambda_i^j,u_i^j\}_{i=1}^n$ at a point $(x_j,y_j)\in T_jD_j$ used in constructing $T_jD_j$, which are re-used in Step 3. 
Let $P^j = [u_1^j,\cdots, u_n^j]$ be the corresponding eigenmatrix, which is used in Algorithm \ref{alg-nbh} later.
Note that we may set $\epsilon = 0$ for calculating eigenpairs of $f_x(x_j,y_j,\epsilon)$.
Moreover, we prepare the left eigenpairs $\{\lambda_i^j,p_i^j\}_{i=1}^n$ of $f_x(x_j,y_j,0)$ for Step 3.
\item Apply Gershgorin's Circle Theorem to validating enclosures of eigenvalues $\{\lambda_i(T_jD_j)\}_{i=1}^n$ of $f_x(x,y,\epsilon)$ on $T_jD_j$ and verify if these enclosures are mutually disjoint. 
\item For each eigenpair of $f_x(x_j,y_j,0)$, apply Krawczyk-type operator $K_{\mathbb{K}}$ to verifying conditions in Proposition \ref{prop-Krawczyk}.
We use the numerical right and left eigenpairs, $\{\lambda_i^j,u_i^j\}_{i=1}^n$ and $\{\lambda_i^j,p_i^j\}_{i=1}^n$, to define the matrix $C_{\mathbb{K}}$ in the definition of $K_{\mathbb{K}}$.
In the practical verifications involving $K_{\mathbb{K}}$, apply the bound of $A(y) \equiv f_x(h_\epsilon(y),y,\epsilon)$ given by
\begin{equation*}
\{f_x(x,y,\epsilon)\mid (x,y)\in T_jD_j, \epsilon\in [0,\epsilon_0]\}.
\end{equation*}
\item (Optional.) Verify the rate condition in Definition \ref{dfn-rate-fast-slow} by calculating the following numbers :
\begin{align*}
k'_{su, j} &:= \frac{\overrightarrow{\mu_{s,2}}(T_jD_j)}{\overrightarrow{\xi_{su,1}}(T_jD_j)},\quad 
k_{su, j} := \lfloor k'_{su, j} \rfloor -1,\quad k_{su} := \min_{j=1,\cdots, m_0}k_{su, j},\\ 
k'_{ss, j} &:= \frac{\overrightarrow{\mu_{ss,2}}(T_jD_j)}{\overrightarrow{\xi_{su,1}}(T_jD_j)},\quad
k_{ss, j} := \lfloor k'_{ss, j} \rfloor -1,\quad k_{ss} := \min_{j=1,\cdots, m_0}k_{ss, j},\\ 
k &:= \min \{k_{su}, k_{ss}\}.
\end{align*} 
\end{enumerate}
If some step fails, then consider the refinement of $Y=\bigcup_{j=1}^{m_0}Y_j$ and try again.
\end{alg}

\begin{thm}
\label{thm-bundle}
Assume that all steps in Algorithm \ref{alg-bundle} returns succeeded for all $j=1,\cdots, m_0$.
Then we obtain the following objects; for each $\epsilon \in [0,\epsilon_0]$,
\begin{enumerate}
\item a slow manifold $S_\epsilon = \{(x,y,\epsilon)\mid x = h_\epsilon(y), y\in Y\}$ for some Lipschitz function $x = h_\epsilon(y)$ defined on $Y$;
\item (right) eigenvectors $\{u_{\epsilon,i}(y)\mid y\in Y\}_{i=1}^n$ of $f_x(h_\epsilon(y),y,\epsilon)$.
\end{enumerate}
These validated objects make vector bundles $V^\alpha_\epsilon$, $\alpha = s,u$, over $S_\epsilon$.
Moreover, if further assume that Step 4 returns $k \geq 1$, then the validated slow manifold $S_{\epsilon}$ is $C^k$ in $Y$.
\end{thm}

\begin{proof}
Step 1 indicates the existence of slow manifolds with the graph representation in the seed.
Cone conditions guarantee that the validated pieces of slow manifolds are uniquely attached in the intersection of seeds (e.g., Lemma 4.9 in \cite{Mat2}).
As a consequence, we obtain the whole branch of slow manifolds $S_\epsilon = \{(x,y,\epsilon)\mid x = h_\epsilon(y), y\in Y\}$.
\par
Step 2 indicates that all eigenvalues of $f_x(h_\epsilon(y),y,\epsilon)$ are simple for each $(y,\epsilon)\in Y\times [0,\epsilon_0]$.
Therefore, assumptions in Lemma \ref{lem-nonsing} with respect to eigenvalues are satisfied and we can appropriately construct a nonsingular matrix $C_{\mathbb{K}}$ so that the Krawczyk-type operator $K_{\mathbb{K}}$ is well-defined.
\par
Step 3 guarantees the existence of eigenvectors by Proposition \ref{prop-Krawczyk}.
Obviously, validated enclosures contain eigenvectors $\{u_{\epsilon,i}(y)\}_{i=1}^n$ for each $(y,\epsilon)\in Y\times [0,\epsilon_0]$, which can be determined uniquely for each $(y,\epsilon)$.
Since $f_x(h_\epsilon(y),y,\epsilon)$ varies continuously on $(y,\epsilon)$, then so do $\{u_{\epsilon,i}(y)\}_{i=1}^n$, and these vectors on the seed constructs a product $\{(h_\epsilon(y),y), u_{\epsilon,i}(y)\}$ for each $i$, which is exactly a trivial vector bundle over $S_\epsilon\cap T_jD_j$ for each $j$.
The uniqueness statement of eigenpairs by the Krawczyk method and simpleness of eigenvalues show that, if $Y_i\cap Y_j \not = \emptyset$ with $i\not = j$, validated eigenpairs coincides, at least up to signatures, in $Y_i\cap Y_j$.
This fact implies that the spanning eigenspaces is uniquely determined in $Y_i\cup Y_j$.
As a consequence, the collections $V^u_\epsilon$ and $V^s_\epsilon$ given by
\begin{align*}
&V^u_\epsilon(h_\epsilon(y),y) := {\rm span}\{u_{\epsilon,1}(y),\cdots, u_{\epsilon,n_u}(y)\},\quad V^s_\epsilon(h_\epsilon(y),y) := {\rm span}\{u_{\epsilon,n_u+1}(y),\cdots, u_{\epsilon,n}(y)\}\\
&V^u_\epsilon := \bigcup_{p=(h_\epsilon(y),y)\in S_\epsilon}V^u_\epsilon(p),\quad V^s_\epsilon := \bigcup_{p=(h_\epsilon(y),y)\in S_\epsilon}V^s_\epsilon(p)
\end{align*}
determine vector bundles over $S_\epsilon$.
Here eigenvectors $\{u_{\epsilon,1}(y),\cdots, u_{\epsilon,n_u}(y)\}$ are associated with eigenvalues with positive real part and $\{u_{\epsilon, n_u+1}(y),\cdots, u_{\epsilon,n}(y)\}$ are associated with eigenvalues with negative real part.
\par
The final assertion directly follows from the fact that $S_\epsilon$ is the intersection of $C^{k_{ss}}$-manifold $W^s_{loc}(S_\epsilon)$ and $C^{k_{su}}$-manifold $W^u_{loc}(S_\epsilon)$.
\end{proof}

%
%
\subsection{Tubular neighborhoods of  slow manifolds with explicit radii}
In Section \ref{section-bundle-validate}, we have validated slow manifolds $S_\epsilon =\{(h_\epsilon(y),y) \mid y\in Y\}$ as well as eigenvectors of $f_x(h_\epsilon(y),y,\epsilon)$ on $S_\epsilon$.
These eigenvectors can be applied to constructing fast-saddle-type blocks, as discussed in Section \ref{section-block}.
Moreover, construction of blocks with positive numbers $\{\eta^\alpha\}_{\alpha = u,s}$ in (\ref{fast-block-2}) yields the block with explicit lower bounds of radii $\eta^u, \eta^s$ centered at $(h_\epsilon(y),y)$ for each $(y,\epsilon)$.
Furthermore, since $h_\epsilon(y)$ depends continuously (possibly smoothly) on $(y,\epsilon)$, we expect that we can validate slices of blocks at each $(y,\epsilon)$ so that they depend continuously on $(y,\epsilon)$.
The collection of such slices are our targeting tubular neighborhoods of $S_\epsilon$.
\par
\bigskip
Before providing the algorithm for constructing tubular neighborhoods of $S_\epsilon$, we prepare an auxiliary concept similar to seeds.
Fix $\epsilon_0 > 0$.
Let $\{\eta^\alpha\}_{\alpha=s,u}$ be given nonnegative numbers.
Assume that all steps in Algorithm  \ref{alg-bundle} with $j=j_0\in \{1,\cdots, m_0\}$ on $[0,\epsilon_0]$ are succeeded, and let $\tilde B_j$ be the validated seed on $Y_{j_0}$.
We define a {\em target} as a fast-saddle-type block constructed by the following steps.
\par
Fix $j_0\in \{1,\cdots, m_0\}$ and $\tilde B = \tilde B_{j_0}$. 
Let $\{[\lambda_i(\tilde B)]\}_{i=1}^n$ be a sequence of interval enclosures of eigenvalues of $f_x(x,y,\epsilon)$ on $\tilde B$; namely,
\begin{equation*}
[\lambda_i(\tilde B)] = \{\lambda \in \mathbb{C} \mid \lambda \text{ is the $i$-th eigenvalue of }f_x(x,y,\epsilon)\text{ for some }(x,y)\in \tilde B,\ \epsilon\in [0,\epsilon_0]\}.
\end{equation*}
We divide $\{[\lambda_i(\tilde B)]\}_{i=1}^n$ into two groups: $\{[\lambda^a_i(\tilde B)]\}_{i=1}^{n_u} \cup \{[\lambda^b_i(\tilde B)]\}_{i=1}^{n_s}$, where $[\lambda^a_i(\tilde B)]$ is the enclosure of the $i$-th eigenvalues with positive real part, and $[\lambda^b_i(\tilde B)]$ is the enclosure of the $i$-th eigenvalues with negative real part.
Also, let $[P(\tilde B)]$ be the interval enclosure of eigenmatrices associated with $\{[\lambda_i(\tilde B)]\}_{i=1}^n$ validated in Step 3 of Algorithm \ref{alg-bundle}.
For a given $ch$-set $E\times Y$ containing the seed $\tilde B$, we compute the following enclosures:
\begin{align*}
\left\{ [P(\tilde B)]^{-1} P^{j_0} F(x,y,\epsilon) \mid (x,y)\in [T(z,w)]\subset E\times Y, \epsilon \in [0,\epsilon_0]\right\}_{i,j_i} \subsetneq [\delta_{i,j_i}^-, \delta_{i,j_i}^+],&\\
\notag
i=1,2,\quad j_1 = 1,\cdots, n_u,\quad j_2 = 1,\cdots, n_s,
\end{align*}
where $[P(\tilde B)]^{-1} = \{P^{-1}\mid P \in [P(\tilde B)]\}$ and $P^{j_0}$ is the sample eigenmatrix computed in Step 1 of Algorithm \ref{alg-bundle}. 
The enclosure $[T(z,w)]$ is given by
\begin{equation*}
[T(z,w)] := \left\{ \left(Pz + \bar x - f_x(\bar x, \bar y)^{-1}f_y(\bar x, \bar y)w, w + \bar y\right) \mid P\in [P(\tilde B)]\right\},
\end{equation*}
where $(\bar x, \bar y)$ is a (numerical) equilibrium for (\ref{layer}), i.e., $f(\bar x,\bar y,0)\approx 0$, for constructing the seed $\tilde B$.
Then define the set $D_{j_0}\subset \mathbb{R}^{n+l}$ by the following: 
\begin{align*}
D_{j_0} &:= \prod_{i=1}^{n_u} [a_i^-, a_i^+]\times \prod_{i=1}^{n_s} [b_i^-, b_i^+]\times Y_{j_0},\\
[a_i^-, a_i^+] &:= \left[-\frac{\delta_{1,i}^+}{\inf [\lambda^a_i(\tilde B)]} - \eta^u,\ -\frac{\delta_{1,i}^-}{\inf [\lambda^a_i(\tilde B)]} + \eta^u \right],\\
[b_i^-, b_i^+] &:= \left[-\frac{\delta_{2,i}^-}{\sup [\lambda^b_i(\tilde B)]} - \eta^s,\ -\frac{\delta_{2,i}^+}{\sup [\lambda^b_i(\tilde B)]} + \eta^s \right].
\end{align*}
If we validate $[TD_{j_0}]\subset E\times Y_{j_0}$, then $D_{j_0}$ becomes a fast-saddle-type block by arguments in Section \ref{section-block}.

\begin{dfn}[Targets]\rm
Let $D_{j_0}$ be a fast-saddle-type block with $\{\eta^\alpha\}_{\alpha = u,s}$ satisfying $[TD_{j_0}]\subset B\times Y_{j_0}$.
We further assume that cone conditions are satisfied in $D_{j_0}$.
We then say $D_{j_0}$ a {\em target on $Y_{j_0}$ with radii $\{\eta^\alpha\}_{\alpha = u,s}$}.
\end{dfn}

Now we are ready to construct tubular neighborhoods of $S_\epsilon$.

\begin{alg}
\label{alg-nbh}
Let $Y\subset \mathbb{R}^l$ be a compact, contractible set and $\epsilon_0 > 0$ be a given positive number.
Divide $Y$ into small pieces of interval sets $Y=\bigcup_{j=1}^{m_0}Y_j$.
For each $j=1,\cdots, m_0$,
\begin{enumerate}
\item Run Steps 1 $\sim$ 3 in Algorithm \ref{alg-bundle}.
\item Construct a target $\mathcal{D}_j$ on $Y_j$ with the sequence of radii $\{\eta^\alpha\}_{\alpha=u,s}$ for $\epsilon\in [0,\epsilon_0]$ containing the seed $\tilde B_j$.
\item Verify
\begin{equation}
\label{incl-tub}
\eta^u > \diag (\pi_a ([P(\tilde B_j)]^{-1}P^j\tilde B_j))\quad \text{ and }\quad \eta^s > \diag (\pi_b  ([P(\tilde B_j)]^{-1}P^j\tilde B_j)).
\end{equation}
\item (Optional.) Verify rate conditions in Definition \ref{dfn-rate-fast-slow} by calculating the following numbers :
\begin{align*}
k'_{su, j} &:= \frac{\overrightarrow{\mu_{s,2}}(\mathcal{D}_j)}{\overrightarrow{\xi_{su,1}}(\mathcal{D}_j)},\quad 
k_{su, j} := \lfloor k'_{su, j} \rfloor -1,\quad k_{su} := \min_{j=1,\cdots, m_0}k_{su, j},\\ 
k'_{ss, j} &:= \frac{\overrightarrow{\mu_{ss,2}}(\mathcal{D}_j)}{\overrightarrow{\xi_{su,1}}(\mathcal{D}_j)},\quad
k_{ss, j} := \lfloor k'_{ss, j} \rfloor -1,\quad k_{ss} := \min_{j=1,\cdots, m_0}k_{ss, j},\\ 
k &:= \min \{k_{su}, k_{ss}\}.
\end{align*} 
\end{enumerate}
If all steps are succeeded, return {\tt true}.
\end{alg}

\begin{rem}[Geometric meaning of (\ref{incl-tub})]\rm
Note that the seed $\tilde B_j$ is constructed in the {\em fixed} coordinate via $T_j(z,w) = (P^jz + x_j - f_x(x_j, y_j)^{-1}f_y(x_j, y_j)w, w+y_j)$ in $Y_j$.
Inequalities (\ref{incl-tub}) estimate the location of seeds and the desiring slice $\mathcal{N}\mid_{\{(y,\epsilon) = (\bar y,\bar \epsilon)\}}$ in the coordinate {\em depending on $y\in Y_j$} via $T(z,y) = (h_\epsilon(y) + P(y)z, y)$.
\end{rem}

\begin{thm}
\label{thm-nbh}
Assume that Algorithm \ref{alg-nbh} returns {\tt true}.
Then, for each $\epsilon\in [0,\epsilon_0]$, the union of targets $\bigcup_{j=1}^{m_0} \mathcal{D}_j$ contains the fast-saddle-type block $\mathcal{N} = \mathcal{N}(S_\epsilon;\{\eta^\alpha\}_{\alpha=u,s})$ of the following form:
\begin{align*}
&\mathcal{N} = T_\epsilon\mathcal{D}_{(\eta^u, \eta^s)} = \{(h_\epsilon(y) + P_\epsilon(y)z, y)\mid z\in R(\eta^u, \eta^s), y\in Y\}\quad \text{ with }\\
&\mathcal{D}_{(\eta^u, \eta^s)} = R(\eta^u, \eta^s) \times Y,\quad R(\eta^u, \eta^s) = \prod_{j=1}^{n_u} [-\eta^u, \eta^u]\times \prod_{j=n_u+1}^{n} [-\eta^s, \eta^s].
\end{align*}
Here $P_\epsilon(y)$ denotes the $n\times n$ matrix whose $i$-th column is the eigenvector $u_{\epsilon,i}(y)$ associated with eigenvalues $\{\lambda_{\epsilon,i}(y)\}_{i=1}^n$ validated in Step 1.
The construction of $\mathcal{D}_{(\eta^u, \eta^s)}$ is considered in the following form of (\ref{fast-slow}):
\begin{align*}
a_i' &= \lambda_{\epsilon,i}^a(y) a_i  + F_{1,i}(x,y,\epsilon),\quad {\rm Re}\lambda_i > 0,\quad i=1,\cdots,  n_u,\\
b_i' &= \lambda_{\epsilon,i}^b(y) b_i  + F_{2,i}(x,y,\epsilon),\quad {\rm Re}\lambda_i < 0,\quad i=1,\cdots, n_s
\end{align*}
via the transformation $x = h_\epsilon(y) +P_\epsilon(y)(a,b)$, where $\{\lambda_{\epsilon,j}(y)\}_{j=1}^n$ is the family of eigenvalues validated in Step 1.
If we further validate Step 4 in Algorithm \ref{alg-nbh} for all $j=1,\cdots, m_0$, then we have $W^s(S_\epsilon)$ as a $C^{k_{ss}}$-manifold in $\mathcal{N}$, $W^u(S_\epsilon)$ as a $C^{k_{su}}$-manifold in $\mathcal{N}$ and $S_\epsilon$ as a $C^k$-manifold in $\mathcal{N}$.
\end{thm}

\begin{proof}
By the definition, for all $y\in Y_j$, a target $\mathcal{D}_j$ on $Y_j$ contains the slice $T_\epsilon \mathcal{D}_j(y;\epsilon)$.
The block $\mathcal{D}_j(y;\epsilon)$ is given by
\begin{align*}
\mathcal{D}_j(y;\epsilon)&:= \prod_{i=1}^{n_u} [a_{\epsilon,i}^-(y), a_{\epsilon,i}^+(y)]\times \prod_{i=1}^{n_s} [b_{\epsilon,i}^-(y), b_{\epsilon,i}^+(y)]\times Y_j,\\
[a_{\epsilon,i}^-(y), a_{\epsilon,i}^+(y)] &:= \left[-\frac{\delta_{1,i}^+(y)}{\lambda^a_{\epsilon,i}(y)} - \eta^u,\ -\frac{\delta_{1,i}^-(y)}{\lambda^a_{\epsilon,i}(y)} + \eta^u \right],\\
[b_{\epsilon,i}^-(y), b_{\epsilon,i}^+(y)] &:= \left[-\frac{\delta_{2,i}^-(y)}{\lambda^b_{\epsilon,i}(y)} - \eta^s,\ -\frac{\delta_{2,i}^+(y)}{\lambda^b_{\epsilon,i}(y)} + \eta^s \right].
\end{align*}

Remark that, by our construction, $[\delta_{i,j_i}^-(y), \delta_{i,j_i}^+(y)]\subset [\delta_{i,j_i}^-, \delta_{i,j_i}^+]$ holds for each $y\in Y_j$.
Obviously, every slice $\mathcal{N}\mid_{\{(y,\epsilon) = (\bar y,\bar \epsilon)\}}$ with $\bar y\in Y_j$ is contained in the slice $T_\epsilon \mathcal{D}_j(\bar y; \bar \epsilon)$.
The rest is to prove that $\mathcal{N}$ is of fast-saddle-type.
Note that cone conditions for both seeds and targets indicate that the validated slow manifold in targets is actually contained in seeds. 
This fact and inequalities (\ref{incl-tub}) for all $j$ imply that the slice  of seeds $\tilde B_j\mid_{\{(y,\epsilon)=(\bar y,\bar \epsilon)\}}$ is contained in the interior of the slice $\mathcal{N}\mid_{\{(y,\epsilon) = (\bar y,\bar \epsilon)\}}$ and that $\partial \tilde B_j\mid_{\{(\bar y,\bar \epsilon)\}} \cap \partial \mathcal{N}\mid_{\{(\bar y,\bar \epsilon)\}}=\emptyset$.
From the construction of blocks in (\ref{fast-block-2}), we then know that all points on the boundary of $\mathcal{D}_{(\eta^u, \eta^s)}$ for all $(y,\epsilon)$ are either exit or entrance points, which shows that $\mathcal{N}$ is a fast-saddle-type block.
\end{proof}

Our construction naturally gives the definition of  the {\em fast-exit} $\mathcal{N}^{f,-}$ and the {\em fast-entrance} $\mathcal{N}^{f,+}$ of tubular neighborhood $\mathcal{N}$ by
\begin{align*}
\mathcal{N}^{f,-} &= T_\epsilon \mathcal{D}_{\eta^u,\eta^s}^{f,-},\quad \mathcal{N}^{f,+} =  T_\epsilon \mathcal{D}_{\eta^u,\eta^s}^{f,+},\\
\mathcal{D}_{\eta^u,\eta^s}^{f,-} &= R(\eta^u,\eta^s)^{f,-}\times Y\\
	&\equiv \bigcup_{J=1}^{n_u}\left(\prod_{j=1}^{J-1} [-\eta^u, \eta^u]\times \{\pm \eta^u\}\times \prod_{j=J+1}^{n_u} [-\eta^u, \eta^u]\times \prod_{j=n_u+1}^{n} [-\eta^s, \eta^s]\right)\times Y,\\
\mathcal{D}_{\eta^u,\eta^s}^{f,+} &= R(\eta^u,\eta^s)^{f,+}\times Y\\
	&\equiv \bigcup_{J=n_u+1}^{n}\left(\prod_{j=1}^{n_u} [-\eta^s, \eta^s] \times \prod_{j=1}^{J-1} [-\eta^s, \eta^s]\times \{\pm \eta^s\}\times \prod_{j=J+1}^{n_s} [-\eta^s, \eta^s]\right)\times Y.
\end{align*}
\par
\bigskip
The proof of Theorem \ref{thm-nbh} induces an important property of tubular neighborhoods, which is just a case of Theorem \ref{thm-nbh} with $m_0 = 2$.
\begin{cor}[Unique continuation of tubular neighborhoods]
\label{cor-extend}
Let $Y_1, Y_2\subset \mathbb{R}^l$ be compact contractible sets such that $Y_1\cap Y_2\not = \emptyset$ is also contractible.
Also let $\mathcal{D}_i\ (i=1,2)$ be targets on $Y_i$ with common radii $\{\eta^\alpha\}_{\alpha = u,s}$ validated for all $\epsilon \in [0,\epsilon_0]$.
Then, for any $\epsilon \in [0,\epsilon_0]$, the union $\mathcal{D}_1\cup \mathcal{D}_2$ contains the tubular neighborhood $\mathcal{N}$ on $Y= Y_1\cup Y_2$.
\end{cor}
This corollary indicates that we can extend tubular neighborhoods of slow manifolds {\em in arbitrary range of slow variables} as long as assumptions of Corollary \ref{cor-extend} are satisfied.

\subsection{Conic and star-shaped neighborhoods of slow manifolds}

Our validations of tubular neighborhoods are easily extended to validation of cones.
Recall from Section \ref{section-inv-mfd} that (un)stable cones with the vertex $(z_0, \eta_0)=((a_0,b_0,y_0),\eta_0)\in \mathbb{R}^{n+l+1}$ are described by
\begin{align*}
C_M^u(z_0, \eta_0) &:= \{(a,b,y,\eta)\mid |a - a_0|^2 \geq  M^2(|b - b_0|^2 + |y - y_0|^2 + |\eta - \eta_0|^2)\},\\
C_M^s(z_0, \eta_0) &:= \{(a,b,y,\eta)\mid |b - b_0|^2 \geq  M^2(|a - a_0|^2 + |y - y_0|^2 + |\eta - \eta_0|^2)\},
\end{align*}
which is considered in the coordinate $(a,b,y)$ through, say, the transformation
\begin{equation*}
(x,y) = (h_\epsilon(y) + P_\epsilon(y)(a,b),y),\quad P_\epsilon(y)^{-1}f_x(h_\epsilon(y), y,\epsilon)P_\epsilon(y) = {\rm diag}(\lambda_{\epsilon,1}(y),\cdots, \lambda_{\epsilon,n}(y))
\end{equation*}
to obtain (\ref{abstract-form}).
Like the slow manifold $S_\epsilon$, the above representation of cones depends continuously on $y$ in general.
We can then apply the continuous family of eigenpairs $\{(\lambda_{\epsilon,i}(y); u_{\epsilon,i}(y))\}_{i=1}^n$ to validating cones which depend continuously on $y$, which yields cone-like neighborhoods of slow manifolds as follows.

\begin{dfn}[Conic neighborhoods]\rm
\label{dfn-conic-nbh}
Let $S_\epsilon = \{x=h_\epsilon(y)\mid y\in Y\}$ be a normally hyperbolic slow manifold with the graph representation on an interval set $Y\subset \mathbb{R}^l$.
Let also $\mathcal{N} = \{(h_\epsilon(y) + x, y) \mid x\in P_\epsilon(y)R(\eta^u, \eta^s),\ y\in Y\}$ be a tubular neighborhood of $S_\epsilon$ for some $\eta^u, \eta^s > 0$, where $P(y):\mathbb{R}^n\to \mathbb{R}^n$ is a homeomorphism.
Let $M^u, M^s \geq 1$ and $l^u, l^s \geq 0$.
\par
\bigskip
Define the set $\mathcal{C}^u_{M^u,l^u}(\eta^u,\eta^s) := T_\epsilon (C^u_{M^u,l^u}(\eta^u, \eta^s)\times Y)$, 
where 
\begin{align*}
C^u_{M^u,l^u}&(\eta^u, \eta^s) = \left[ R(\eta^u, \eta^s)\cup \{(a,b)\mid |a-a_0|\geq M^u |b-b_0|, \right.\\
	 &\quad \quad (a_0,b_0)\in R(\eta^u, \eta^s)^{f,-},\ (a-a_0,b-b_0)\cdot \nu(a_0,b_0)>0\} \left. \right]\cap R\left(\eta^u+l^u, \eta^s+ \frac{l^u}{M^u}\right)
\end{align*}
and $\nu(a_0,b_0)$ is the outer unit normal vector of $R(\eta^u, \eta^s)$ at $(a_0,b_0)\in \partial R(\eta^u, \eta^s)$. 
\par
Similarly, define
$\mathcal{C}^s_{M^s,l^s}(\eta^u,\eta^s) := T_\epsilon(C^s_{M^s,l^s}(\eta^u, \eta^s)\times Y)$, 
where 
\begin{align*}
C^s_{M^s,l^s}&(\eta^u, \eta^s) = \left[ R(\eta^u, \eta^s)\cup \{(a,b)\mid |b-b_0| \geq M^s |a-a_0|, \right. \\
	&\quad \quad (a_0,b_0)\in R(\eta^u, \eta^s)^{f,+},\ (a-a_0,b-b_0)\cdot \nu(a_0,b_0)>0\} \left. \right]\cap R\left(\eta^u+\frac{l^s}{M^s}, \eta^s+l^s \right).
\end{align*}
We say the set $\mathcal{C}^u\equiv \mathcal{C}^u_{M^u,l^u}(\eta^u,\eta^s)$ (resp. $\mathcal{C}^s \equiv \mathcal{C}^s_{M^s,l^s}(\eta^u,\eta^s)$) the {\em unstable (resp. stable) conic neighborhood of $S_\epsilon$} if the unstable $M^u$-cone condition (resp. the stable $M^s$-cone condition) is satisfied in $\mathcal{C}^u_{M^u,l^u}(\eta^u,\eta^s)$ (resp. $\mathcal{C}^s_{M^s,l^s}(\eta^u,\eta^s)$).
See Figure \ref{fig-pic_tub}-(b).
\par
\bigskip
Finally, we call the set $\mathcal{S} := \mathcal{C}^u \cup \mathcal{C}^s$ a {\em star-shaped neighborhood} of $S_\epsilon$.
See Figure \ref{fig-pic_tub}-(c).
\end{dfn}

\par
\bigskip
Conic neighborhoods consist of tubular neighborhoods and points inside corresponding cones whose vertices are on the boundary $\mathcal{N}^{f,\pm}$.
Properties of cones and the definition of $\mathcal{C}^\alpha,\ \alpha = u,s$, immediately yield the following properties.

\begin{thm}
Let $\mathcal{C}^u$ and $\mathcal{C}^s$ be unstable and stable conic neighborhoods of a slow manifold $S_\epsilon = \{x=h_\epsilon(y)\}$.
Then both $\mathcal{C}^u$ and $\mathcal{C}^s$ are homeomorphic to $\mathcal{N}$.
The fast-exits of $\mathcal{C}^u$ and $\mathcal{C}^s$ are given as follows, respectively (compare with $\mathcal{N}^{f,-}$) : 
\begin{align*}
(\mathcal{C}^u)^{f,-} &= T_\epsilon \left((C^u)^{f,-}\times Y\right),\\
(C^u)^{f,-} &= R\left(\eta^u + l^u,\eta^s + \frac{l^u}{M^u}\right)^{f,-},\\
(\mathcal{C}^s)^{f,-} &= T_\epsilon \left((C^s)^{f,-}\times Y\right),\\
(C^s)^{f,-} &= R\left(\eta^u,\eta^s\right)^{f,-} \cup \{(a,b)\in C^s \mid |b-b_0| = M^s |a-a_0|,\ (a_0,b_0)\in R(\eta^u,\eta^s)^{f,-}\cap \overline{R(\eta^u,\eta^s)^{f,+}} \},
\end{align*}
where $C^u = C^u_{M^u,l^u}(\eta^u, \eta^s)$ and $C^s = C^s_{M^s,l^s}(\eta^u, \eta^s)$.
Similarly, the fast-entrances of $\mathcal{C}^u$ and $\mathcal{C}^s$ are given as follows, respectively (compare with $\mathcal{N}^{f,+}$) : 
\begin{align*}
(\mathcal{C}^u)^{f,+} &= T_\epsilon \left((C^u)^{f,+}\times Y\right) \equiv T_\epsilon \left((C^u_{M^u,l^u}(\eta^u, \eta^s))^{f,+} \times Y\right),\\
(C^u)^{f,+} &= R\left(\eta^u,\eta^s\right)^{f,+} \cup \{(a,b)\in C^u \mid |a-a_0| = M^u |b-b_0|\} \setminus (C^u)^{f,-},\\
(\mathcal{C}^s)^{f,+} &= T_\epsilon \left((C^s)^{f,+}\times Y\right) \equiv T_\epsilon \left((C^s_{M^s,l^s}(\eta^u, \eta^s))^{f,+} \times Y\right),\\
(C^s)^{f,+} &= R\left(\eta^u + \frac{l^s}{M^s},\eta^s + l^s\right)^{f,+} \setminus (C^s)^{f,-}.
\end{align*}
\end{thm}

\begin{proof}
We give a proof only for $\mathcal{C}^u$.
The case $\mathcal{C}^s$ is similar.
\par
The first assertion immediately follows from the definition of $\mathcal{C}^u$ and $\mathcal{N}$.
Note that it is sufficient to consider the structure of $(C^u)^{f,\pm}$ since $Y$ does not affect the fast-boundary and the homeomorphism $T_\epsilon$ preserves the exit-entrance information of boundary.
\par
We know that
\begin{align*}
\partial C^u &\subset R(\eta^u,\eta^s)^{f,+}\\
	&\cup \{(a,b)\in C^u\mid  |a-a_0|= M^u |b-b_0|,\ (a_0,b_0)\in R(\eta^u,\eta^s)^{f,-}\cap \overline{R(\eta^u,\eta^s)^{f,+}} \}\\
	&\cup \left\{(a,b)\in C^u\mid a = \pm \left(\eta^u + l^u\right)\right\}.
\end{align*}
It is thus sufficient to study the dynamics on the right-hand side.
The first set $R(\eta^u,\eta^s)^{f,+}$ corresponds to a part of fast-entrance, and hence to a part of $(C^u)^{f,+}$.
Consider the second set $\{(a,b)\in C^u\mid  |a-a_0|= M^u |b-b_0|,\ (a_0,b_0)\in R(\eta^u,\eta^s)^{f,-}\cap \overline{R(\eta^u,\eta^s)^{f,+}} \}$.
The unstable cone condition indicates that the flow intersects the set transversely so that points enter $C^u$.
Namely, the second set is the fast-entrance of $C^u$, which is regarded as the part of $(C^u)^{f,+}$.
Move to the final set $\left\{(a,b)\in C^u\mid a = \pm \left(\eta^u + l^u\right)\right\}$.
We only consider the case $a = \eta^u + l^u$.
Another part follows from mirror arguments.
Notice that any points $(a,b)$ on the set is included in the closure of the unstable $M^u$-cone centered at a point $(a_0,b_0)$ on $R(\eta^u,\eta^s)^{f,-}$.
The expansion result in the unstable cone under the unstable cone condition (Proposition \ref{prop-unst-m-cone}) indicates that the differential $\frac{d}{dt}|a-a_0|^2$ is positive, which shows the final set is the exit.
Indeed, since $(a_0, b_0)$ is on the fast-exit $\mathcal{N}^{f,-}$, then $(a_0, b_0)$ penetrates $\mathcal{C}^u \setminus \mathcal{N}$, which indicates that $\pi_a \varphi(t,(a_0,b_0,y,\epsilon)) > \eta^u$ for sufficiently small $t > 0$.
At the same time, the inequality $\frac{d}{dt}|a-a_0|^2 > 0$ indicates that $\pi_a \varphi(t,(a,b,y, \epsilon)) > \eta^u + l^u$ for the same $t$.
This observation indicates that the final set is a part of $(C^u)^{f,-}$.
\end{proof}

Validation of conic neighborhoods is quite simple, as shown in the following proposition, which immediately follows from the geometry of cones and evolution of disks.

\begin{prop}
Let $\mathcal{N} = T_\epsilon \mathcal{D}_{(\eta^u, \eta^s)}$ be a tubular neighborhood of $S_\epsilon$ validated in Theorem \ref{thm-nbh}.
Assume further that the unstable $M^u$-cone condition is satisfied in the set $\mathcal{N}^u(M^u,l^u) \equiv S_\epsilon + P_\epsilon(y)R\left(\eta^u+l^u, \eta^s+ \frac{l^u}{M^u}\right)$.
Then the set $\mathcal{N}^u(M^u,l^u)$ contains the conic neighborhood $\mathcal{C}^u_{M^u,l^u}$ of $S_\epsilon$.
\par
Similarly, assume that the stable $M^s$-cone condition is satisfied in the set $\mathcal{N}^s(M^s,l^s) \equiv S_\epsilon + P_\epsilon(y)R\left(\eta^u+\frac{l^s}{M^s}, \eta^s+l^s \right)$.
Then the set $\mathcal{N}^s(M^s,l^s)$ contains the conic neighborhood $\mathcal{C}^s_{M^s,l^s}$ of $S_\epsilon$.
\end{prop}

\begin{proof}
This immediately follows from the inclusion
\begin{equation*}
\mathcal{C}^u_{M^u,l^u}(\eta^u,\eta^s) = T_\epsilon (C^u_{M^u,l^u}(\eta^u, \eta^s)\times Y) \subset T_\epsilon \mathcal{D}_{\left(\eta^u+l^u, \eta^s+ \frac{l^u}{M^u}\right)}.
\end{equation*}
The statement for $\mathcal{C}^s$ is similar.
Compare with Definition \ref{dfn-conic-nbh}.
\end{proof}

\begin{figure}[htbp]\em
\begin{minipage}{1.0\hsize}
\centering
\includegraphics[width=12.0cm]{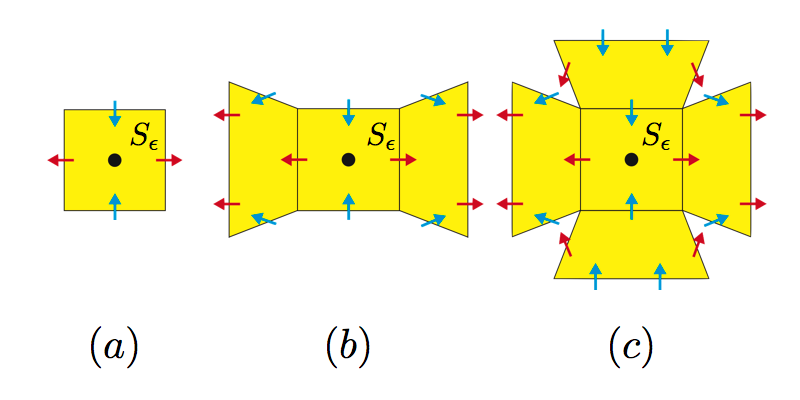}
\end{minipage}
\caption{Sections of tubular, conic and star-shaped neighborhoods of $S_\epsilon$}
\label{fig-pic_tub}
Unlike the usual validations of equilibria in terms of isolating blocks or radii polynomials (e.g., \cite{CL2013, CLM2015}), {\bf the center of all neighborhoods is not the approximate equilibrium but the rigorous equilibrium}.
\end{figure}

As a corollary of Invariant Manifold Theorem (Theorem \ref{thm-smooth-slow-mfd}), we obtain the following extended representation of $W^s(S_\epsilon)$ and $W^u(S_\epsilon)$ in conic neighborhoods.

\begin{cor}
Consider (\ref{abstract-form}). 
Let $\mathcal{C}^u$ and $\mathcal{C}^s$ be unstable and stable conic neighborhoods of $S_\epsilon$, respectively, for $\epsilon\in [0,\epsilon_0]$, and $\mathcal{S}=\mathcal{C}^u\cup \mathcal{C}^s$ be a star-shaped neighborhood. 
Assume that the stable $M^s$-cone condition in $\mathcal{C}^s$ and the stable $M^s$-cone condition in $\mathcal{C}^s$ are satisfied.
Then
\begin{enumerate}
\item the smooth function $h^s(b,y,\epsilon)$ validated in Theorem \ref{thm-smooth-slow-mfd} is extended to $\prod_{i=n_u+1}^n [-(\eta^s+l^s),\eta^s+l^s]\times Y\times [0,\epsilon_0]$ with the range $\prod_{i=1}^{n_u} [-(\eta^u+l^s/M^s),\eta^u+l^s/M^s]$.
The graph $(h^s(b,y,\epsilon), b,y)$ is included in $C^s_{M^s}$.
\item the smooth function $h^u(a,y,\epsilon)$ validated in Theorem \ref{thm-smooth-slow-mfd} is extended to $\prod_{i=1}^{n_u} [-(\eta^u+l^u),\eta^u+l^u]\times Y\times [0,\epsilon_0]$ with the range $\prod_{i=n_u+1}^n [-(\eta^s+l^u/M^u),\eta^s+l^u/M^u]$.
The graph $(a, h^u(a,y,\epsilon),y)$ is included in $C^u_{M^u}$.
\end{enumerate}
\end{cor}


%
%
\subsection{Remark : an aspect of parameterization}
\label{section-param}

We comment about relationships to {\em parameterization method} (e.g., \cite{CFdlL2003, CFdlL2005}) as well as their rigorous numerics.
\begin{dfn}[e.g., \cite{CFdlL2003}]\rm
Consider a nonlinear vector field $F$ on $\mathbb{R}^m$ and another vector field $G$ on $\mathbb{R}^m$.
A {\em parameterization} of $F$ to $G$ is a homeomorphism $K : \mathbb{R}^m\to \mathbb{R}^m$ such that the following equation holds:
\begin{equation*}
F\circ K = K\circ G,
\end{equation*}
namely, $K$ is a topological conjugacy between $F$ and $G$.
\end{dfn}

Validated slow manifolds $S_\epsilon = \{x = h_\epsilon(y)\}$ and eigenpairs $\{\lambda_i(y;\epsilon),u_i(y;\epsilon)\}_{i=1}^n$ gives a parameterization between (\ref{fast-slow})$_\epsilon$ and (\ref{abstract-form}) on $Y\subset \mathbb{R}^l$, which can be shown as follows. 
Let $P_\epsilon(y)$ be a nonsingular matrix whose $j$-th column is $u_j(y;\epsilon)$.
Then the mapping
\begin{equation}
\label{our-param}
K_\epsilon(y) : \mathbb{R}^n\to \mathbb{R}^n,\quad K_\epsilon(y)(z) := h_\epsilon(y) + P_\epsilon(y)z
\end{equation}
gives a $C^k$-family of change of coordinates between (\ref{fast-slow})$_\epsilon$ and (\ref{abstract-form}).
In particular, the following statement holds true.
\begin{cor}
Assume that there is a compact $C^k$-slow manifold $S_\epsilon = \{x = h_\epsilon(y)\mid y\in Y\}$ with $k\geq 1$, where $Y\subset \mathbb{R}^l$ is an $h$-set.
Then the $C^k$-diffeomorphic family of mappings $\mathcal{K}=[K]\times id_l,\ [K] = \{K_\epsilon(y) : \mathbb{R}^n\to \mathbb{R}^n\mid y\in Y, \epsilon \in [0,\epsilon_0]\}$ gives a family of parameterizations of (\ref{fast-slow})$_\epsilon$ to (\ref{abstract-form}) in the sense that
\begin{equation*}
\begin{pmatrix}
f(K_\epsilon(y)z,y,\epsilon)\\
\epsilon g(K_\epsilon(y)z,y,\epsilon)
\end{pmatrix}
=
\{(h_\epsilon(y),0)\} + 
(P_\epsilon(y)\ id_{l})
\begin{pmatrix}
\begin{pmatrix}
A_\epsilon(y) a + F_1(K_\epsilon(y)z, y,\epsilon)\\
B_\epsilon(y) b + F_2(K_\epsilon(y)z, y,\epsilon)\\
\end{pmatrix}\\
\epsilon g(K_\epsilon(y)z,y,\epsilon)
\end{pmatrix}
\end{equation*}
for some smooth functions $F_1$ and $F_2$, such that
\begin{equation*}
P_\epsilon(y)f_x(h_\epsilon(y),y,\epsilon)P_\epsilon(y)^{-1} = 
\begin{pmatrix}
A_\epsilon(y) & 0 \\ 0 & B_\epsilon(y)
\end{pmatrix}\equiv \diag (\lambda_{\epsilon,1}(y), \cdots, \lambda_{\epsilon,n}(y))
\end{equation*}
and that $F_i(0,y,\epsilon) \equiv 0$ for $y\in Y$, where $id_{l}$ is the identity map on $\mathbb{R}^l$.
\end{cor}

\begin{proof}
By assumption of $S_\epsilon$ and smooth dependence of eigenpairs with respect to the vector field $f$, the mapping $K_\epsilon(y)$ is $C^k$ for all $(y,\epsilon)$.
It is also $C^k$ with respect to $(y,\epsilon)$.
The rest of statements directly follows from definitions.
\end{proof}

The above statement shows that our change of coordinate around slow manifolds stated in Theorem \ref{thm-nbh} gives a parameterization $\mathcal{K}$ of vector field (\ref{fast-slow})$_\epsilon$ for all $\epsilon\in [0,\epsilon_0]$ up to the first (namely, linear) order term.
Applications of general (namely, up to higher order) parameterization with both non-rigorous and rigorous numerical calculations open the door for calculating Fenichel normal forms around concrete slow manifolds for concrete fast-slow systems.
Numerical applications of parameterization can be seen in e.g., \cite{CLM2015}.

\section{Numerical validation examples and discussions}
\label{section-examples}
In this section, we demonstrate validations of tubular neighborhoods centered at slow manifolds.
Our procedure also validates associated vector bundles over slow manifolds.
We further demonstrate validations of conic and star-shaped neighborhoods associated with tubular neighborhoods centered at slow manifolds.
Our examples here focus on the following three points:
\begin{itemize}
\item {\em twisted tubular neighborhoods} (Section \ref{section-twisted});
\item {\em global and smooth neighborhoods along nonlinear curves} (Section \ref{section-FN}); and
\item applicability of our procedures for fast-slow systems with {\em multi-dimensional fast and slow variables} (Section \ref{section-PP}).
\end{itemize}

All computations are done by MacBook Pro Early 2015 model (3.1 GHz, Intel Core i7 Processor, 16GB 1867 MHz DDR3 Memory), GCC version 7.0 (with option {\tt -O3 -DNDEBUG -DKV\_FASTROUND}) and kv library \cite{kv} version 0.4.43.
Computation times in our validation results in these environments stated here are listed in Table \ref{table-time}; at the end of this section.
Validation codes are available at \cite{Mat_RM}.

\subsection{Twisted slow-periodic motion}
\label{section-twisted}
The first example is an artificial but simple system in cylindrical coordinate in $\mathbb{R}^3$ given by 
\begin{equation}
\label{ex-Fenichel}
\begin{cases}
r' = r(1-r^2)\cos \theta - z\sin \theta,\\
z' = r(1-r^2)\sin \theta + z\cos \theta,\\
\theta' = \epsilon.
\end{cases}
\end{equation}
The aim of this example is to construct \lq\lq twisted" neighborhoods of slow manifolds.
\par
First consider the case $\epsilon = 0$, in which case $\theta$ is just a parameter.
Obviously the set $S = \{r=1, z=0\}$ consists of equilibria, which is actually an invariant circle.
We then follow Algorithms \ref{alg-bundle} and {alg-nbh}, which validate, if succeeded, slow manifolds $S_\epsilon$ near $S$ as well as vector bundles over $S$ and tubular neighborhoods centered at $S_\epsilon$.

\begin{car}
\label{car-Fenichel}
Consider (\ref{ex-Fenichel}).
We validate a tubular neighborhood centered at the slow manifold $S_\epsilon$ near or equal to the nullcline $S=\{r=1, z = 0, \theta\in [0,2\pi]\}$ with radii $\eta^u = \eta^s = 1.0\times 10^{-4}$.
The vector bundles $V^u_\epsilon$ and $V^s_\epsilon$ over $S_\epsilon$ are shown in Fig. \ref{fig-fiber-test1}.
\end{car}
Note that $\epsilon$ gives no effect on whole validations in this example.
This example, we show that we can validate twisted vector bundles as well as tubular neighborhoods of slow manifolds which reflect the twistedness.

\begin{figure}[htbp]\em
\begin{minipage}{0.5\hsize}
\centering
\includegraphics[width=7.0cm]{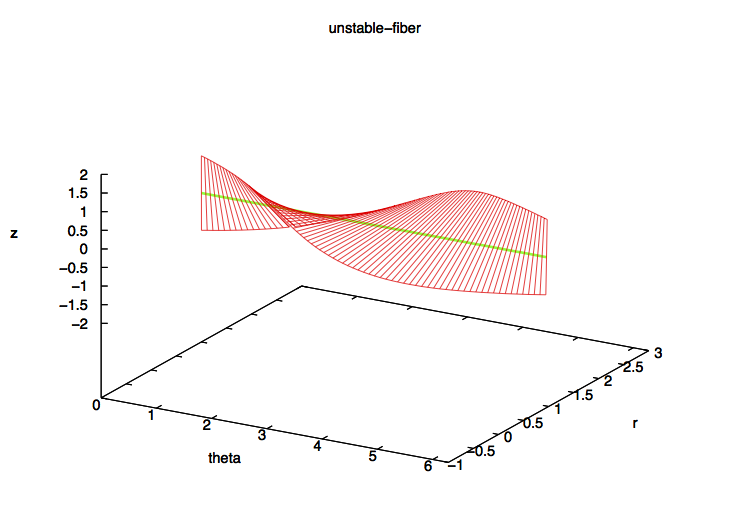}
(a)
\end{minipage}
\begin{minipage}{0.5\hsize}
\centering
\includegraphics[width=7.0cm]{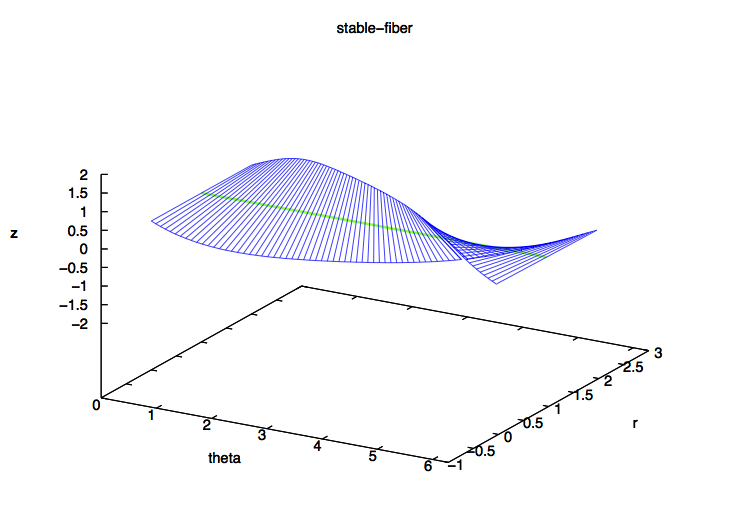}
(b)
\end{minipage}
\caption{The center of validated vector bundles $V^u_\epsilon$ and $V^s_\epsilon$ for (\ref{ex-Fenichel})}
\label{fig-fiber-test1}
(a) : The bundle $V^u_\epsilon$. 
The base space is $S_\epsilon$ with $\theta \in [0,2\pi]$ colored by green. 
Each fiber is spanned by the unstable eigenvector $u(\theta)$ colored by red.
\par
(b) : The bundle $V^s_\epsilon$. 
The base space is $S_\epsilon$ with $\theta \in [0,2\pi]$ colored by green. 
Each fiber is spanned by the stable eigenvector $s(\theta)$ colored by blue.
\par
\end{figure}

\begin{center}
  \begin{table}[h]
    \begin{center}
      \begin{tabular}{|c|c|c|} \hline
	$\theta$ & $\lambda_1(\theta)$	 & $u_1(\theta)$ \\ \hline
	$0$ & $[0.999992741958,1.000006791375]$ &  $\begin{pmatrix} [-6.673440138039\times 10^{-4},6.772434334871\times 10^{-7}]\\ [0.9999979331077,1.000002155782]\end{pmatrix}$\\
	$\pi/2$ & $1.4143_{05544072}^{25544075}$ &  $\begin{pmatrix} 0.577_{0155092804}^{8514435909}\\ -0.816_{7315376799}^{1441838033}\end{pmatrix}$ \\
	$\pi$ & $[1.999994416494, 2.000003915277]$ & $\begin{pmatrix} [0.9999977581224,1.000002285772] \\ [-1.063112822102\times 10^{-3}, 2.72908020214\times 10^{-4}]  \end{pmatrix}$ \\
	$3\pi/2$ & $1.413_{898226973}^{918226976}$ & $\begin{pmatrix} 0.57_{66828025463}^{75189566192}\\ 0.816_{3794854755}^{9664974717}\end{pmatrix}$  \\ \hline
	$\theta$ & $\lambda_2(\theta)$	 & $u_2(\theta)$ \\ \hline
	$0$ & $[-2.00000489336,-1.999992839973]$ & $\begin{pmatrix} [0.9999974324377,1.000002323118] \\  [-1.345486294031\times 10^{-6},1.334678804813\times 10^{-3}] \end{pmatrix}$  \\
	$\pi/2$ & $-1.4141_{21870851}^{01870848}$ & $\begin{pmatrix} 0.57_{684921267}^{76851432715}\\ 0.816_{2618998034}^{8490032262} \end{pmatrix}$  \\
	$\pi$ & $[-1.000005814453,-0.9999935855922]$ & $\begin{pmatrix} [-1.364580534662\times 10^{-4},5.31560477539\times 10^{-4}] \\ [0.9999980145027,1.000002146471]\end{pmatrix}$  \\
	$3\pi/2$ & $-1.4145_{29246614}^{09246611}$ & $\begin{pmatrix} 0.57_{71816924196}^{80178576221}\\ -0.816_{6141008822}^{0263373207} \end{pmatrix}$  \\ \hline
      \end{tabular}
    \end{center}
    \caption{Several validated eigenpairs in Computer Assisted Result \ref{car-Fenichel}.}
    \label{table-Fenichel}
  $(\lambda_1(\theta), u_1(\theta))$ denotes the parameter family of eigenpairs associated with unstable eigenvalue.
  $(\lambda_2(\theta), u_2(\theta))$ denotes the parameter family of eigenpairs associated with stable eigenvalue.
  \end{table}
\end{center}

\subsection{The FitzHugh-Nagumo system}
\label{section-FN}
The second example is the FitzHugh-Nagumo system
\begin{equation}
\label{FN}
\begin{cases}
u' = v & \text{}\\
v' = \delta^{-1}(cv - f(u) + w) & \\
w' = \epsilon c^{-1}(u-\gamma w),
\end{cases}
\end{equation}
where $a\in (0,1/2)$, $c,\gamma$ and $\delta$ are positive parameters, and $f(u) = u(u-a)(1-u)$. (\ref{FN}) is well-known as the system of traveling wave solutions $(U,W) = (\psi_U(x-ct), \psi_W(x-ct))$ of the following partial differential equation: 
\begin{equation}
\label{FN-PDE}
\begin{cases}
U_t = U_{xx} + f(U) - W &\\
W_t = \epsilon(U-\gamma W) &
\end{cases},\quad t>0,\ x\in \mathbb{R}.
\end{equation}
Validation of slow manifolds for (\ref{FN}) with explicit ranges of $\epsilon$ are discussed in \cite{CZ, Mat2}.
In \cite{Mat2}, fast-saddle-type blocks are constructed {\em independently} on intervals $[y_j^-,y_j^+]$ centered at sample points $y_j$, which violates the smoothness of the union of blocks.
On the other hand, isolating segments, which is a counterpart of fast-saddle-type blocks in \cite{CZ}, are constructed which forms an $h$-set.
Here we focus on how we construct smooth neighborhoods, say $h$-sets, of {\em nonlinear} slow manifolds {\em with arbitrary length systematically}.
A series of our validation procedures stated in Section \ref{section-main} gives an answer to this problem; that is, 
smooth neighborhoods can be extended in arbitrary length as long as Algorithms \ref{alg-bundle} and \ref{alg-nbh} return {\tt true}.
In the following result, we only set parameters including the radii $\eta^u, \eta^s$ of our desire, approximate initial equilibria (one equilibrium for each branch of $S_0$) in advance.

\begin{car}
\label{car-FN}
Consider (\ref{FN}).
Set $a = 0.3$, $\gamma = 10.0$ and $\delta = 9.0$.
For all $c\in [0.799, 0.801]$ and $\epsilon\in [0,1.0\times 10^{-4}]$, we validate a tubular neighborhood centered at the slow manifold $S_\epsilon$ near or on the portion of nullcline $S=\{f(u) = w\}\cap \{w\in [-0.0002, 0.08]\}$ with radii $\eta^u = \eta^s = 1.0\times 10^{-3}$. 
Slope $\{M^\alpha\}_{\alpha=u,s}$ and length $\{l^\alpha\}_{\alpha=u,s}$ of cones for validating conic and star-shaped neighborhoods are
\begin{align*}
M^u = 5,\ M^s = 10,\ l^u = 0.01,\ l^s = 0.09 &\quad \text{ for branch containing }(u,w) = (0,0),\\
M^u = 4,\ M^s = 6,\ l^u = 0.008,\ l^s = 0.08 &\quad \text{ for branch containing }(u,w) = (1,0).
\end{align*}
Validated tubular and star-shaped neighborhoods are shown in Figures \ref{fig-fn-tube1} - \ref{fig-star-fn}.
\par
Moreover, the leftmost branch of $S_\epsilon$ is $C^{17}$ and the rightmost branch of $S_\epsilon$ is $C^7$ for all $\epsilon\in [0,1.0\times 10^{-4}]$.
\end{car}
The parameter values in Computer Assisted Result \ref{car-Fenichel} is those for validating homoclinic orbits for (\ref{FN}) with $\epsilon > 0$ in \cite{Mat2}.
\par
Smoothness validations in terms of rate conditions (Figure \ref{fig-fn-smoothness}) estimate normal hyperbolicity of slow manifolds.
From Figure \ref{fig-fn-smoothness}, we expect that slow manifolds as well as their tubular neighborhoods can be extended arbitrarily in suitable directions.
Indeed, we also obtain the following result, for example.
What we changed from Computer Assisted Result \ref{car-FN} is just the number of iterations $m_0$.

\begin{car}
\label{car-FN-2}
Under the same settings as Computer Assisted Result \ref{car-FN}, for all $c\in [0.799, 0.801]$ and $\epsilon\in [0,1.0\times 10^{-4}]$, we validate a tubular neighborhood centered at the slow manifold $S_\epsilon$ near or on the portion of nullcline $S=\{f(u) = w\}\cap \{u\leq 0.01, w\in [-0.002, 82.4716]\}$ with radii $\eta^u = \eta^s = 1.0\times 10^{-3}$. 
\end{car}

\begin{figure}[htbp]\em
\begin{minipage}{1.0\hsize}
\centering
\includegraphics[width=9.0cm]{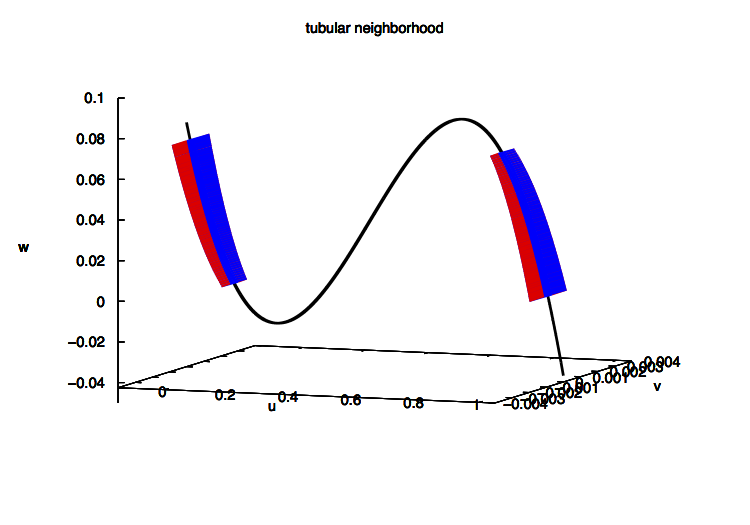}
\end{minipage}
\caption{Tubular neighborhood of slow manifolds for (\ref{FN})}
\label{fig-fn-tube1}
Black curve represents the nullcline $f(u)-w = 0$.
Two tubular neighborhoods are validated around $(u,v)=(0,0)$ and $(1,0)$.
Red surfaces are fast exit. Blue surfaces are fast entrance.
\end{figure}

\begin{figure}[htbp]\em
\begin{minipage}{0.5\hsize}
\centering
\includegraphics[width=7.5cm]{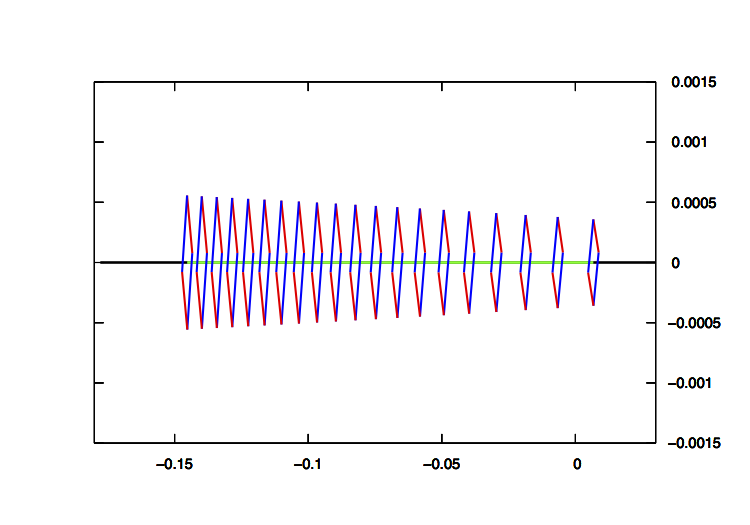}
(a)
\end{minipage}
\begin{minipage}{0.5\hsize}
\centering
\includegraphics[width=7.5cm]{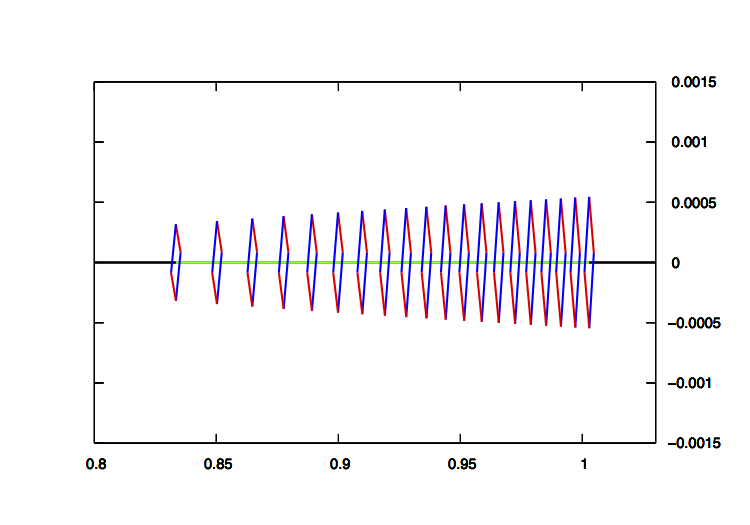}
(b)
\end{minipage}
\caption{Tubular neighborhood of slow manifolds for (\ref{FN}) : Projections of sections $\{w=\text{const}.\}$}
\label{fig-fn-tube2}
(a) : Sections of tubular neighborhoods around $(u,v) = (0,0)$. 
(b) : Sections of tubular neighborhoods around $(u,v) = (1,0)$. 
Red surfaces are fast exit. Blue surfaces are fast entrance.
\end{figure}

\begin{figure}[htbp]\em
\begin{minipage}{0.5\hsize}
\centering
\includegraphics[width=7.0cm]{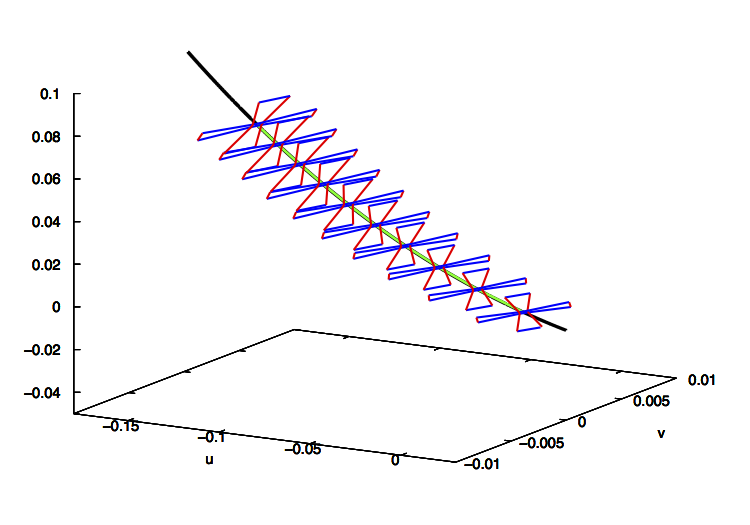}
(a)
\end{minipage}
\begin{minipage}{0.5\hsize}
\centering
\includegraphics[width=7.0cm]{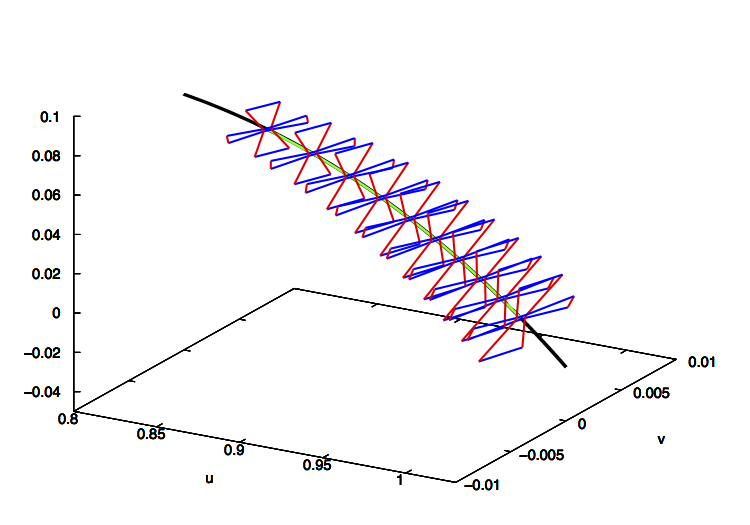}
(b)
\end{minipage}
\begin{minipage}{0.5\hsize}
\centering
\includegraphics[width=7.0cm]{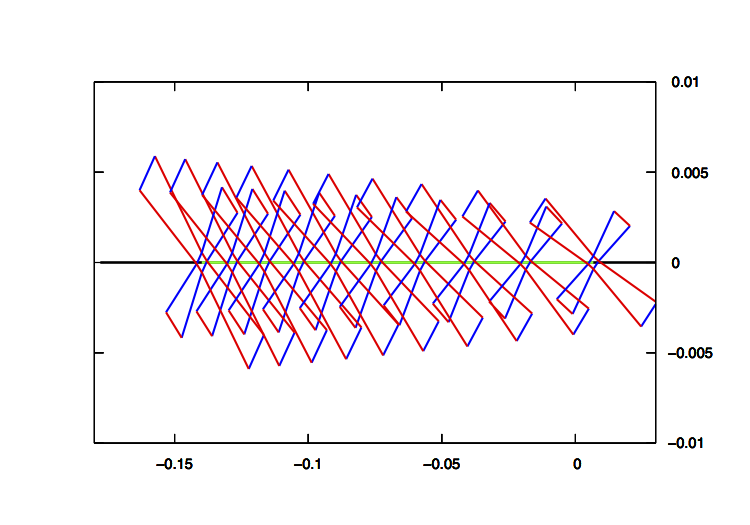}
(c)
\end{minipage}
\begin{minipage}{0.5\hsize}
\centering
\includegraphics[width=7.0cm]{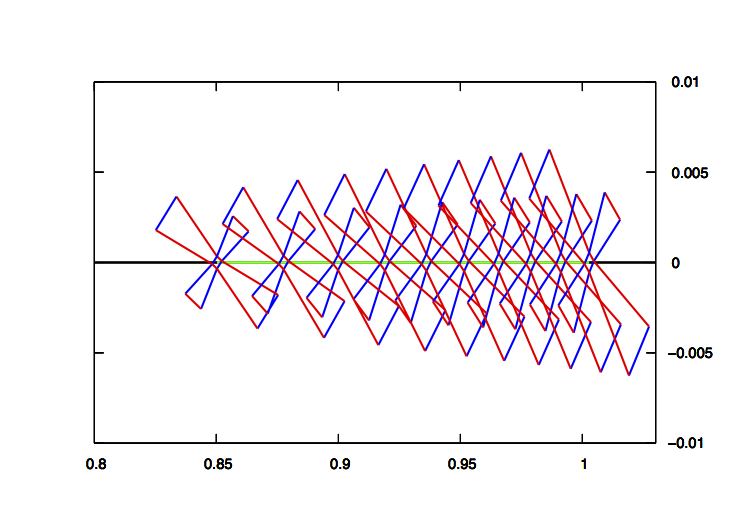}
(d)
\end{minipage}
\caption{Fast-saddle-type stars (star-shaped neighborhoods) of slow manifolds for (\ref{FN})}
\label{fig-star-fn}
(a) : Sections of star-shaped neighborhoods around $(u,v)=(0,0)$: the 3D plot. 
\par
(b) : Sections of star-shaped neighborhoods around $(u,v)=(1,0)$: the 3D plot.
 \par
(c) : Sections of star-shaped neighborhoods around $(u,v)=(0,0)$: the 2D plot.
\par
(d) : Sections of star-shaped neighborhoods around $(u,v)=(1,0)$: the 2D plot.
\par
The central blocks of stars correspond to slices of tubular neighborhoods shown in Figs. \ref{fig-fn-tube1}- \ref{fig-fn-tube2}.
Red surfaces are fast exit. Blue surfaces are fast entrance.
All star-shaped domains drawn here are validated ones replacing $l^s$ in Computer Assisted Result \ref{car-FN-2} by $0.02$.
\end{figure}

\begin{center}
  \begin{table}[h]
    \begin{center}
      \begin{tabular}{|c|c|c|} \hline
	$(u,w)$ & $\begin{pmatrix} [-1.041095921349\times 10^{-6}, 3.362281652217\times 10^{-4}] \\ [-1.0\times 10^{-4}, 0.0] \end{pmatrix}$ & $\begin{pmatrix} -0.1_{400372866998}^{393648241443} \\ [0.0699, 0.0700] \end{pmatrix}$ \\ \hline
	$\lambda_1$ & $0.23_{13774849585}^{30656780011}$ & $0.33_{03365343961}^{18664771135}$ \\
	$u_1$ & $\begin{pmatrix} 0.97_{36944780339}^{44663427829} \\ 0.22_{53340238595}^{70709196773} \end{pmatrix}$  & $\begin{pmatrix} 0.94_{89022030275}^{97320044481} \\ 0.31_{35775665615}^{50630234213}  \end{pmatrix}$ \\
	$\lambda_2$ & $-0.14_{41882997479}^{2476985692}$ & $-0.24_{29806561325}^{14444659639}$ \\
	$u_2$ & $\begin{pmatrix} -0.9_{901936675038}^{895735154845} \\ 0.14_{10723669084}^{26928675384} \end{pmatrix}$  & $\begin{pmatrix} -0.97_{22545527814}^{1540037228} \\ 0.23_{46861516551}^{61252670818}
 \end{pmatrix}$\\
	 \hline
	$(u,w)$ &  $\begin{pmatrix} [0.9997685653125,1.000577904404] \\ [-1.0\times 10^{-4}, 0.0] \end{pmatrix}$ &  $\begin{pmatrix} 0.85_{03581502084}^{12217617279} \\ [0.0699, 0.0700] \end{pmatrix}$ \\ \hline
	$\lambda_1$ & $0.32_{61480456307}^{277846534857}$ & $0.2_{187758769091}^{21146299268}$ \\
	$u_1$ & $\begin{pmatrix} 0.950_{0552084232}^{9113637252} \\ 0.3_{099857738407}^{115662745882} \end{pmatrix}$ & $\begin{pmatrix} 0.97_{62048927254}^{71002135111} \\ 0.21_{36346494841}^{60164344356} \end{pmatrix}$ \\
	$\lambda_2$ & $-0.23_{88987929435}^{72560025185}$ & $-0.1_{322689543302}^{298752159966}$ \\
	$u_2$ & $\begin{pmatrix} -0.97_{31768060885}^{24436494073} \\ 0.23_{08352040574}^{23730048381} \end{pmatrix}$ & $\begin{pmatrix} -0.991_{8632571519}^{1752943809} \\ 0.1_{288226770962}^{310982954026} \end{pmatrix}$ \\
	 \hline
      \end{tabular}
    \end{center}
    \caption{Several validated data associated with Computer Assisted Result \ref{car-FN}.}
    \label{table-FN}
 As in Table \ref{table-Fenichel}, $(\lambda_1, u_1)$ denotes the parameter family of eigenpairs at $(u,w)$ associated with unstable eigenvalue, while $(\lambda_2, u_2)$ denotes the parameter family of eigenpairs  at $(u,w)$ associated with stable eigenvalue.
 
  \end{table}
\end{center}

\begin{figure}[htbp]\em
\begin{minipage}{0.5\hsize}
\centering
\includegraphics[width=7.5cm]{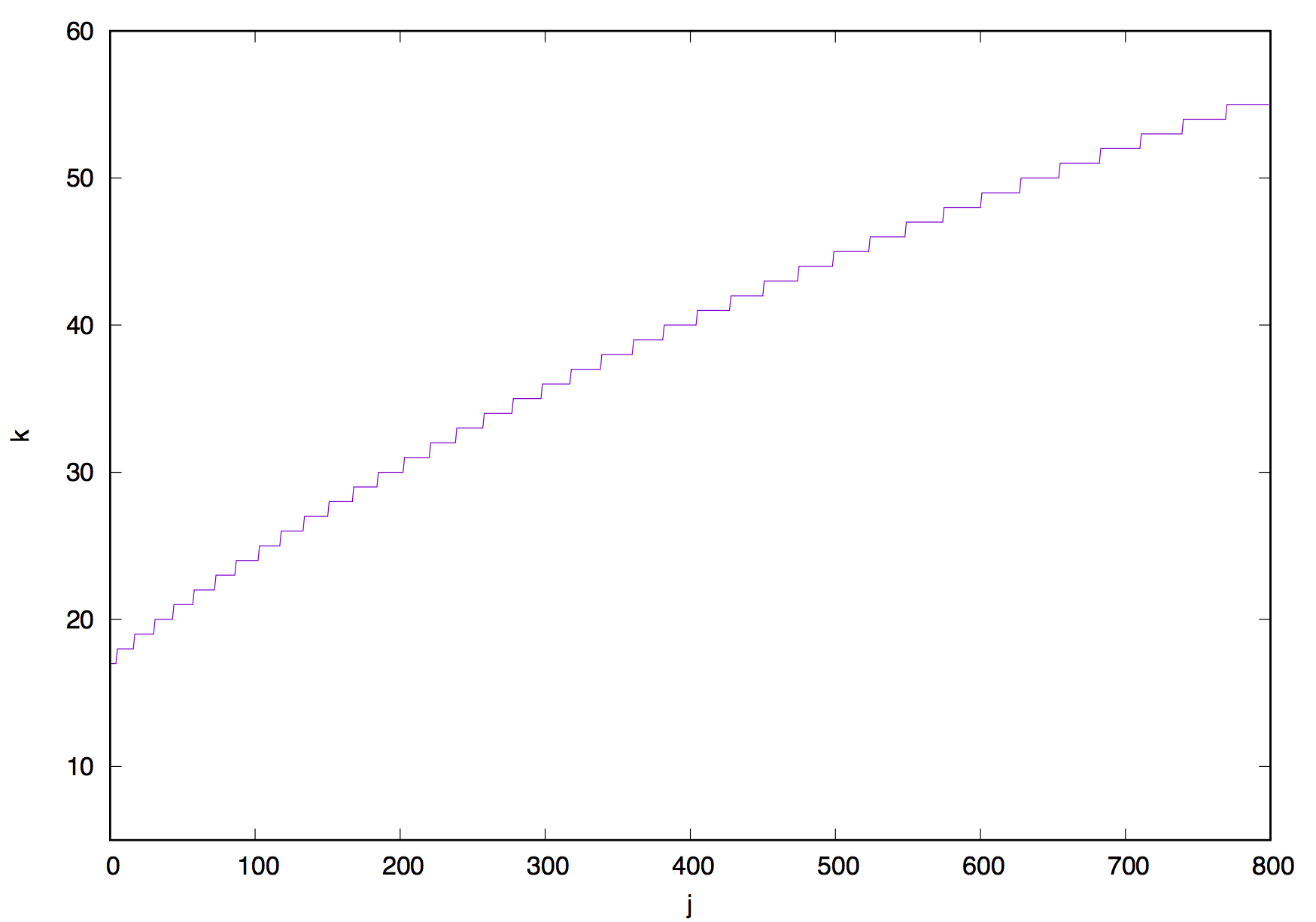}
(a)
\end{minipage}
\begin{minipage}{0.5\hsize}
\centering
\includegraphics[width=7.5cm]{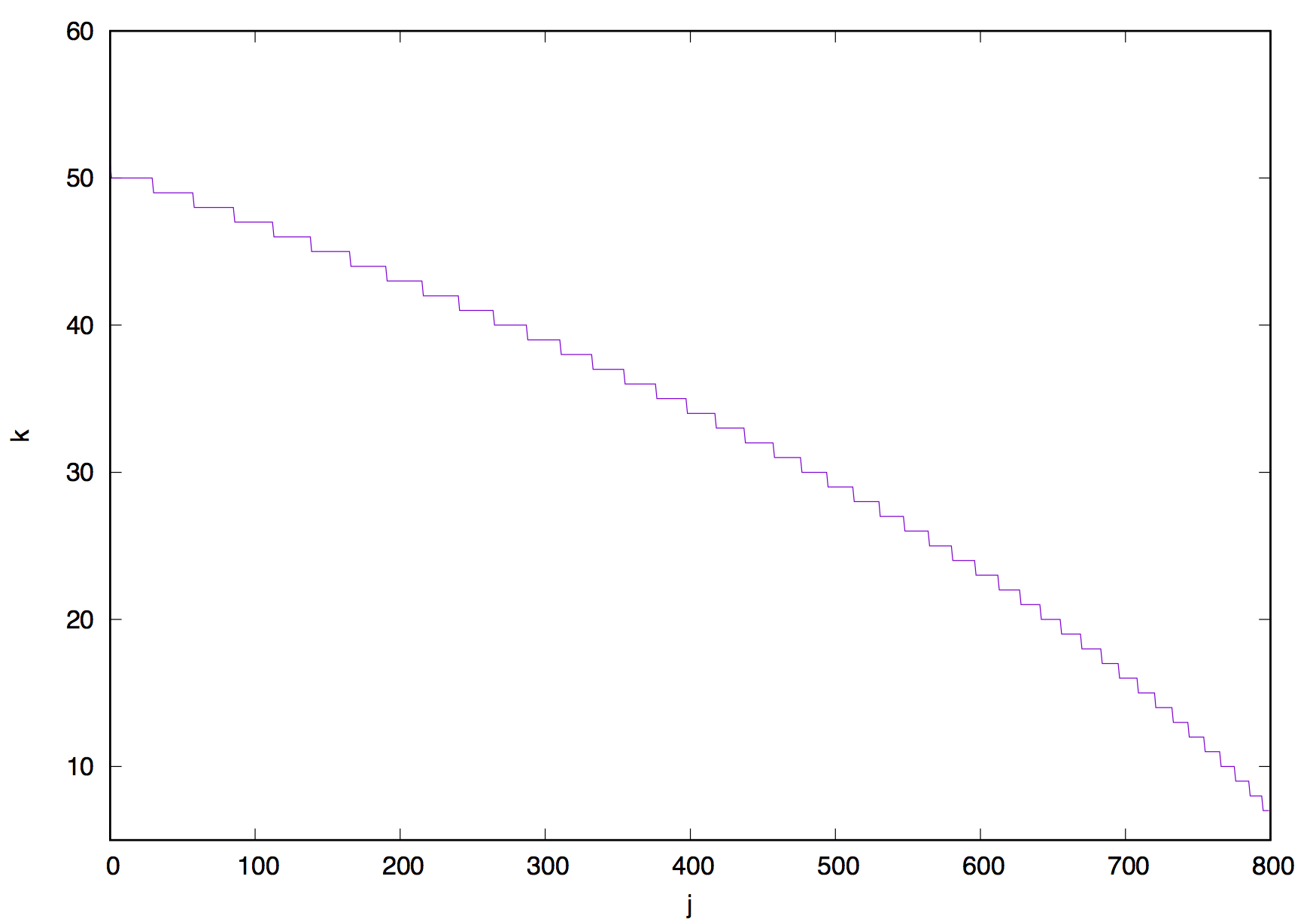}
(b)
\end{minipage}
\caption{Validated local smoothness of $S^\epsilon$ for (\ref{FN})}
\label{fig-fn-smoothness}
(a) : Local smoothness of $S_\epsilon$ near $u=0$. (b) : Local smoothness of $S_\epsilon$ near $u=1$.
\par
In both graphs, the horizontal axis is $w$ and the vertical axis is $k$, where 
$k = \lfloor \min \{k_{su,j}, k_{ss,j}\} \rfloor - 1$
in the target block $\mathcal{D}_j$ with $\pi_y(\mathcal{D}_j) = \{-0.002\} + [j\Delta,(j+1)\Delta]$, $\Delta = 1.0\times 10^{-4}$.
These graphs indicate that the slow manifolds lose their smoothness near fold points $(u_{fold},w_{fold})\approx (0.137060186089, -0.0192716562198), (0.729606480577, 0.0847531377013)$, as described in the geometric singular perturbation theory including non-hyperbolic points (e.g., \cite{KS}).
\end{figure}

\subsection{The predator-prey system}
\label{section-PP}
The third example is the predator-prey system
\begin{equation}
\label{ODE-pp}
\begin{cases}
\dot u = w, &\\
\dot w = -\theta w - uf_1(u,v), &\\
\dot v = \epsilon z, &\\
\dot z = -\epsilon (\theta z + vf_2(u,v)),&
\end{cases}
\end{equation}
where
\begin{equation}
\label{pp-per}
f_1(u,v) := (1-u)(u-v),\quad f_2(u,v) := au-b-v.
\end{equation}
This system is regarded as the traveling wave system of the partial differential equations with the following form:
\begin{equation}
\label{pred-prey}
\begin{cases}
\displaystyle \epsilon \frac{\partial u_1}{\partial t} = \epsilon^2 \frac{\partial^2 u_1}{\partial x^2} + u_1f_1(u), & \\
\displaystyle \frac{\partial u_2}{\partial t} = \frac{\partial^2 u_2}{\partial x^2} + u_2f_2(u), & \\
\end{cases}\quad t>0,\ x\in \mathbb{R},
\end{equation}
where $u = (u_1,u_2)$ is restricted to the nonnegative quadrant $\{u_1, u_2 \geq 0\}$.
The nonlinearities $f_1$ and $f_2$ satisfy suitable assumptions reflecting biological phenomena, which are not stated here. 

The reaction-diffusion system (\ref{pred-prey}) is motivated by predator-prey system from ecology, in which case $u_1$ and $u_2$ represents the living predator and prey, respectively.
This model is an example of (modified) {\em Rosenzweig-MacArthur equations}, and the existence of periodic traveling wave solutions of this system is considered in \cite{GS}, which is based on the Conley index theory.
Arguments in \cite{GS} are revisited in \cite{GGKKMO} and several trajectories of (\ref{pred-prey}) are validated via rigorous numerics for sufficiently small $\epsilon > 0$.
The aim of this section is to validate smooth tubular, conic and star-shaped neighborhoods of slow manifolds with an explicitly given range $[0,\epsilon_0]$.
The validation indicates that our validation procedure is applicable to systems with multi-dimensional fast and slow variables.
Our basic verification strategy is exactly same as preceding subsections.

\begin{car}
\label{car-PP}
Consider (\ref{ODE-pp}).
Set $a = 1.65$, $b = 0.25$ and $\theta = -0.25$.
Then, for all $\epsilon\in [0,1.0\times 10^{-4}]$, tubular neighborhood centered at the slow manifold $S_\epsilon$ near or on the portion of nullcline $S=\{u=0,1, w=0\}$ are validated in the slow variable range $(v,z)\in [0.2, 0.8]\times [-0.6, 0.2]$. 
Validated radii of the neighborhood near $u=0$ are $\eta^u = \eta^s = 1.3\times 10^{-4}$.
Validated radii of the neighborhood near $u=1$ are $\eta^u = \eta^s = 1.5\times 10^{-4}$.
\par
Conic and star-shaped neighborhoods of $S_\epsilon$ are also validated with slope $\{M^\alpha\}_{\alpha=u,s}$ and length $\{l^\alpha\}_{\alpha=u,s}$ of cones given by
\begin{align*}
M^u = 1.1,\ M^s = 1.1,\ l^u = 0.001,\ l^s = 0.005 &\quad \text{ for branch containing }(u,w) = (0,0),\\
M^u = 1.1,\ M^s = 1.1,\ l^u = 0.001,\ l^s = 0.005 &\quad \text{ for branch containing }(u,w) = (1,0).
\end{align*}
Moreover, the slow manifold near $u=0$ is $C^8$ and the slow manifold near $u=1$ is $C^4$.
\end{car}

The local smoothness of slow manifolds as well as their (un)stable manifolds is listed in Figure \ref{fig-PP-smoothness} and, as in Computer Assisted Result \ref{car-FN}, the smoothness of slow manifolds stated in Computer Assisted Result \ref{car-PP} is put as the minimum of these validated smoothness indices.
The right graph in Figure \ref{fig-PP-smoothness} indicates that $S_\epsilon$ near $u=1$ loses smoothness near $v=1$.
Indeed, direct calculations show that the eigenvalue of the fast component of linearized vector field (\ref{ODE-pp}) at $(u,w) = (1,0)$ is 
\begin{equation*}
\begin{pmatrix}
0 & 1 \\
3u^2 - 2(v+1)u + v  & -\theta
\end{pmatrix}_{(u,w)=(1,0), \theta = -0.25}
=
\begin{pmatrix}
0 & 1 \\
1-v  & 0.25
\end{pmatrix}.
\end{equation*}
Eigenvalues are $\lambda = \frac{1}{8} \pm \frac{1}{2}\sqrt{\frac{1}{16}-4(1-v)}$, which becomes $0$ at $v=1$. 
This indicates that normal hyperbolicity of $S_\epsilon$ near $u=1$ breaks near $v=1$.
Our result in Figure \ref{fig-PP-smoothness} reflects this observation.
Also note that eigenvalues become complex in $v \geq 65/64$.

\begin{figure}[htbp]\em
\begin{minipage}{0.5\hsize}
\centering
\includegraphics[width=7.5cm]{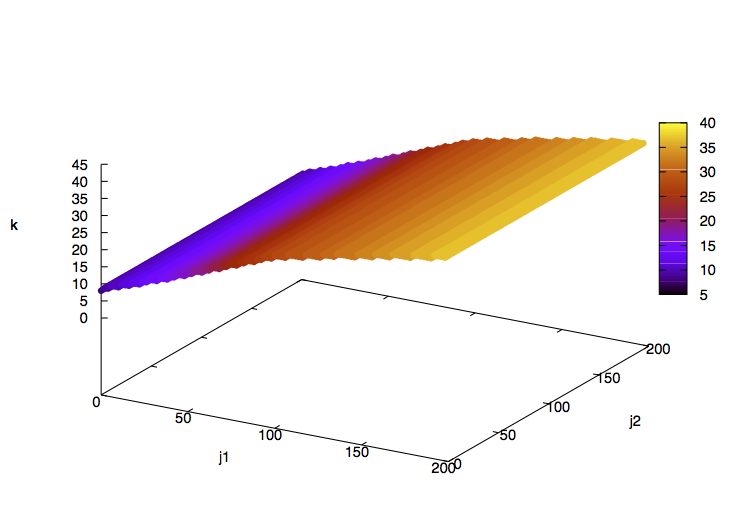}
(a)
\end{minipage}
\begin{minipage}{0.5\hsize}
\centering
\includegraphics[width=7.5cm]{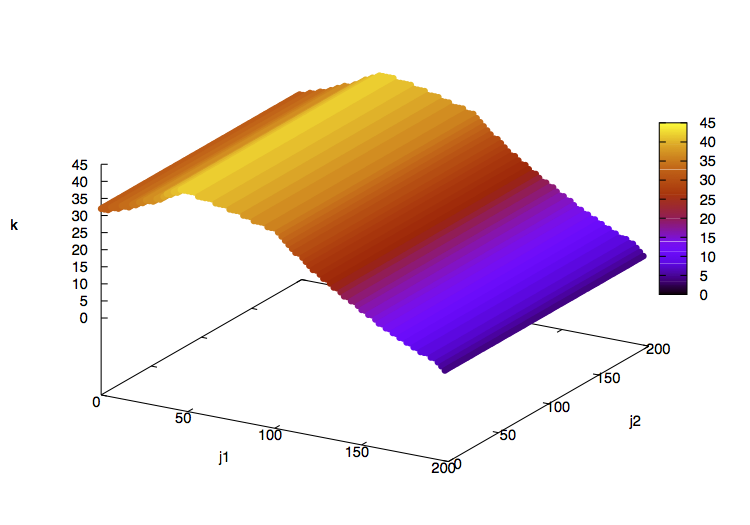}
(b)
\end{minipage}
\caption{Validated local smoothness of $S_\epsilon$ for (\ref{ODE-pp})}
\label{fig-PP-smoothness}
(a) : Local smoothness of $S_\epsilon$ near $u=0$. (b) : Local smoothness of $S_\epsilon$ near $u=1$.
\par
In both graphs, the height describes the smoothness $k$ of $S_\epsilon$ at ${\bf j} = (j_1,j_2)$, where 
$k = \lfloor \min \{k_{su, {\bf j} }, k_{ss, {\bf j}}\} \rfloor - 1$
in the target block $N_{\bf j}$ with $\pi_y(N_{\bf j}) = \{(0.2, -0.6)\} + [j_1\Delta_v,(j_1+1)\Delta_v]\times  [j_2\Delta_z,(j_2+1)\Delta_z]$, $\Delta_v = 0.003, \Delta_z = 0.004$.
These graphs indicate that, as in Figure \ref{fig-fn-smoothness}, the slow manifolds lose their smoothness near a non-hyperbolic curve $\{(u,w,v,z)\mid (u_{nh},w_{nh}, v_{nh})= (1,0,1)\}$.
\end{figure}

\begin{center}
  \begin{table}[h]
    \begin{center}
      \begin{tabular}{|c|c|} \hline
	Computer Assisted Result & Computation time \\ \hline
	\ref{car-Fenichel} & $7.109$ sec. \\
	\ref{car-FN} & $3.552$ sec. \\
	\ref{car-FN-2} & $26$ min. $10.872$ sec. \\
	\ref{car-PP} & $2$ min. $53.418$ sec. \\ \hline
      \end{tabular}
    \end{center}
    \caption{Total computation times of Computer Assisted Results.}
    \label{table-time}
    Computation time for \ref{car-FN} and \ref{car-PP} sums up corresponding times for two branches of slow manifolds.
    Note that, in Computer Assisted Results \ref{car-Fenichel}, \ref{car-FN}, \ref{car-FN-2} and \ref{car-PP}, there are totally $m_0 = 3142, 800, 824736, 40000$ iterations of calculations in Algorithms \ref{alg-bundle} and \ref{alg-nbh}, respectively.
    We easily see that the number of iterations $m_0$ directly affects computation times.
  \end{table}
\end{center}

\section{Conclusion}
In this paper, we have discussed a validation method to construct isolating blocks continuously depending on points on slow manifolds with computer assistance.
Essential arguments in our methodology are summarized as follows:
\begin{itemize}
\item Validations of continuous families of eigenpairs of $f_x(h_\epsilon(y),y,\epsilon)$ on slow manifolds $S_\epsilon = \{x=h_\epsilon(y)\}$.
\item Validations of tubular, conic and star-shaped neighborhoods centered at slow manifolds with computer assistance.
\item Smoothness of slow manifolds as well as tubular neighborhoods with the help of rate conditions.
\end{itemize}
Our procedure realizes nonlinear (diffeomorphic) transformations of tubular neighborhoods of slow manifolds as well as manifolds themselves.
Note that our validations of continuous families of eigenpairs of $f_x(h_\epsilon(y),y,\epsilon)$ on slow manifolds $S_\epsilon = \{x=h_\epsilon(y)\}$ yields vector bundles $V^u_\epsilon$ and $V^s_\epsilon$ over $S_\epsilon$ which reflects the normal hyperbolicity of $S_\epsilon$.
Tubular neighborhood validations like ours will be the basis of computer assisted analysis of fast-slow systems when we want to apply general and abstract arguments in geometric singular perturbation theory to concrete systems.
\par
\bigskip
We conclude this paper providing several applications and perspectives of the current work.
\par
One is the validation of global trajectories of fast-slow systems with multi-dimensional slow variables with an explicit range of $\epsilon$, which is one of our essential motivations of this work as the sequel to \cite{Mat2}.
A numerical validation of global trajectories for (\ref{fast-slow})$_\epsilon$ with an explicit range $[0,\epsilon_0]$ of $\epsilon$ is discussed there.
A topological concept called {\em slow shadowing} plays a role to validate trajectories which shadow slow manifolds.
This concept measures the strength of normal hyperbolicity relative to the speed of slow dynamics.
An essential estimate in this notion is constructed in locally constructed fast-saddle-type blocks, while there is a trade-off of validations, which is mainly because of {\em non-smooth} attachments of local blocks.
Such a trade-off prevents us from validating true trajectories with a range $[0,\epsilon_0]$ with large $\epsilon_0$.
We believe that our present results overcome this difficulty because tubular neighborhoods of slow manifolds are viewed as {\em smooth} attachments of local blocks.
Indeed, there is a numerical validation evidence in \cite{CZ} that (piecewise) smooth neighborhoods of slow manifolds enable us to validate trajectories for (\ref{fast-slow})$_\epsilon$ for $\epsilon \in (0,\epsilon_0]$ with large $\epsilon_0$ enough to bridge the standard analysis.
However, a choice of candidate blocks in \cite{CZ} requires artificial trail and error, which makes applications of preceding works to multi-dimensional slow variables very difficult.
On the other hand, our proposing method does not need heuristic processes for constructing fast-saddle-type blocks under the suitable assumption to eigenvalues.
Our result in this paper will thus bridge ideas in preceding works to fast-slow systems with multi-dimensional slow variables in a systematic way.
\par
We also note that our results in this paper contain constructions of vector bundles with computer assistance.
It is then natural to consider topological invariants of such bundles, say {\em characteristic classes}. 
In our case, these topological invariants measure how twist slow manifolds are.
Twists of bundles over invariant manifolds can involve bifurcations of infinitely many global trajectories (e.g. \cite{D1991twist, D1991traveling, D1993}) and is of great importance for understanding global dynamics.
Note that numerical validation of vector bundles over trajectories are already presented in, e.g. \cite{CL2013, CLM2015} with different motivations from ours.
A well-known application of characteristic classes to dynamical systems is the {\em stability index} (e.g. \cite{Jones1984}) relating to stability of self-similar (traveling) wave solutions $V(\xi)$ of reaction diffusion systems with the linearized operator $L$ along $V(\xi)$.
\par
Finally, we put a comment in a direction to higher order parameterizations of slow manifolds.
As indicated in Section \ref{section-param}, our validated change of coordinates is considered as a parameterization of slow manifolds as well as their (un)stable manifolds up to the linear order term.
Readers who are familiar with parameterization method expect that this parameterization can be generalized to that with {\em higher order terms}.
There are several preceding studies of numerical computation, possibly with rigorous numerics, of invariant manifolds (e.g., \cite{vdBMM2013}), and these studies open the door to rigorous numerics of normal forms around invariant manifolds.
As for the {\em Fenichel normal form} (\ref{Fenichel-normal-form}) well-known in the theory of fast-slow systems, we need to straighten slow manifolds, their (un)stable manifolds as well as all fibers.
We believe that parameterization method opens the door to rigorous numerics of Fenichel normal forms with various numerical applications to fast-slow systems, and that our procedure presented here is a basis of this direction from the  geometric viewpoint.

\section*{Acknowledgements}
This research was partially supported by Coop with Math Program (The Institute of Statistical Mathematics), a commissioned project by MEXT, Japan. 
The author thanks to Prof. Freddy Dumortier and Jason D. Mireles-James for giving him helpful suggestions of current and further directions of this research.

\bibliographystyle{plain}
\bibliography{smooth_block}

\begin{thebibliography}{10}

\bibitem{BLZ2000}
P.~Bates, K.~Lu, and C.~Zeng.
\newblock Invariant foliations near normally hyperbolic invariant manifolds for
  semiflows.
\newblock {\em Transactions of the American Mathematical Society},
  352(10):4641--4676, 2000.

\bibitem{BKM2007}
E.~Boczko, W.D. Kalies, and K.~Mischaikow.
\newblock Polygonal approximation of flows.
\newblock {\em Topology and its Applications}, 154(13):2501--2520, 2007.

\bibitem{CFdlL2003}
X.~Cabr{\'e}, E.~Fontich, and R.~de~la Llave.
\newblock The parameterization method for invariant manifolds {I}: manifolds
  associated to non-resonant subspaces.
\newblock {\em Indiana University mathematics journal}, 52(2):283--328, 2003.

\bibitem{CFdlL2005}
X.~Cabr{\'e}, E.~Fontich, and R.~de~la Llave.
\newblock The parameterization method for invariant manifolds {III}: overview
  and applications.
\newblock {\em Journal of Differential Equations}, 218(2):444--515, 2005.

\bibitem{CZ2015}
M.J. Capi{\'n}ski and P.~Zgliczy{\'n}ski.
\newblock Geometric proof for normally hyperbolic invariant manifolds.
\newblock {\em Journal of Differential Equations}, 259(11):6215--6286, 2015.

\bibitem{CZ2016}
M.J. Capi\'{n}ski and P.~Zgliczy\'{n}ski.
\newblock Beyond the {M}elnikov method: a computer assisted approach.
\newblock {\em arXiv preprint arXiv:1603.07131}, 2016.

\bibitem{CL2013}
R.~Castelli and J.-P. Lessard.
\newblock Rigorous numerics in {F}loquet theory: {C}omputing stable and
  unstable bundles of periodic orbits.
\newblock {\em SIAM Journal on Applied Dynamical Systems}, 12(1):204--245,
  2013.

\bibitem{CLM2015}
R.~Castelli, J.-P. Lessard, and J.D. Mireles~James.
\newblock Parameterization of invariant manifolds for periodic orbits {I}:
  {E}fficient numerics via the {F}loquet normal form.
\newblock {\em SIAM Journal on Applied Dynamical Systems}, 14(1):132--167,
  2015.

\bibitem{Con}
C.~Conley.
\newblock {\em Isolated invariant sets and the {M}orse index}, volume~38 of
  {\em CBMS Regional Conference Series in Mathematics}.
\newblock American Mathematical Society, Providence, R.I., 1978.

\bibitem{CZ}
A.~Czechowski and P.~Zgliczy{\'n}ski.
\newblock Existence of {P}eriodic {S}olutions of the {F}itz{H}ugh-{N}agumo
  {E}quations for an {E}xplicit {R}ange of the {S}mall {P}arameter.
\newblock {\em arXiv preprint arXiv:1502.02451}, 2015.

\bibitem{D1991twist}
B.~Deng.
\newblock The bifurcations of countable connections from a twisted heteroclinic
  loop.
\newblock {\em SIAM journal on Mathematical Analysis}, 22(3):653--679, 1991.

\bibitem{D1991traveling}
B.~Deng.
\newblock The existence of infinitely many traveling front and back waves in
  the {F}itz{H}ugh-{N}agumo equations.
\newblock {\em SIAM J. Math. Anal.}, 22(6):1631--1650, 1991.

\bibitem{D1993}
B.~Deng.
\newblock Homoclinic twisting bifurcations and cusp horseshoe maps.
\newblock {\em Journal of Dynamics and Differential equations}, 5(3):417--467,
  1993.

\bibitem{F1973}
N.~Fenichel.
\newblock Asymptotic stability with rate conditions.
\newblock {\em Indiana Univ. Math. J}, 23(1109-1137):74, 1973.

\bibitem{F1977}
N.~Fenichel.
\newblock Asymptotic stability with rate conditions. 2.
\newblock {\em Indiana University Mathematics Journal}, 26(1):81--93, 1977.

\bibitem{F1979}
N.~Fenichel.
\newblock Geometric singular perturbation theory for ordinary differential
  equations.
\newblock {\em J. Differential Equations}, 31(1):53--98, 1979.

\bibitem{GGKKMO}
M.~Gameiro, T.~Gedeon, W.~Kalies, H.~Kokubu, K.~Mischaikow, and H.~Oka.
\newblock Topological horseshoes of traveling waves for a fast-slow
  predator-prey system.
\newblock {\em J. Dynam. Differential Equations}, 19(3):623--654, 2007.

\bibitem{GS}
R.~Gardner and J.~Smoller.
\newblock The existence of periodic travelling waves for singularly perturbed
  predator-prey equations via the {C}onley index.
\newblock {\em J. Differential Equations}, 47(1):133--161, 1983.

\bibitem{GJM2012}
J.~Guckenheimer, T.~Johnson, and P.~Meerkamp.
\newblock Rigorous enclosures of a slow manifold.
\newblock {\em SIAM Journal on Applied Dynamical Systems}, 11(3):831--863,
  2012.

\bibitem{GK}
J.~Guckenheimer and C.~Kuehn.
\newblock Computing slow manifolds of saddle type.
\newblock {\em SIAM J. Appl. Dyn. Syst.}, 8(3):854--879, 2009.

\bibitem{Jones1984}
C.K.R.T. Jones.
\newblock Stability of the travelling wave solution of the
  {F}itz{H}ugh-{N}agumo system.
\newblock {\em Trans. Amer. Math. Soc.}, 286(2):431--469, 1984.

\bibitem{Jones}
C.K.R.T. Jones.
\newblock Geometric singular perturbation theory.
\newblock In {\em Dynamical systems ({M}ontecatini {T}erme, 1994)}, volume 1609
  of {\em Lecture Notes in Math.}, pages 44--118. Springer, Berlin, 1995.

\bibitem{JKK}
C.K.R.T. Jones, T.J. Kaper, and N.~Kopell.
\newblock Tracking invariant manifolds up to exponentially small errors.
\newblock {\em SIAM J. Math. Anal.}, 27(2):558--577, 1996.

\bibitem{JK}
C.K.R.T. Jones and N.~Kopell.
\newblock Tracking invariant manifolds with differential forms in singularly
  perturbed systems.
\newblock {\em J. Differential Equations}, 108(1):64--88, 1994.

\bibitem{kv}
M.~Kashiwagi.
\newblock kv - {C++} {N}umerical {V}erification {L}ibraries.
\newblock {\tt http://verifiedby.me/kv/}.

\bibitem{KS}
M.~Krupa and P.~Szmolyan.
\newblock Extending geometric singular perturbation theory to nonhyperbolic
  points---fold and canard points in two dimensions.
\newblock {\em SIAM journal on mathematical analysis}, 33(2):286--314, 2001.

\bibitem{L}
W.~Liu.
\newblock Exchange lemmas for singular perturbation problems with certain
  turning points.
\newblock {\em J. Differential Equations}, 167(1):134--180, 2000.

\bibitem{Mat_RM}
K.~Matsue.
\newblock {\tt http://researchmap.jp/7000003451}.

\bibitem{Mat2}
K.~Matsue.
\newblock Rigorous numerics for fast-slow systems with one-dimensional slow
  variable: topological shadowing approach.
\newblock {\em Topological Methods in Nonlinear Analysis}, 50(2):357--468,
  2017.

\bibitem{Mc}
C.K. Mc{C}ord.
\newblock Mappings and homological properties in the {C}onley index theory.
\newblock {\em Ergodic Theory and Dynamical Systems}, 8(8*):175--198, 1988.

\bibitem{Mis}
K.~Mischaikow.
\newblock Conley index theory.
\newblock In {\em Dynamical systems ({M}ontecatini {T}erme, 1994)}, volume 1609
  of {\em Lecture Notes in Math.}, pages 119--207. Springer, Berlin, 1995.

\bibitem{Smo}
J.~Smoller.
\newblock {\em Shock waves and reaction-diffusion equations}, volume 258 of
  {\em Grundlehren der Mathematischen Wissenschaften [Fundamental Principles of
  Mathematical Sciences]}.
\newblock Springer-Verlag, New York, second edition, 1994.

\bibitem{S}
P.~Szmolyan.
\newblock Transversal heteroclinic and homoclinic orbits in singular
  perturbation problems.
\newblock {\em J. Differential Equations}, 92(2):252--281, 1991.

\bibitem{SW2001}
P.~Szmolyan and M.~Wechselberger.
\newblock Canards in $\mathbb{R}^3$.
\newblock {\em Journal of Differential Equations}, 177(2):419--453, 2001.

\bibitem{TKJ}
S.-K. Tin, N.~Kopell, and C.K.R.T. Jones.
\newblock Invariant manifolds and singularly perturbed boundary value problems.
\newblock {\em SIAM J. Numer. Anal.}, 31(6):1558--1576, 1994.

\bibitem{Tbook}
W.~Tucker.
\newblock {\em Validated numerics: a short introduction to rigorous
  computations}.
\newblock Princeton University Press, 2011.

\bibitem{vdBMM2013}
J.B. van~den Berg and J.D. Mireles~James.
\newblock Parameterization of slow-stable manifolds and their invariant vector
  bundles: Theory and numerical implementation.

\bibitem{Y1980}
T.~Yamamoto.
\newblock Error bounds for computed eigenvalues and eigenvectors.
\newblock {\em Numerische Mathematik}, 34(2):189--199, 1980.

\bibitem{ZCov}
P.~Zgliczy{\'n}ski.
\newblock Covering relations, cone conditions and the stable manifold theorem.
\newblock {\em J. Differential Equations}, 246(5):1774--1819, 2009.

\bibitem{ZG}
P.~Zgliczy{\'n}ski and M.~Gidea.
\newblock Covering relations for multidimensional dynamical systems.
\newblock {\em J. Differential Equations}, 202(1):32--58, 2004.

\bibitem{ZM}
P.~Zgliczy{\'n}ski and K.~Mischaikow.
\newblock Rigorous numerics for partial differential equations: the
  {K}uramoto-{S}ivashinsky equation.
\newblock {\em Found. Comput. Math.}, 1(3):255--288, 2001.

\end{thebibliography}

\end{document}